%% file: 0MAIN_FILE__MFG_DL_arXiv.tex
\documentclass[preprint,10pt]{elsarticle}
\usepackage{lipsum}
\makeatletter
\def\ps@pprintTitle{
 \let\@oddhead\@empty
 \let\@evenhead\@empty
 \def\@oddfoot{}
 \let\@evenfoot\@oddfoot}
 \renewcommand{\MaketitleBox}{
  \resetTitleCounters
  \def\baselinestretch{1}
  \begin{center}
    \def\baselinestretch{1}
    \Large \@title \par
    \vskip 18pt
    \normalsize\elsauthors \par
    \vskip 10pt
    \footnotesize \itshape \elsaddress \par
  \end{center}
  \vskip 12pt
}

\usepackage{geometry}
\usepackage{amsmath}
\usepackage{amssymb}
\usepackage{amsfonts}
\usepackage{graphicx}
\usepackage{cases}
\usepackage{algorithmic}
\biboptions{sort&compress}
\usepackage{mathtools}
\usepackage[mathscr]{eucal}
\usepackage{hyperref}
    \ifpdf
      \hypersetup{
        pdftitle={Neural Operators Can Play LQ-Mean Field Games in Hilbert Spaces},
        pdfauthor={Dena Firoozi, Anastasis Kratsios, and Xuwei Yang}
      }
    \fi
\usepackage{cleveref}
\usepackage[colorinlistoftodos]{todonotes}
\usepackage{enumitem}
\usepackage{comment}
\usepackage{todonotes}
\usepackage{dsfont}
\usepackage{amsthm}
\hypersetup{
  colorlinks=true,
  urlcolor=blue,
  linkcolor=blue,
  citecolor=blue,
  bookmarksdepth=paragraph
}

\newcommand\numberthis{\addtocounter{equation}{1}\tag{\theequation}}

\newtheorem{assumption}{Assumption}[section]
\newtheorem{example}{Example}
\numberwithin{equation}{section}
\newtheorem{remark}{Remark}
\newtheorem{theorem}{Theorem}[section]
\newtheorem{lemma}[theorem]{Lemma}
\newtheorem{proposition}[theorem]{Proposition}

\newtheorem{definition}{Definition}[section]
\definecolor{darkcyan}{rgb}{0.0, 0.55, 0.55}
\definecolor{MidnightBlue}{RGB}{25,25,112}
\definecolor{MidnightBlueComplementingGreen}{RGB}{25,112,25}
\definecolor{MidnightBlueComplementingPurple}{RGB}{112,25,112}
\definecolor{MidnightBlueComplementingRed}{RGB}{112,25,69}
\definecolor{WowColor}{rgb}{.75,0,.75}
\definecolor{MildlyAlarming}{rgb}{0.85,0.25,0.1}
\definecolor{SubtleColor}{rgb}{0,0,.50}
\definecolor{antiquefuchsia}{rgb}{0.57, 0.36, 0.51}
\definecolor{fashionfuchsia}{rgb}{0.96, 0.0, 0.63}
\definecolor{jade}{rgb}{0.0, 0.66, 0.42}
\definecolor{caribbeangreen}{rgb}{0.0, 0.8, 0.6}
\definecolor{aquamarine}{rgb}{0.5, 0.8, 0.85}
\definecolor{darkmidnightblue}{rgb}{0.0, 0.2, 0.4}
\definecolor{attentioncolor}{RGB}{152,90,81}
\definecolor{burgred}{RGB}{40,3,22}
\definecolor{AKGreen}{RGB}{17,123,92}
\definecolor{egyptianblue}{rgb}{0.06, 0.2, 0.65}
\definecolor{Turquoise}{RGB}{64,224,208}
\definecolor{darkjade}{RGB}{0,122,84}
\definecolor{Window1}{RGB}{92,150,31}
\definecolor{Window1dark}{RGB}{41,67,13}
\definecolor{Window2}{RGB}{255,168,28}
\definecolor{Window2dark}{RGB}{114,75,12}
\definecolor{Window3}{RGB}{255,96,33}
\definecolor{Window3dark}{RGB}{97,36,12}
\definecolor{InputColor}{RGB}{20,255,177}
\definecolor{InputColorlight}{RGB}{222,237,229}

\newcounter{termcounter}
\renewcommand{\thetermcounter}{\Roman{termcounter}}
\crefname{term}{term}{terms}
\creflabelformat{term}{#2\textup{(#1)}#3}

\makeatletter
\def\term{\@ifnextchar[\term@optarg\term@noarg}%]
\def\term@optarg[#1]#2{%
  \textup{#1}%
  \def\@currentlabel{#1}%
  \def\cref@currentlabel{[][2147483647][]#1}%
  \cref@label[term]{#2}}
\def\term@noarg#1{%
  \refstepcounter{termcounter}%
  \textup{(\thetermcounter)}%
  \cref@label[term]{#1}}
\makeatother

\newcommand{\mb}{\mathbb}

\newcommand{\mc}{\mathcal}

\newcommand{\mfT}{{\mathcal{T} }}

\newcommand{\mT}{\mathrm{T}}

\newcommand{\mfF}{{\mathfrak{F}}}
\DeclareMathOperator{\tr}{tr}

\NewDocumentCommand{\F}{o}{
  \IfValueT{#1}{\mathbb{F}_{#1}}
  \IfValueF{#1}{\mathbb{F}}
}

\NewDocumentCommand{\R}{o}{
  \IfValueT{#1}{\mathbb{R}^{#1}}
  \IfValueF{#1}{\mathbb{R}}
}
\NewDocumentCommand{\N}{o}{
  \IfValueT{#1}{\mathbb{N}^{#1}}
  \IfValueF{#1}{\mathbb{N}}
}
\newcommand{ \eqdef }{
  \ensuremath{\stackrel{\mbox{\upshape\tiny def.}}{=}}
}

\newcommand{\HS}{\HRules}
\newcommand{\HRules}{\mathcal{H}}
\newcommand{\Lt}{\mathcal{M}^2(\mfT,U)}

\usepackage{tikz}
\usepackage{tikz-cd}

\include{Comments}

\begin{document}

\begin{frontmatter}

\title{Simultaneously Solving Infinitely Many LQ Mean Field Games In Hilbert Spaces: The Power of Neural Operators}

\author[1]{Dena Firoozi}
\author[2]{Anastasis Kratsios}
\author[2]{Xuwei Yang\footnote{Corresponding Author.}}

\address[1]{Department of Statistical Sciences, University of Toronto (email:dena.firoozi@utoronto.ca).}

\address[2]{Department of Mathematics and Statistics, McMaster University, and The Vector Institute (email:kratsioa@mcmaster.ca, henryyangxuwei@gmail.com).}

\begin{abstract}
Traditional mean-field game (MFG) solvers operate on an instance-by-instance basis, which becomes infeasible when many related problems must be solved (e.g., for seeking a robust description of the solution under perturbations of the dynamics or utilities, or in settings involving continuum-parameterized agents.). We overcome this by training neural operators (NOs) to learn the \textit{rules-to-equilibrium} map from the problem data (``rules'': dynamics and cost functionals) of LQ MFGs defined on separable Hilbert spaces to the corresponding equilibrium strategy. Our main result is a statistical guarantee: an NO trained on a small number of randomly sampled rules reliably solves unseen LQ MFG variants, even in infinite-dimensional settings. The number of NO parameters needed remains controlled under appropriate rule sampling during training.

Our guarantee follows from three results: (i) local-Lipschitz estimates for the highly nonlinear rules-to-equilibrium map; (ii) a universal approximation theorem using NOs with a prespecified Lipschitz regularity (unlike traditional NO results where the NO’s Lipschitz constant can diverge as the approximation error vanishes); and (iii) new sample-complexity bounds for $L$-Lipschitz learners in infinite dimensions, directly applicable as the Lipschitz constants of our approximating NOs are controlled in (ii).
\end{abstract}


\end{frontmatter}

\section{Introduction} 
\label{sec:intro}

Many phenomena across the social sciences, such as opinion formation in social networks~\cite{bauso2016opinion} and price formation in microeconomics and finance~\cite{Firoozi2022MAFI,gomes_mean-field_2018,fujii_mean_2021}, as well as in engineering, epidemiology~\cite{laguzet2015individual}, traffic flow in urban planning~\cite{bauso2016density}, ensemble Kalman filtering~\cite{del2008mean,carrillo2024mean,ertel2025mean}, and deep learning~\cite{mei2019mean}, are effectively modelled as systems consisting of a large number of interacting, indistinguishable agents, each typically observing only its own state. These systems often involve continuous strategic interaction, where agents compete to optimize their individual objectives. The resulting competitive (game theoretic) \textit{equilibria}, where each agent acts optimally given the aggregate behaviour of the population, is typically a highly nonlinear function of the problem data, related to the dynamics of the environment and the individual objectives of the agents. Such models fall under the framework of mean field games (MFGs), introduced in~\cite{huang2006large,lasry2007mean}. 
One might expect that certain classes of these games are easily solvable numerically; for example, games where the agent's behaviour and the objective function are first-order and second-order approximations of the ground truth, respectively--specifically, the linear-quadratic (LQ) MFGs. However, solutions to these games are only partially available in closed form under certain conditions. These solutions, in turn, rely on resolving multi-dimensional, high-dimensional or even infinite-dimensional coupled forward-backward ODEs, depending on the size of the state space dimension. Furthermore, these ODEs themselves depend on the rules of the MFG in a highly nonlinear manner (see e.g.~\cite{Huang-Nguyen-2012,LF24,FGG24}).
% }

In situations involving model uncertainty, several variations of the rules of an LQ MFG problem typically need to be resolved, each variation of which quantifies a plausible alternative to the rules of the MFG and thus may admit its own plausible equilibrium state.  Once each plausible variation on the rules of the game is resolved, then an explicit robust description of the MFG system is possible, typically taking the form of a worst-case or (weighted) average case. 
% {\color{blue}
Moreover, to more closely emulate reality, LQ MFG theory has been developed to incorporate several distinct subpopulations of agents. Within each subpopulation, agents share the same rules, which differ from those of other subpopulations. Similarly, the theory has also been extended to accommodate continuum-parametrized agents.  In these respective scenarios, several or infinitely many sets of coupled forward-backward ODEs or deterministic evolution equations must be solved to characterize the equilibrium strategies (see e.g. \cite{huang2007large, huang2010large,Huang-Nguyen-2012}). With these motivation, this paper focuses on tools capable of resolving such infinite systems of LQ MFG equations.

The critical observation is that classical solvers are not viable for resolving \textit{infinitely many} MFG equations. Examples include finite-difference schemes~\cite{briceno2019implementation}, semi-Lagrangian schemes~\cite{angiuli2019cemracs}, and Newton-type iterative methods~\cite{camilli2023convergence}. More recently, deep learning-based approaches have been proposed~\cite{fouque2020deep,carmona2021convergence,germain2022numerical,cao2024connecting,soner2025learning}. However, all of these solvers are designed to address a \textit{single} MFG instance at a time; thus, they must be re-run from scratch for each variation of the problem under consideration.

To illustrate the limitation, suppose we are solving an LQ MFG on the state space $\mathbb{R}^d$, and we wish to consider all additive perturbations to the initial condition $x \in [0,1]^d$. This would correspond to an \textit{uncountably infinite} family of games; clearly, an infeasible computational task. Even if the domain is discretized using a step size $1/S$ for some $S \in \mathbb{N}_+$, restricting attention to initial states on the grid $\{(s_i/S)_{i=1}^d\}_{s_1,\dots,s_d=0}^S \subset [0,1]^d$, we are still left with $(S+1)^d$ distinct MFGs to solve. This number grows exponentially with the dimension $d$, making the task intractable as $S$ tends to infinity.  
This problem is significantly exacerbated in the setting of infinite-dimensional LQ MFGs studied in \cite{LF24,FGG24}, where the number of solver runs can easily become exponential due to lower bounds on the metric entropy of compact subsets of the space of problem variations in infinite dimensions; see~\cite[Chapter 15]{lorentz1996constructive}. Such infinite-dimensional games naturally arise when considering Markovian lifts of Volterra processes on finite-dimensional state spaces with completely monotone Volterra kernels; see, e.g.,~\cite{cuchiero2019markovian,cuchiero2020generalized,hamaguchi2024markovian}. In short, ``game-by-game'' solvers are poorly suited to settings involving (infinitely) many variations of the rules defining a given LQ MFG.

Our main contribution proposes a way out of this predicament, not by solving a large set of MFGs but rather by directly learning the solution operator defining the entire family.  That is, we design a single solver which yields an (approximate) solution to every relevant LQ MFG \textit{simultaneously}.  At first glance, this may seem unlikely; however, such approaches have recently been successfully deployed in computational physics~\cite{wang2021learning,de2022generic,de2022deep,goswami2023physics,benitez2024out,li2024physics,azizzadenesheli2024neural} and more recently for Stackelberg games~\cite{alvarez2024neural} using \textit{infinite-dimensional deep learning} models known as \textit{Neural Operators} (NOs).   The commonality threading each of these approaches together is that they attempt to learn a \textit{solution operator} parameterizing every problem in the (possibly infinite) family of problems.  In our case, we wish to \textit{learn} (not only approximate) the \textit{rules-to-equilibrium} operator, by which we mean the function mapping the dynamics and objective operators of each agent to their Nash equilibrium strategy, from finite data.  Here, the \textit{data} consists of random pairs comprising rules for the LQ MFG and the corresponding equilibrium strategy. These rules \textit{rules} details how the drift coefficient, volatility, and objective of each agent are affected by the population's aggregate behaviour, i.e. the mean field, and the individual agent's state and control action.  

While recent numerical studies provide \textit{experimental} evidence that NOs offer a promising computation tool in solving MFGs~\cite{HUANG2025114057,chen2024physics}, there still remains no theoretical guarantee of their reliability. This work fills that gap by developing a rigorous theory; a fortiori in the infinite-dimensional setting.

\paragraph{Contributions}
We therefore take a major first step in this direction by showing that a broad range of families of LQ MFGs are approximately solvable from finite training samples when using any training algorithm; that is, we demonstrate the probably approximately correct (PAC) learnability of the rules-to-equilibrium operator using NOs (Theorems~\ref{thrm:main} and~\ref{thrm:main_pt2__quantitative_existenceERM}).  Our results additionally quantifies the required number of samples to learn the rules-to-equilibrium map, independently of the training algorithm used to optimize the deep learning model, as well as the distribution by which pairs of rules and equilibria are sampled. 
% {\color{blue} 
It is worth noting that although our focus in this paper is on infinite-dimensional LQ MFGs, the application of the approximating NOs under study extends beyond this scope. For instance, for robustness purposes or continuum-parameterized cases, they can be effectively used to solve infinitely many finite-dimensional LQ MFGs (\cite{huang2007large}), finite-dimensional or infinite-dimensional LQ single-agent control problems (\cite{ichikawa1979dynamic,tessitore1992some}), and LQ optimal control problems over large-size networks (\cite{dunyak2024quadratic}).
% }  

\paragraph{Technical Contributions}
The derivation of our main result is an interdisciplinary combination of two new results.  The first (quantitatively) establishes the \textit{stability} of the perturbations of infinite-dimensional LQ MFGs; namely, we show the Lipschitz stability of the rules-to-equilibrium operator (Theorem~\ref{thrm:Stability_R2E}). Additionally, our analysis relies on approximating the sample complexity guarantee (Theorem~\ref{thrm:MainLearning}) when learning $ L$-Lipschitz (non-linear) operators using $ L$-Lipschitz neural operators, the analysis of which relies on techniques from optimal transport.  These results are the first of their kind and extend their very recent analogues in finite dimensions~\cite{hong2024bridging}, since classical NO learning and approximation guarantees do not control nor leverage the Lipschitz regularity of the NO model itself.  

\paragraph{Organization of Paper}
The rest of the paper is organized as follows. Section~\ref{s:Preliminaries} reviews the necessary background. Section~\ref{s:Main_Results} presents our main results along with the key technical components and proof strategy. Section~\ref{s:Proofs} and Section~\ref{s:Proofs_PACLearning} contain all proofs. Additional background material is provided in the appendices~\ref{App1} and \ref{s:HilbertStructures}.

\section{Preliminaries}
\label{s:Preliminaries}
The interdisciplinary nature of our paper, which combines mean field games with the approximation theory of nonlinear operators (from a deep learning perspective) and elements of statistical learning theory, necessitates a brief overview of each of these concepts to ensure a self-contained reading of our results; which we now do.

\subsection{Infinite-Dimensional LQ MFGs}
\label{HLMFG}
% \subsubsection{$N$-Player Game in Hilbert Spaces}
Fix $T>0$ and let $\mfT\eqdef [0,T]$.
We consider a class of LQ MFGs defined on Hilbert spaces over $\mathcal{T}$. Such an MFG may be viewed as the limiting model for an $N$-player game as $N \rightarrow \infty$, where the state, control, and noise processes associated with each agent take values in infinite-dimensional spaces. These models are naturally suited when an agent's behavior is impacted by delayed state or control processes, and a Markovian lifting is employed. 

Let $H$, $U$ and $V$ denote separable Hilbert spaces. Moreover, let $\mathcal{L}(V,H)$ denote the space of all bounded linear operators from $V$ to $H$, which is a Banach space equipped with the norm $\left \| \mT \right \|_{\mathcal{L}(V,H)}=\sup_{\left \| x \right \|_{V}=1}\left \| \mT x \right \|_{H}$. The dynamics of a representative agent in a Hilbert space-valued MFG model is governed by a stochastic evolution equation given by 
\begin{align} 
\label{dx_LQ_MFG}
\begin{aligned} 
 & dx(t) = \left(A x(t) + B u(t) + F_1 \bar{x}(t)\right) dt + \left( D x(t) + E u(t) + F_2 \bar{x}(t) + \sigma \right) 
 d W(t), \\ 
 &  x(0) = \xi ,  
 \end{aligned} 
\end{align} 
where $\xi \in L^2(\Omega;H)$. In the above equation, $x(t) \in H$ denotes the state and $u(t) \in U$ the control action at time $t$ of the agent. The control process $u=\{u(t): t\in \mfT\}$ is assumed to be in $\mathcal{M}^2(\mfT;U)$, which denotes the Hilbert 
space of all ${U}$-valued progressively measurable processes $u$ satisfying 
    \(\left\|u  \right\|_{\mathcal{M}^2(\mfT;U)}^2 \eqdef
    % \left (
        \mathbb{E}\int_{0}^{T}\left\|u(t) \right\|^{2}_{U}dt  
    % \right )^\frac{1}{2}
    < \infty\). The $Q$-Wiener process $W$ is defined on a filtered probability space $\left (\Omega, \mfF,\mathcal{F}, \mathbb{P} \right )$ and takes values in $V$, where \(\mathcal{F}=\left\{ \mathcal{F}_t: t \in \mfT\right\}\) is the filtration under which the process \(W\) is a \(Q\)-Wiener process. The process $\bar{x}\in C(\mfT ;H)$ represents the mean field, where $C(\mfT ;H)$ denotes the set of all continuous mappings $h:\mfT \rightarrow H$. The mean field represents the aggregate state of the population in the limit as the number of agents $N \rightarrow \infty$. 
 The unbounded linear operator $A$, with domain $\mathcal{D}(A)$, is the infinitesimal generator of a $C_{0}$-semigroup $S(t) \in \mathcal{L}(H),\, t \in \mfT$. Moreover, there exists a constant $M_T$ such that 
\begin{equation}
\left\|S(t) \right\|_{\mathcal{L}(H)}   \leq  M_T,\quad \forall t \in \mfT,\label{S_bound}
\end{equation}
 where $M_T \eqdef M_A e^{\alpha T}$, with $M_A \geq 1$ and $\alpha \geq 0$ \cite{goldstein2017semigroups}. The choices of $M_A$ and $\alpha$ are independent of $T$. Furthermore, $B \in \mathcal{L}(U,H)$, $D \in \mathcal{L}(H,\mathcal{L}(V,H))$, $E \in \mathcal{L}(U,\mathcal{L}(V,H))$, $F_1 \in \mathcal{L}(H)$, $F_2 \in \mathcal{L}(H;\mathcal{L}(V;H))$ and $\sigma \in \mathcal{L}(V;H)$. 
We focus on the mild solution of \eqref{dx_LQ_MFG} which involves the $C_0$-semigroup $S$ (cf.~\cite{LF24}). 
\begin{assumption}\label{assm:xiL2} 
The initial condition
$\xi \in L^2(\Omega;H)$ is independent of the $Q$-Wiener process $W$ and $\mathcal{F}_0$-measurable. Moreover,  $\mb{E}[\xi]=\bar{\xi}$.
\end{assumption} 

\begin{assumption}(Filtration \& Admissible Control)
\label{assm:FiltrationControl}
The filtration available to the representative agent is $\mathcal{F}$. Subsequently, the set of admissible control laws for the agent, denoted by $\mc{U}$, is defined as the collection of  $\mc{F}$-adapted control laws $u$ that belong to $ \mathcal{M}^2(\mfT;U)$.   
\end{assumption}

A representative agent seeks a strategy $u$ to minimize its cost functional
\begin{align} 
J(u) = & \mathbb{E} \int_0^T 
\Big[ \big\| M^{\frac{1}{2}} \big( x(t) - \widehat{F}_1 \bar{x}(t) \big) \big\|^2_H + \|u(t)\|^2_U \Big] dt  
%\notag \\ 
% & \hspace{2.5cm} 
 + \mathbb{E} \big\| G^{\frac{1}{2}} \big( x(T) - \widehat{F}_2 \bar{x}(T) \big) \big\|^2_H,
\label{Jinfty}
\end{align} 
where $M$ and $G$  are positive operators on $H$, and $M, G, \widehat{F}_1, \widehat{F}_2 \in \mathcal{L}(H)$. 

An MFG equilibrium strategy for the representative agent is characterized by solving the following two problems.
\begin{itemize}
    \item[(i)] A stochastic control problem obtained by fixing the mean field process at $g \in C(\mfT ;H)$. Solving this problem yields the optimal response strategy of the representative agent, denoted by $u^\circ$, to the fixed mean field. The resulting state from this strategy is denoted by $x^\circ$.
    \item[(ii)] A mean field consistency equation requiring that the resulting mean field, $\mathbb{E}[x^\circ]$, match the assumed mean field $g$, i.e., $\mathbb{E}[x^\circ(t)] = g(t)$ for all $t \in \mathcal{T}$. The fixed point to this equation characterizes the mean field $\bar{x}$.
\end{itemize}

By~\cite[Definition 4.1]{LF24}, for any $\mathcal{R} \in \mathcal{L}(H)$ the mappings
$\Delta_1: \mathcal{L}(H) \to \mathcal{L}(H;U)$, $\Delta_2: \mathcal{L}(H) \to \mathcal{L}(H)$, and 
$\Delta_3: \mathcal{L}(H)\to \mathcal{L}(U)$ 
are defined as follows using the Riesz representation theorem
\begin{align} 
\begin{aligned} 
& \tr( (Eu)^\star \mathcal{R}(Dx) Q ) = \langle \Delta_1(\mathcal{R}) x, u \rangle , \quad 
\forall \, x\in H, \forall \, u \in U, 
 \\ 
& \tr( (Dx)^\star \mathcal{R}(Dy) Q ) = \langle \Delta_2(\mathcal{R}) x, y \rangle , \quad 
\forall \, x, y\in H ,  \\ 
& \tr((Eu)^\star \mathcal{R}(Ev)Q ) = \langle \Delta_3(\mathcal{R}) u, v \rangle , \quad 
\forall u, v\in U, 
\end{aligned}  
\label{def:Delta123}
\end{align} 
where ${}^\ast$ is used to denote associated adjoint operators. Furthermore, the mappings
$\Gamma_1: \mathcal{L}(H) \to H$ and $\Gamma_2: \mathcal{L}(H) \to U$ are similarly defined as 
\begin{align} 
\begin{aligned}  
& \tr(\mathcal{R}(Dx) Q) = \langle \Gamma_1(\mathcal{R}), x \rangle, \quad 
 \forall x\in H, 
  \\ 
& \tr(\mathcal{R}(Eu) Q) = \langle \Gamma_2(\mathcal{R}), u \rangle, \quad 
 \forall u \in U .  
 \end{aligned}  
 \label{def:Gamma12} 
\end{align} 
As demonstrated in~\cite[Theorem 4.9]{LF24}, there exists a unique equilibrium strategy to the LQ MFG~\eqref{dx_LQ_MFG}-\eqref{Jinfty} in a semi-closed form. Specifically, the expression of this strategy depends on the solution to a set of forward-backward deterministic evolution equations, which includes an infinite-dimensional differential Riccati equation. We recapitulate this result, which we routinely use in our stability analysis, in the supplementary material of our paper (Section~\ref{s:MFGRecalEquilibrium}).  

\subsubsection{Space of Variations on Rules of LQ MFGs}
\label{s:Notation__Rules}
{
We denote the set of Hilbert-Schmidt operators between two Hilbert spaces $H_1$ and $H_2$ by $\mathcal{HS}(H_1,H_2)$ with inner-product $\langle ,\rangle_{\mathcal{HS}(H_1,H_2)}$.

We begin by fixing a reference set of LQ MFG \textit{rules} $(A^{\dagger},B^{\dagger},F_2^{\dagger})$ satisfying the assumptions of Section~\ref{HLMFG} relative to which we define our space $\mathcal{H}$ of triples $(A,B,F_2)$ of a (possibly) unbounded linear operator $A:H\to H$, and bounded linear operators $B:U\to H$ and $F_2:U\to \mathcal{HS}(V,H)$ for which the difference operators $\Delta A\eqdef A-A^{\dagger}$, $\Delta B\eqdef B-B^{\dagger}$, and $\Delta F_2\eqdef F_2-F_2^{\dagger}$ are Hilbert-Schmidt.  The elements of $\mathcal{H}$ constitute variations on \textit{rules} of the LQ MFG~\eqref{dx_LQ_MFG}-\eqref{Jinfty}.  The space $\mathcal{H}$ can be made into a Hilbert space which is isomorphic (as a Hilbert space) to the direct sum of 
$             
\mathcal{HS}(H)
        \oplus 
            \mathcal{HS}(U,H)
       \oplus
            \mathcal{HS}\big(U,
                \mathcal{HS}(V,H)
            \big)$ equipped with the inner-product
            
\begin{equation}
\label{eq:3HS_2_DirSum}
% \label{eq:innerproduct_H}
        \langle(A,B,F_2),(\tilde{A},\tilde{B},\tilde{F}_2) \rangle
        % _{\mathcal{H}}
    \le 
        \langle 
        \Delta A, \Delta\tilde{A}\rangle_{\mathcal{HS}(H)}
        +
        \langle \Delta B,\Delta \tilde{B}\rangle_{\mathcal{HS}(U,H)}
        +
        \langle \Delta F_2,\Delta \tilde{F}_2\rangle_{\mathcal{HS}(U,
                \mathcal{HS}(V,H)
            )},
\end{equation}
for all triples $(A,B,F_2),(\tilde{A},\tilde{B},\tilde{F}_2) \in \mathcal{H}$. 
An explicit orthogonal basis for $\mathcal{H}$ is constructed in Appendix~\ref{eq:ONB_HS}, up to identification by the isomorphism $(A,B,F_2)\mapsto (\Delta A,\Delta B,\Delta F_2)$.  We henceforth denote this orthogonal basis by $(e_i)_{i\in \mathbb{N}}$.  
}

We note that in this paper, we present variations on only three operators $(A,B,F_2)$ related to the dynamics \eqref{dx_LQ_MFG}. However, the analysis may be extended to any operator appearing in the model \eqref{dx_LQ_MFG}-\eqref{Jinfty} in a similar manner as the perturbation analysis for these three operators is the most challenging.
\subsubsection{Rules-to-Equilibrium Operator}
\label{s:Notation__Controls}
As shown in~\cite{LF24}, the equilibrium of any such LQ MFG on a Hilbert space must belong to the space $\mathcal{M}^2(\mfT,U)$
% L^2(\mfT,\Lt)
consisting of all progressively-measurable strategies $u$ for which the following norm is finite
\begin{equation}
\label{eq:Sup_to_L2}
        \|u\|_{\mathcal{M}^2(\mfT,U)}^2
    \eqdef 
        {\mathbb{E}\biggl[
            \int_0^T\,
                \|u(t)\|_U^2\,
            dt
        \biggr]}
.
\end{equation}
This is a separable Hilbert space, for which elementary tensor product considerations show that it admits an orthogonal basis of the form
$
(\psi_n\,y_j)_{n\in \mathbb{N},\,j\in J}
$, where $\{\psi_n\}_{n\in \mathbb{N}}$ enumerates an {orthonormal basis of $\mathcal{M}^2(\mfT,\mathbb{R})$; e.g.\ the Wiener chaos and Haar wavelet-based construction given in~\cite[Lemma B.6]{alvarez2024neural},}
and $(y_j)_{j\in J}$ enumerates a fixed choice of an orthonormal basis of $U$ where $J$ is a non-empty subset of $\mathbb{N}$.  
We enumerate and abbreviate these basic vectors as $(\eta_{n})_{n\in \mathbb{N}}$ (see Appendix~\ref{s:HilbertStructures_ss:BLTensor} for further details).  
Our objective is to learn the \textit{rules-to-equilibrium operator} $\mathfrak{R}:\HRules\to 
\mathcal{M}^2(\mfT,U)
% L^2(\mathcal{T},\Lt)
$ from finite random noiseless training data.  The \textit{rules-to-equilibrium} operator thus maps a triple $(A,B,F_2)\in \HRules$ to the equilibrium of the LQ MFG it defines; i.e.
\begin{equation}
\label{eq:R2EOp}
        \mathfrak{R}(A,B,F_2) 
    \eqdef 
        \mbox{\textit{equilibrium strategy} of the LQ MFG}~\eqref{dx_LQ_MFG}\mbox{-}\eqref{Jinfty}
.
\end{equation}
Under mild conditions, we prove (Theorem~\ref{thrm:Stability_R2E}) that this map is well-posed; that is, it not only exists, as was shown in~\cite{LF24}, but it is also \textit{locally Lipschitz continuous} and we estimate its local Lipschitz constant.

\subsection{Statistics in Infinite Dimensions}
\label{s:Notation__ss:HilbertStructure}

The \textit{sub-Gaussian} (resp.\ \textit{Exponential}) (Orlicz) norm of a real-valued random variable $\tilde{Z}$ is defined by
$
        \|\tilde{Z}\|_{\psi_i}
    \eqdef 
        \inf
            \big\{c>0:\,
                    \mathbb{E}\big[
                        e^{(\tilde{Z}/c)^i}
                    \big]
                \le 
                    2
            \big\}
$, where $i=2$ (resp.\ $i=1$). Our first main technical result (Theorem~\ref{thrm:Stability_R2E}) shows that $\mathfrak{R}$ is \textit{well-posed}; namely, that it exists and that it is $L$-Lipschitz continuous on any suitable compact subsets $\mathcal{K}$ of $\mathcal{H}$; for some $L\ge 0$ depending on $\mathcal{K}$.
Now, for any fixed non-empty compact set $\mathcal{K}\subseteq \HRules$ we consider an $L$-Lipschitz \textit{extension}%
\footnote{An $L$-Lipschitz exists by the Benyamini-Lindenstrauss theorem; see e.g.~\cite[Theorem 1.12]{BenyaminiLindenstrauss_2000_NonlinearFunctionalAnalysis}.} $\mathfrak{R}^{\mathcal{K}}$ 
of $\mathfrak{R}|_{\mathcal{K}}$ to all of $\HRules$, where $\mathfrak{R}|_{\mathcal{K}}$ denotes the restriction of the operator
$\mathfrak{R}$ to the subset $\mathcal{K}$.  That is, $\mathfrak{R}^{\mathcal{K}}$ is the rules-to-equilibrium operator on $\mathcal{K}$, and it maintains its Lipschitz continuity globally, even outside $\mathcal{K}$.
Consequently, we may meaningfully operate under the following assumptions on the \textit{training data} consisting of \textit{random} draws of a set of rules, e.g.\ $X_n=(A_n,B_n,F_{2,n})$, in $\HRules$ \textit{paired} with its \textit{solution} $Y_n=\mathfrak{R}^{\mathcal{K}}(X_n)$ in $\mathcal{M}^2(\mathcal{T}, U)$ to induce the LQ MFG~\eqref{dx_LQ_MFG}-\eqref{Jinfty}.
\begin{assumption}[Data Generating Distribution's Structure]
\label{ass:Statistical}
Let $\mathbb{P}_X$ be a centred Borel probability measure on $\HRules$, {fix a non-empty compact set $\mathcal{K}\subseteq \HRules$,} and let $\mathfrak{R}^{\mathcal{K}}$ be the $L$-Lipschitz extension of $\mathfrak{R}|_{\mathcal{K}}$ to all of $\HRules$; for some $L\ge 0$ (depending on $\mathcal{K}$).
\noindent Fix a sample size $N\in \mathbb{N}_+$ and let $(X_1,Y_1),\dots,(X_N,Y_N)$ be i.i.d.\ random variables {taking values} in $\HRules\times \mathcal{M}^2(\mathcal{T},U)$ 
{defined on a probability space $(\Omega,\mathcal{A},\mathbb{P})$}
with
\begin{enumerate}[label=(\roman*)]
    \item \textbf{Sampling Distribution:} $X_1\sim\dots\sim X_N\sim \mathbb{P}_X$,
    \item \textbf{Realizable Supervised Setting:} $Y_n=\mathfrak{R}^{\mathcal{K}}(X_n)
    $ for each $n=1,\dots,N$.
    \item \textbf{Exponential Karhunen–Lo\'{e}ve decomposition:} If $X\sim \mathbb{P}_X$, then
    \begin{equation}
    \label{eq:KLDecomposition}
        \underbrace{
        X = \sum_{i=1}^{\infty}\,\sigma_i Z_i\,e_i 
        ,
        }_{\text{Karhunen–Lo\'{e}ve decomposition}}
        \,\,
        \underbrace{
            0\le \sigma_i\lesssim e^{-r i}
        \mbox{ and }
            \sup_{j\in \mathbb{N}_+}\,
                \|Z_j\|_{\psi_2}
            <
                \infty,
        }_{\text{Rapid decay}}
    \end{equation}
    for all $i\in \mathbb{N}_+$; and decreasing $\sigma_1\ge \sigma_2\ge \dots \ge 0$, where $(Z_i)_{i\in\mathbb{N}_+}$ are independent real-valued random variables, {and $(e_i)_{i\in \mathbb{N}_+}$ is the orthonormal basis for $\mathcal{H}$ fixed in Section~\ref{s:Notation__Rules}.}
\end{enumerate}
\end{assumption}

\begin{example}[Centred Gaussian with Diagonal Covariance]
\label{ex:Gaussian}
In Assumption~\ref{ass:Statistical} (iii), if for each $i\in \mathbb{N}_+$, each $Z_i$ is a scalar standard normal random variable, then $\mathbb{P}_X$ is a centred Gaussian measure on $\HRules$ with diagonal covariance operator $Q$, diagonalized by $\sigma_1\ge \sigma_2\ge \dots$.  
In this case, one may also choose $(e_i)_{i\in \mathbb{N}}$ in $\mathcal{H}$ to be the an eigenbasis of the covariance operator of $\mathbb{P}_X$.
\end{example}

\subsection{Infinite-Dimensional Deep Learning: Neural Operator}
\label{s:Notation__ss:OperatorLearning}
The first layer of our neural operators will compress infinite-dimensional inputs into finite-dimensional features, which can be processed by classical deep learning models.  As in~\cite{galimberti2022designing}, we rely on linear projection operators, defined for each $i\in I$, where $I$ is a non-empty subset of $\mathbb{N}_{+}$, by $P_i^{\mathcal{H}}:\HRules \to \mathbb{R}^i$ by sending
\begin{equation}
\label{eq:ProjOp}
\begin{aligned}
            P_i^{\mathcal{H}}(x) 
        & \eqdef
            \big(
                \langle x,e_k\rangle
            \big)_{k=1}^i
\end{aligned}
\end{equation}
for each $x\in \mathcal{K}$ where $\langle \cdot,\cdot\rangle$ denotes the inner-product on $\HS$ 
(cf.~\ref{eq:HS_Beyond})
and $(e_i)_{i\in I}$ denotes 
{the orthonormal basis of $\mathcal{H}$ fixed in Section~\ref{s:Notation__Rules}.}
% a choice of orthonormal basis on $\HS$.  
% 
We will also make use of the embedding operators in $\Lt$ defined using a fixed orthonormal basis thereof $(\eta_j)_{j\in J}$ as follows: for each $j\in J$ let $E_j
% ^{\mathcal{H}}
:\mathbb{R}^{j}  \to \Lt$ map
\begin{equation}
\label{eq:EmbOp}
\begin{aligned}
        E_j(\beta) \eqdef \sum_{1\le k\le j}\, \beta_k \,\eta_k
.
\end{aligned}
\end{equation}
We consider the following class of residually guided neural operators (RNOs), which includes PCA-net~\cite{lanthaler2023operator} and several other neural operator architectures.
\begin{definition}[Residually-Guided Neural Operator (RNO)]
\label{defn:RNONO}
Fix a triple $x^{\dagger}=(A,B,F_2)$ as in Section~\ref{s:Notation__Rules}, and
ordered bases $(e_i)_{i=0}^{\infty}$ and $(\eta_j)_{j=0}^{\infty}$ for $\mathcal{X}$ and $\Lt$ respectively.
%%%
Fix regularity parameters $0<\alpha \le 1$ and $0\le L$,
 and a connectivity parameter $C\in \mathbb{N}_+$.
The class of $(\alpha,L)$-regular $C$-connected RNOs $\mathcal{RNO}^{\alpha,L}_C(e_{\cdot},{\eta}_{\cdot})$
consists of all $(\alpha,L)$-H\"{o}lder operators $f:
\HS\to \Lt$
with iterative representation mapping any $x\in \HRules$ to 
\begin{equation}
\label{eq:defn_RNONO}
\begin{aligned}
    f(x) & \eqdef E_{N_2}(A_{\Delta}x_\Delta+b_{\Delta}) + y^{\dagger}
    \allowdisplaybreaks\\
    x_{l+1}& \eqdef \operatorname{ReLU}\big(A_lx_l+b_l\big),\,
    \mbox{ for }l=0,\dots,\Delta-1
    \allowdisplaybreaks\\
    x_0 & \eqdef P_{N_1}^\mathcal{H}(x+x^{\dagger})
\end{aligned}
\end{equation}
for some rank parameters $N_1,N_2\in \mathbb{N}_+$, a depth parameter $\Delta\in \mathbb{N}_+$, 
a width parameter $W\in \mathbb{N}_+$, and hidden weights $A_l\in \mathbb{R}^{d_{l+1}\times d_l}$ and biases 
$b_l\in \mathbb{R}^{d_{l+1}}$ for $l=0,\dots,\Delta$; with $\max_{l=0,\dots,\Delta}\, d_l\le W$ satisfying 
$%$\[
        \sum_{l=0}^{\Delta}\,
            \|A_l\|_0 + \|b_l\|_0
    \le 
        C
$, where $\|.\|_0$ denotes the count of non-zero entries, $E_{N_2}$ and $P_{N_1}^\mathcal{H}$ are defined as in~\eqref{eq:EmbOp} and ~\eqref{eq:ProjOp}, respectively,
and $x^{\dagger}\in \HRules$ and  $y^{\dagger}\in \Lt$ denote reference points.
\end{definition}

\begin{remark}
\label{rem:Noation}
Since the space $\mathcal{H}$ (cf.~Section~\ref{s:Notation__Rules}) is defined \textit{relative} to a reference set of rules $x^{\dagger}$, which by~\cite[Theorem 4.1]{LF24} corresponds to the unique equilibrium $y^{\dagger}=\mathcal{R}(x^{\dagger})$, the choice of orthonormal bases $e_{\cdot}$ and $\eta_{\cdot}$ in the RNO (cf.~Definition~\ref{defn:RNONO}) implicitly fixes the reference pair $(x^{\dagger},y^{\dagger})$.
\end{remark}
In short, our NO architecture, described above, performs perturbations around the reference points $(x^\dagger, y^\dagger)$.  
We close this preliminary section by presenting a simple, fully explicit class of RNO~\eqref{eq:defn_RNONO} in which the reference rules $x^{\dagger}$ of the MFG \eqref{dx_LQ_MFG}-\eqref{Jinfty} are set such that the corresponding equilibrium strategy $y^{\dagger}$ is explicitly computable.

\begin{example} Let $a$, $s$, $m$, $b$ and $g$ be positive and finite real numbers, and $\mathds{I}$ denote the identity operator defined on the Hilbert space $H$. Define $A=-a \mathds{I}$, $B=\mathds{I}$, $F_1=a\mathds{I}$, $\sigma=s\mathds{I}$, $M=m\mathds{I}$, $\hat{F}_1=b\,\mathds{I}$, $G=g\mathds{I}$, $\hat{F}_2=b\,\mathds{I}$, and the operators $D,E$ and $F_2$ to be zero operators for the MFG model \eqref{dx_LQ_MFG}-\eqref{Jinfty}. Then, the optimal strategy is given by $u^\circ(t)=\Pi(t)(x(t)-\bar{\xi})$, where $\Pi(t)=\pi(t)\mathds{I}$ with \begin{gather*}
\pi_t = \frac{c_1 r_1 e^{r_1 t} + c_2 r_2 e^{r_2 t}}{c_1 e^{r_1 t} + c_2 e^{r_2 t}}\label{ric1}\\  r_1, r_2 = a \pm \sqrt{a^2 - m}  \label{ric2}
\end{gather*}
for some real-valued constants $c_1$ and $c_2$, and $\bar{\xi}=\mathbb{E}[x(0)]$. Here $x^\dagger\eqdef(-a\mathds{I}, \mathds{I},  0)$ and $y^\dagger\eqdef u^\circ$ 
in~\eqref{eq:defn_RNONO}.
\end{example}

\section{Main Results}
\label{s:Main_Results}

Before addressing the approximability of the rules-to-equilibrium operator $\mathfrak{R}$, as defined in~\eqref{eq:R2EOp}, we first establish its well-posedness. Specifically, we show that $\mathfrak{R}$ is well-defined and locally Lipschitz continuous. 

In what follows we fix a priori radii $\rho_A,\rho_B$, $\rho_{F_2}>0$ determining the maximal allowable perturbations on the rules of the MFG, relative to a set of reference rules $(A^{\dagger},B^{\dagger},F_2^{\dagger})$.  The set of all such perturbations is quantified by 
the ellipsoid 
\begin{equation}
\label{eq:Ellispoid}
        \mathbb{B}_{\mathcal{H}}(\rho_A,\rho_B,\rho_{F_2})
    \eqdef 
        \big\{
            \|\Delta A\|_{\mathcal{HS}(U)}
        \le 
            \rho_A
        \,
        ,
            \|\Delta B\|_{\mathcal{HS}(U,H)}
        \le 
            \rho_B
        ,
        \,
            \|\Delta F_2\|_{\mathcal{HS}(U,
                    \mathcal{HS}(V,H)
                )}
        \le 
            \rho_{F_2}
    \big\}
.
\end{equation}
This result follows directly from our main technical stability estimate, stated in Theorem~\ref{thrm:Stability_R2E} below, relative to the reference rules $(A^{\dagger},B^{\dagger},F_2^{\dagger})$ and within the ellipsoid $\mathbb{B}_{\mathcal{H}}(\rho_A,\rho_B,\rho_{F_2})$ in $\mathcal{H}$.
\begin{proposition}[Local Well-Posedness of the Rules-to-Equilibrium Operator]
\label{prop:Solution_Operator_continuous}
Fix a reference model $(A^{\dagger},B^{\dagger},F_2^{\dagger})$ and radii $\rho_A,\rho_B,\rho_{F_2}>0$.
{There exists a time $T^{\star}>0$, depending only on the radii  $(\rho_A,\rho_B,\rho_{F_2})$ and on the $C_0$-semi-group of $A^{\dagger}$, such that}
the operator $\mathfrak{R}$ is well-defined and, for every  $(A,B,F_2),(\tilde{A},\tilde{B},\tilde{F}_2) $ in $\mathbb{B}_{\mathcal{H}}(\rho_A,\rho_B,\rho_{F_2})$, $\mathfrak{R}$ satisfies
 \begin{align} 
    \big\| \mathfrak{R}(A,B,F_2)  - \mathfrak{R}(\tilde{A},\tilde{B},\tilde{F}_2) \big\|_{\Lt} 
 \leq 
 \,
    L
    \,
    \big\|
        (A,B,F_2)
        -
        (\tilde{A},\tilde{B},\tilde{F}_2)
    \big\|_{\HRules}
 \notag 
 \end{align} 
 where the Lipschitz constant $L> 0$ depends only\footnote{An explicit dependence of $L$ on $\mathbb{B}_{\mathcal{H}}(\rho_A,\rho_B,\rho_{F_2})$ is given through its relationship to the constants $C^u_{A, A^\dagger}, C^{u,1}_{B, B^\dagger}, C^{u,2}_{B, B^\dagger}, C^{u}_{F_2, F_2^\dagger}>0$ as precisely shown in the stability estimate of Theorem~\ref{thrm:Stability_R2E}.}~%
on the ellipsoid $\mathbb{B}_{\mathcal{H}}(\rho_A,\rho_B,\rho_{F_2})$.
\end{proposition} 
Rigorously speaking, our results concern a regularized version of the rules-to-equilibrium map~\eqref{eq:R2EOp}, whose existence and \textit{extension properties} are now established.
\begin{proposition}[Global Lipschitzness of the Regularized Rules-to-equilibrium Map]
\label{prop:ExtensionExistence}
{In the setting of Proposition~\ref{prop:Solution_Operator_continuous}:}
For every compact subset $\mathcal{K} \subseteq \mathbb{B}_{\mathcal{H}}(\rho_A,\rho_B,\rho_{F_2})$, there is an $L_{\mathcal{K}}\ge 0$, depending only on $\mathbb{B}_{\mathcal{H}}(\rho_A,\rho_B,\rho_{F_2})$, and an $L_{\mathcal{K}}$-Lipschitz extension $\mathfrak{R}^{\mathcal{K}}:\mathcal{H}\to \Lt$ of $\mathfrak{R}|_{\mathcal{K}}$ to all of $\mathcal{H}$.\ 
 That is\ 
\begin{equation}
\label{eq:Def_Extension_R2E}
        \mathfrak{R}^{\mathcal{K}}(A,B,F_2) 
    =
        \mathfrak{R}(A,B,F_2),
\qquad
\text{for each }
(A,B,F_2)\in \mathcal{K}.
\end{equation}
\end{proposition}

We refer to any map $\mathfrak{R}^{\mathcal{K}}$ given by Proposition~\ref{prop:ExtensionExistence} as the regularized-rules-to-equilibirum map.  This map is an extension of the rules-to-equilibrium map $\mathfrak{R}$ beyond any a-priori prescribed compact set $\mathcal{K}$ which has the additional property that it is \textit{globally Lipschitz}, seemingly unlike the rules-to-equilibrium map; thus, $\mathfrak{R}^{\mathcal{K}}$ effectively \textit{regularizes} $\mathfrak{R}$ using the geometry of $\mathcal{K}$.  Our theory is constructed relative to this extension.  
\begin{theorem}[Regularized Rules-to-Equilibria Operators are PAC-Learnable by RNOs]
\label{thrm:main}
Fix a reference model $(A^{\dagger},B^{\dagger},F_2^{\dagger})$ and radii $\rho_A,\rho_B,\rho_{F_2}>0$ and suppose that Assumption~\ref{ass:Statistical} holds. 
%%%
There exists a time $T^{\star}>0$, depending only on the radii  $(\rho_A,\rho_B,\rho_{F_2})$ and on the $C_0$-semi-group of $A^{\dagger}$
and there is a compact subset $\mathcal{K}\subseteq \mathbb{B}_{\mathcal{H}}(\rho_A,\rho_B,\rho_{F_2})$ of positive $\mathbb{P}_X$-measure 
{and an $r>0$, depending only on $\mathcal{K}$,}
such that: For every ``approximation error'' $\varepsilon>0$, and each ``failure probability'' $0<\delta\le 1$ there is a connectivity parameter $C>0$, depending on $\varepsilon,\delta$ and on $\mathcal{K}$, such that if $\hat{F}\in \mathcal{RNO}^{1,L}_C(e_{\cdot},\eta_{\cdot})$ {(i.e.\ $\alpha=1$)}
is empirical risk \textit{minimizer}; i.e.\
\begin{equation}
\tag{ERM}
\label{eq:ERM_CONDITION_MAIN}
        \frac1{N}
        \sum_{n=1}^N
        \,
            \|
                \hat{F}(X_n)
                -
                \mathfrak{R}^{\mathcal{K}}(X_n)
            \|_{\Lt}
    =
        \inf_{
                \tilde{F}
            \in 
                \mathcal{RNO}^{1,L}_C(e_{\cdot},\eta_{\cdot})
        }\,
        \frac1{N}
        \sum_{n=1}^N
        \,
        \|
            \tilde{F}(X_n)
            -
            \mathfrak{R}^{\mathcal{K}}(X_n)
        \|_{\Lt}
\end{equation}
then, $\hat{F}$ must satisfy the following PAC-learnability guarantee 
\begin{equation*}
\resizebox{0.95\hsize}{!}{$%
\begin{aligned}
    {\mathbb{P}}
    % _X
    \Biggl(
            % \mathcal{R}(\hat{F})
            \mathbb{E}_{X\sim \mathbb{P}_X}\big[
                \|
                    \hat{F}(X)
                    -
                    \mathfrak{R}^{\mathcal{K}}(X)
                \|_{\Lt}
            \big]
        \le 
                \varepsilon
            +
                \bar{L}
                     e^{-\sqrt{
                        \log(N^{2r})
                    }}
                +
                    \frac{
                    \ln\big(
                            \frac{2}{\delta}
                        \big) 
                    }{N}
                    +
                    \frac{
                        \sqrt{
                        \ln\big(
                            \frac{2}{\delta}
                        \big) 
                        }
                    }{
                        \sqrt{N}
                    }
    \Biggr)
\ge 
    \mathcal{P}_X(\mathcal{K})^N-\delta
\end{aligned}
$}
\end{equation*}
where $\bar{L}\eqdef 2\max\{L_{\mathcal{K}},1\}$ and $\mathcal{K}$ is as in Proposition \ref{prop:ExtensionExistence}.
\end{theorem}
Theorem~\ref{thrm:main} establishes that regularized rules-to-equilibrium maps are PAC learnable by an RNO $\hat{F}$, provided that $\hat{F}$ is an empirical risk minimizer—i.e., it satisfies condition~\eqref{eq:ERM_CONDITION_MAIN}. The following result extends this finding by not only guaranteeing the existence of such an RNO $\hat{F}$ (i.e., that the left-hand side of~\eqref{eq:ERM_CONDITION_MAIN} is realizable by some RNO $\hat{F}$), but also showing that such an $\hat{F}$ can be chosen to have a surprisingly small number of non-zero (trainable) parameters.
In what follows, we use $W_0$ to denote the principal branch of the Lambert $W$ function\footnote{The principal branch $W_0$ of the Lambert $W$ function is the solution of $W e^W=z$ that is real-valued for $z\geq -1/e$ and satisfies $W_0(0)=0$.}
.
\begin{theorem}[Small Empirical Risk Minimizing RNOs Exist]
\label{thrm:main_pt2__quantitative_existenceERM}
In the setting of Theorem~\ref{thrm:main}, there exists an RNO $\hat{F}$ satisfying~\eqref{eq:ERM_CONDITION_MAIN} which has depth $\mathcal{O}(1)$ and both the width $W$ and the connectivity $C$ of $\hat{F}$ are in the order of $\mathcal{O}\big(
        \log(\varepsilon^{-1})
        /
        \sqrt[c_r]{\varepsilon}
    \big)$; where $
    c_{r
    % ,\alpha
    }\eqdef
    % \tfrac{1}{\alpha}
    \Big\lceil
        \log\big(
            \tfrac{\sqrt[r]{4L}}{\sqrt[r]{r}}
        \big)
    \Big\rceil$.
\end{theorem}

We now examine our main results by situating them within the supporting theoretical framework on which they rely.  Our proofs will merge elements of dynamic game-theory, via estimates related to the stability analysis of LQ MFG with respect to its coefficients, with approximation and PAC-learning results for \textit{operator learning}.
\subsection{Explanation of Proofs via Supporting Results}
\label{s:ProofOverview}
The joint proofs of Theorems~\ref{thrm:main} and~\ref{thrm:main_pt2__quantitative_existenceERM} will be undertaken in two overall steps.  All technical details are relegated to Sections \ref{s:Proofs} and \ref{s:Proofs_PACLearning}.
First, we will prove that the rules-to-equilibrium map $\mathfrak{R}$ is well-posed by establishing the local Lipschitz dependence of the equilibrium strategy of each LQ MFG given by~\eqref{dx_LQ_MFG}-\eqref{Jinfty} on the rules $(A,B,F_2)$.  Our \textit{stability}, i.e.\ local-Lipschitz, estimates will allow for a more general setting than is treatable in Theorem~\ref{thrm:main} and will quantify the stability in terms of the operator and supremum norms. Specifically, in this section, we consider more general operator spaces--namely, the space of bounded linear operators--and a general Lipschitz continuous function 
$f$, instead of the rules-to-equilibrium operator $\mathfrak{R}$.

In the remainder of the paper, we denote $C(\mfT; H)$ as the set of all continuous mappings $h:\mfT \rightarrow H$, a Banach space equipped with the supremum norm denoted by $\|.\|_{C(\mc{T};H)}$. Moreover, we denote $\mathcal{H}^2(\mfT;\mathcal{X})$ as the Hilbert 
space of all $\mathcal{X}$-valued progressively measurable processes $x$, equipped with the norm 
$
\left\|x \right\|_{\mathcal{H}^2(\mfT;\mathcal{X})}= \Big(\displaystyle \sup_{t \in \mfT} \mathbb{E}\left\|x(t) \right\|_{\mathcal{X}} ^{2}\Big )^\frac{1}{2}$. For brevity, we omit the index $\mathcal{X}$ when no confusion arises. Obviously, $\mathcal{H}^2(\mfT;\mathcal{X}) \subseteq \mathcal{M}^2(\mfT;\mathcal{X})$.

\begin{theorem}
\label{thrm:Stability_R2E}
{Consider a reference set of rules $(A^{\dagger},B^{\dagger}, F_1, D, E, F_2^{\dagger}, \sigma, M, \widehat{F}_1, G, \widehat{F}_2)$ for the MFG system \eqref{dx_LQ_MFG}-\eqref{Jinfty} 
and let $(\Pi^\dagger, q^\dagger, \bar{x}^\dagger, x^\dagger, u^\dagger)$ denote the solution to this MFG system given by \eqref{u_MFG_eqm}-\eqref{x_MFG_eqm}. For any radii $\rho_A,\rho_B,\rho_{F_2}>0$, there exists a time $T^{\star}>0$, depending only on the radii  $(\rho_A,\rho_B,\rho_{F_2})$ and on the $C_0$-semi-group of $A^{\dagger}$ such that:
}
\hfill\\
 For any solutions $(\Pi^{A}, q^{A}, \bar{x}^{A}, x^{A}, u^{A})$, $(\Pi^{B}, q^{B}, \bar{x}^{B}, x^{B}, u^{B})$, and $(\Pi^{F_2}, q^{F_2}, \bar{x}^{F_2}, x^{F_2}, u^{F_2})$ of the MFG system, corresponding to the operators $A^\dagger$, $B^\dagger$, and $F_2^\dagger$ being perturbed to $A$ (provided that $A-A^{\dagger}$ is bounded), $B$, and $F_2$, respectively, while other operators remain unchanged, the following stability estimates hold. 
\begin{align} 
& \begin{cases}
& \sup_{t\in\mathcal{T}} \big\| \Pi^A(t) - \Pi^{\dagger}(t) \big\|_{\mathcal{L}(H) } \leq C^{\Pi}_{A, A^\dagger} \big\| A - A^\dagger \big\|_{\mathcal{L}(H)} , 
\notag \allowdisplaybreaks\\ 
& \|\bar{x}^A - \bar{x}^{\dagger} \|_{C(\mathcal{T}; H)} 
  \leq C^{\bar{x}}_{A, A^\dagger} \big\| A - A^\dagger \big\|_{\mathcal{L}(H)} ,  
\notag \allowdisplaybreaks\\ 
&  \| q^A - q^{\dagger} \|_{C(\mathcal{T}; H)} 
  \leq C^{q }_{A, A^\dagger} \big\| A - A^\dagger \big\|_{\mathcal{L}(H)} , 
  \notag \allowdisplaybreaks\\ 
 & \| x^A - x^{\dagger} \|_{\mathcal{H}^2(\mfT;{H})} \leq C^{x}_{A, A^\dagger} \big\| A - A^\dagger \big\|_{\mathcal{L}(H)}, 
\notag \allowdisplaybreaks\\ 
& \| u^A - u^{\dagger} \|_{\mathcal{H}^2(\mfT; U)} \leq C^u_{A, A^\dagger} \big\| A - A^\dagger \big\|_{\mathcal{L}(H)} , 
\notag \\ 
\end{cases} \allowdisplaybreaks\\
%\end{align} 
%\begin{align} 
& \begin{cases}
& \sup_{t\in\mathcal{T}} \big\| \Pi^B(t) - \Pi^{\dagger}(t) \big\|_{\mathcal{L}(H)} \leq C^{\Pi, 1}_{B, B^\dagger} \big\| B - B^\dagger \big\|_{\mathcal{L}(U; H)} 
 + C^{\Pi, 2}_{B, B^\dagger} \big\| B - B^\dagger \big\|_{\mathcal{L}(U; H)}^2 , 
\notag \allowdisplaybreaks\\ 
& \|\bar{x}^B - \bar{x}^{\dagger} \|_{C(\mathcal{T}; H)} 
  \leq C^{\bar{x}, 1}_{B, B^\dagger} \big\| B - B^\dagger \big\|_{\mathcal{L}(U; H)}  
   + C^{\bar{x}, 2 }_{B, B^\dagger} \big\| B - B^\dagger \big\|_{\mathcal{L}(U; H)}^2  , 
\notag  \allowdisplaybreaks\\ 
&  \| q^B - q^{\dagger} \|_{C(\mathcal{T}; H)} 
  \leq C^{q,1 }_{B, B^\dagger} \big\| B - B^\dagger \big\|_{\mathcal{L}(U; H)} 
  + C^{q,2 }_{B, B^\dagger} \big\| B - B^\dagger \big\|_{\mathcal{L}(U; H)}^2 , 
  \notag \allowdisplaybreaks\\ 
& \| x^B - x^{\dagger} \|_{\mathcal{H}^2(\mfT;{H})} 
 \leq C^{x, 1}_{B, B^\dagger} \big\| B - B^\dagger \big\|_{\mathcal{L}(U; H)} 
 + C^{x, 2 }_{B, B^\dagger} \big\| B - B^\dagger \big\|_{\mathcal{L}(U; H)}^2 , 
\notag \allowdisplaybreaks\\ 
& \| u^B - u^{\dagger} \|_{\mathcal{H}^2(\mfT; {U})} \leq C^{u,1}_{B, B^\dagger} \big\| B - B^\dagger \big\|_{\mathcal{L}(U; H)} 
 + C^{u,1}_{B, B^\dagger} \big\| B - B^\dagger \big\|_{\mathcal{L}(U; H)}^2, 
\notag \allowdisplaybreaks\\ 
\end{cases} \allowdisplaybreaks\\ 
& \begin{cases} 
&   \Pi^{F_2}(t) = \Pi^{\dagger}(t),\, \forall t \in \mathcal{T} ,  
\notag \allowdisplaybreaks\\ 
& \|\bar{x}^{F_2} - \bar{x}^{\dagger} \|_{C(\mathcal{T}; H)} 
  \leq C^{\bar{x}}_{F_2, F_2^\dagger} \big\| F_2 - F_2^\dagger \big\|_{\mathcal{L}(H; \mathcal{L}(V;H))} , 
\notag \allowdisplaybreaks\\ 
&  \| q^{F_2} - q^{\dagger} \|_{C(\mathcal{T}; H)} 
  \leq C^{q }_{F_2, F_2^\dagger} \big\| F_2 - F_2^\dagger \big\|_{ \mathcal{L}(H; \mathcal{L}(V;H)) } , 
  \notag \allowdisplaybreaks\\  
& \| x^{F_2} - x^{\dagger} \|_{\mathcal{H}^2(\mfT;{H})} \leq C^{x}_{F_2, F_2^\dagger} \big\| F_2 - F_2^\dagger \big\|_{ \mathcal{L}(H; \mathcal{L}(V;H)) }, 
\notag \allowdisplaybreaks\\ 
& \| u^{F_2} - u^{\dagger} \|_{\mathcal{H}^2(\mfT; {U})} \leq C^u_{F_2, F_2^\dagger} \big\| F_2 - F_2^\dagger \big\|_{ \mathcal{L}(H; \mathcal{L}(V;H)) } .  
\notag 
\end{cases} 
\end{align} 
The constants 
$(C^{\Pi}_{A, A^\dagger}, C^{\bar{x}}_{A, A^\dagger}, C^{q}_{A, A^\dagger}, C^{x}_{A, A^\dagger}, C^{u}_{A, A^\dagger})$, 
$(C^{\Pi, i}_{B, B^\dagger}, C^{\bar{x}, i}_{B, B^\dagger}, C^{q, i}_{B, B^\dagger}, C^{x, i}_{B, B^\dagger}, C^{u, i}_{B, B^\dagger})$, $i=1, 2$, and \\
$(C^{\bar{x}}_{F_2, F_2^\dagger}, C^{q}_{F_2, F_2^\dagger}, C^{x}_{F_2, F_2^\dagger}, C^{u}_{F_2, F_2^\dagger})$ 
depend on the operators $(A^{\dagger},B^{\dagger}, F_1, D, E, F_2^{\dagger}, \sigma, M, \widehat{F}_1, G, \widehat{F}_2)$ and the time horizon $T^*$, as well as on the perturbed operators $A$, $B$, and $F_2$ in their respective cases.  
The explicit forms of these constants are detailed in the lemmas found in Sections~\ref{s:StabilityA}, \ref{s:StabilityB}, and~\ref{s:StabilityF2}. 
\end{theorem}
\begin{proof}
    See Sections~\ref{s:StabilityA}, \ref{s:StabilityB}, and~\ref{s:StabilityF2} for details.
\end{proof}

Having established that the rules-to-equilibrium operators are locally Lipschitz\footnote{We note that if $\Delta B= B-B^{\dagger}$ lies in a uniformly bounded subset of $\mathcal{L}(U; H)$, then quadratic terms $\big\| B - B^\dagger \big\|_{\mathcal{L}(U; H)}^2$ may be bounded by linear terms $\big\| B - B^\dagger \big\|_{\mathcal{L}(U; H)}$.}, the remaining step is to demonstrate that such operators are learnable from Gaussian samples by Lipschitz RNOs with favourable \textit{sample complexity}. We show that any $L$-Lipschitz map from $\mathcal{H}$ to $\mathcal{M}^2(\mfT,U)$ is PAC-learnable by RNOs that preserve the same $L$-Lipschitz regularity. Our sample complexity bounds are favourable precisely because RNOs can reproduce the Lipschitz continuity of the target map
\footnote{
Our approximation and PAC-learning guarantees apply more generally to arbitrary separable Hilbert spaces. For clarity of exposition, however, we restrict to the spaces $\mathcal{H}$ and $\mathcal{M}^2(\mfT,U)$ introduced above.
}.
\begin{theorem}[Sample Complexity of Lipschitz-RNOs When Learning Non-Linear Lipschitz Operators]
\label{thrm:MainLearning}
Let $L\ge 0$ and $f:\mathcal{H}\to \mathcal{M}^2(\mfT,U)$ be an $L$-Lipschitz operator. Suppose that Assumption~\ref{ass:Statistical}
with Assumption~\ref{ass:Statistical} (ii) generalized to $Y_n=f(X_n)$ for $n=1,\dots,N$, holds.  
{For any compact $\mathcal{K}\subseteq \mathcal{H}$ of measure $p\eqdef \mathbb{P}_X(\mathcal{K})>0$ containing $0\in \mathcal{H}$}, for every ``approximation error'' $\varepsilon>0$, and each ``failure probability'' $0<\delta\le 1$ there is a connectivity parameter $C\eqdef C(\epsilon,\delta,\mathcal{K})\in \mathbb{N}_+$ such that for any empirical risk minimizer  $\hat{F}\in \mathcal{RNO}^{1,L}_C(e_{\cdot},\eta_{\cdot})$; i.e.\ 
\begin{equation}
\tag{ERM}
\label{eq:ERM_CONDITION_MAIN2}
              \frac1{N}
        \sum_{n=1}^N
        \,
            \|
                \hat{F}(X_n)
                -
                f(X_n)
            \|_{\Lt}
    =
        \inf_{
                \tilde{F}
            \in 
                \mathcal{RNO}^{1,L}_C(e_{\cdot},\eta_{\cdot})
        }\,
               \frac1{N}
        \sum_{n=1}^N
        \,
        \|
            \tilde{F}(X_n)
            -
            f(X_n)
        \|_{\Lt}
\end{equation}
must satisfy the following learnability guarantee
\begin{equation*}
\resizebox{0.95\hsize}{!}{$%
\begin{aligned}
    {\mathbb{P}}
       \Biggl(
                       \mathbb{E}_{X\sim \mathbb{P}_X}\big[
                \|
                    \hat{F}(X)
                    -
                    f(X)
                \|_{\Lt}
            \big]
    \le 
            \varepsilon
        +
            \bar{L}
                 e^{-\sqrt{
                    \log(N^{2r})
                }}
            +
                \frac{
                \ln\big(
                        \frac{2}{\delta}
                    \big) 
                }{N}
                +
                \frac{
                    \sqrt{
                    \ln\big(
                        \frac{2}{\delta}
                    \big) 
                    }
                }{
                    \sqrt{N}
                }
\Bigg)
\ge 
    p^N-\delta
\end{aligned}
$}
\end{equation*}
where $\bar{L}\eqdef 2\max\{L,1\}$ {and $r>0$ is as in Assumption~\eqref{ass:Statistical} (iii).}  Furthermore, if $C\in \mathbb{N}_+$ is large enough, then there exists some $\hat{F}\in \mathcal{RNO}_C^{1,L}(e_{\cdot},\eta_{\cdot})$ satisfying~\eqref{eq:ERM_CONDITION_MAIN2}.
\end{theorem}
\begin{proof}
    See Proposition~\ref{prop:LearninGuarantee} for a generalization.
\end{proof}
\begin{remark}
If $\mathbb{P}_N$ is compactly supported on $\mathcal{K}$ then, the lower-bound in Theorem~\eqref{thrm:MainLearning} is $1-\delta$.
\end{remark}
The proof of Theorem~\ref{thrm:MainLearning} relies on our infinite-dimensional \textit{regular} universal approximation theorem for \textit{Lipschitz} RNOs. This theorem extends recent advances in Lipschitz-constrained neural network approximation~\cite{hong2024bridging,riegler2024generating,murari2025approximation} to the infinite-dimensional setting. Our analysis builds on recent developments in optimal transport techniques for establishing PAC-learning guarantees; see~\cite{amit2022integral,hou2023instance,kratsios2024tighter,benitez2024out,detering2025learninggraphtransductivelearning}.

\subsection{Unlocking Favourable Rates}
Favourable rates are obtainable when approximating a function on a suitable compact subset of the domain, called an \textit{exponentially ellipsoidal} set in~\cite{alvarez2024neural}.  In~\cite{galimberti2022designing}, it was shown that exponentially ellipsoidal subsets of Lebesgue ``function'' spaces on Euclidean domains correspond to infinitely smooth functions; thus, we may interpret the following as a ``smoothness condition'' on inputs.
\begin{definition}[$(\rho,r)$-Exponentially Ellipsoidal]
\label{defn:Ellispoids}
We say that $\mathcal{K}\subseteq \mathcal{H}$ is $(\rho,\bar{r})$-\textit{exponentially ellipsoidal} to $(e_i)_{i\in I}$ if there exists some $x^{\dagger}\in \mathcal{H}$ and some $\bar{r},\rho>0$ such that 
\begin{equation}
\label{eq:Exp_Ellipsoidal}
    \mathcal{K}
    \subseteq 
    \biggl\{
        x\in \mathcal{H}:\,
        x-x^{\dagger}=\sum_{i\in I}\, \langle x-x^{\dagger},e_i\rangle e_i
        \mbox{ and }
        |\langle x-x^{\dagger},e_i\rangle|\le \rho\, e^{-ir}
    \biggr\}
.
\end{equation}
We say that $\mathcal{K}$ is centred if $x^{\dagger}=0\in \mathcal{H}$; in which case $0\in \mathcal{K}$.
Intuitively, $\bar{r}>0$ captures how ``close'' the ``ellipsoidal set'' $\mathcal{K}$ is aligned to the span of the basis vectors in $(e_i)_{i\in I}$ while $\rho$ captures its ``radius''.
\end{definition}
\begin{proposition}[Regular Universal Approximation for Lipschitz RNOs]
\label{prop:theorem_universality__regular___specialcase}
% \hfill\\
Under Assumption~\ref{ass:Statistical}, suppose that $0\in \mathcal{K}$, $\mathcal{K}\subset \HS$ is non-empty and compact, $0<\alpha\le 1$, $L>0$, $f:\mathcal{K}\to \Lt$ be an $(\alpha,L)$-H\"{o}lder map, and fix an error $\varepsilon>0$ and a $x^{\dagger}\in \HRules$.
There is an $(\mathbf{\alpha},{L})$-H\"{o}lder RNO $\hat{F}:\HS\to \Lt$ satisfying 
\[
    \sup_{x\in \mathcal{K}}\,
        \big\|
        \hat{F}(\Delta x)
            -
            f(x)
        \big \|_{\Lt}
    \le 
        \varepsilon
.
\]
Furthermore, $\hat{F}$ is base-point preserving; in that $\hat{F}(0)=\hat{y}^{\dagger}=f(x^{\dagger})$.
\hfill\\
If $\mathcal{K}$ is centred $(r,\varepsilon)$-exponentially ellipsoidal for some $r>0$, then there is an $\hat{F}$ of depth $\mathcal{O}(1)$ and whose number of non-zero parameters are $\mathcal{O}\Big(
        \varepsilon^{-c_{r,\alpha}}
        \log(\varepsilon^{-1})
    \Big)$; where $
    c_{r,\alpha}\eqdef
    \frac1{\alpha}
    \Big\lceil
        \log\big(
            \tfrac{(4L)^{\frac1{r\alpha}}}{r^{1/r}}
        \big)
    \Big\rceil$.
If, additionally, $0<L<4e^{\alpha}$ and $r=
% -
1/W_0\Big(
% -\,
(4L)^{-\tfrac{1}{\alpha}}\Big)$ then, its width and connectivity are
$\mathcal{O}(\varepsilon^{-1}\log(\varepsilon^{-1}))$.
\end{proposition}
\begin{proof}
   See the more general result proved in Proposition~\ref{prop:theorem_universality__regular} in Section \ref{sec:Approx-Guarantees}.
\end{proof}
In situations where the user controls the sampling distribution $\mathbb{P}_X$—used to select points on $\mathcal{K}$ for training the neural operator—it becomes possible to tune $\mathbb{P}_X$ so that points from $\mathcal{K}$ are sampled with high probability. While this is not always feasible in theory, it can often be arranged in simulations.
By coupling the radius $\rho$ of the $(\rho,r)$-exponentially ellipsoidal set $\mathcal{K}$ to the approximation error, and selecting the sampling “temperature” $r>0$ in Assumption~\ref{ass:Statistical} (iii) according to the failure probability $\delta$, we guarantee high-probability learnability by small neural operators. This temperature coupling is formalized in the following assumption.

\begin{assumption}[Tempered Sampling]
\label{ass:high_tempirature}
Fix a desired sampling failure probability $0<\delta_X \le 1$, let $\Delta \eqdef \tfrac{\ln(1/(1-\delta_X))}{2+\ln(1/(1-\delta_X))}$, let $r>0$ and fix an $(\rho,r)$-ellipsoidal set $\mathcal{K}$ containing $0$.
For simplicity, assume that: for every $i\in \mathbb{N}_+$, and every $t>0$ $\mathbb{P}(|Z_i|\ge t)\le 2e^{-t^2/2}$.  
\hfill\\
Suppose that we pick $\mathbb{P}_X$ in Assumption~\ref{ass:Statistical} so that: for every $i\in \mathbb{N}_+$
\[
        0
    <
        \sigma_i 
    \le 
        \frac{
            e^{-ri}
        }{
            \sqrt{
                \ln(2) - i\ln(\Delta)
            }
        }
.
\]
\end{assumption}
In situations where the user controls the sampling distribution $\mathbb{P}_X$—used to select points on $\mathcal{K}$ for training the neural operator—we obtain the following strengthened version of Theorem~\ref{thrm:MainLearning}.
\begin{theorem}[{Quantitative Refinement of Theorem~\ref{thrm:MainLearning}}]
\label{thrm:MainLearning__Faster}
{In the setting of Theorem~\ref{thrm:MainLearning} if, additionally, there is a constant $\bar{r}>0$ such that $\mathcal{K}$ is a centred $(\varepsilon,\bar{r})$-exponentially ellipsoidal%
\footnote{Cf.\ Definition~\ref{defn:Ellispoids}.} set%
then, there exists some $\hat{F}\in \mathcal{RNO}_C^{1,L}(e_{\cdot},\eta_{\cdot})$ satisfying~\eqref{eq:ERM_CONDITION_MAIN2} whose depth is $\mathcal{O}(1)$ and whose width and number of non-zero parameters are both $\mathcal{O}\big(
        \varepsilon^{-c_{\bar{r},1}}
        \log(\varepsilon^{-1})
    \big)$; where 
$
    c_{\bar{r},1}
\eqdef
    % \frac1{1}
    \Big\lceil
        \tfrac{\log(4\bar{L})}{\bar{r}}
        -
        \tfrac{\log(\bar{r})}{\bar{r}}
    \Big\rceil
$;
satisfying
\begin{equation*}
\resizebox{0.95\hsize}{!}{$%
\begin{aligned}
    {\mathbb{P}}
    \Biggl(
            \mathbb{E}_{X\sim \mathbb{P}_X}\big[
                \|
                    \hat{F}(X)
                    -
                    f(X)
                \|_{\Lt}
            \big]
    \le 
            \varepsilon
        +
            \bar{L}
                 e^{-\sqrt{
                    \log(N^{2r})
                }}
            +
                \frac{
                \ln\big(
                        \frac{2}{\delta}
                    \big) 
                }{N}
                +
                \frac{
                    \sqrt{
                    \ln\big(
                        \frac{2}{\delta}
                    \big) 
                    }
                }{
                    \sqrt{N}
                }
\Bigg)
\ge 
        e^{-\delta}
    -
        \delta
\end{aligned}
$}
\end{equation*}
}
\end{theorem}
\begin{proof}
    See Appendix~\ref{a:PROOF_thrm:MainLearning__Faster}.
\end{proof}

\section{Proofs of Stability Estimates}
\label{s:Proofs}
This section contains detailed proofs of our main results. 
Consider a reference set of rules $(A^{\dagger},B^{\dagger}, F_1, D, E,\\ F_2^{\dagger}, \sigma, M, \widehat{F}_1, G, \widehat{F}_2)$ for the MFG system \eqref{dx_LQ_MFG}-\eqref{Jinfty} 
and let $(\Pi^\dagger, q^\dagger, \bar{x}^\dagger, x^\dagger, u^\dagger)$ denote the solution to this MFG system given by \eqref{u_MFG_eqm}-\eqref{x_MFG_eqm}. 
In Sections~\ref{s:StabilityA}, \ref{s:StabilityB}, and \ref{s:StabilityF2}, we perturb the operators $A^\dagger$, $B^\dagger$, and $F_2^\dagger$ to $A$ (provided that $A-A^{\dagger}$ is bounded), $B$, and $F_2$, respectively, while other operators remain unchanged, and derive estimates for the resulting changes in the system. 

We first present regularity results for the operators and processes associated with the reference model, i.e., $(\Pi^\dagger, q^\dagger, \bar{x}^\dagger, x^\dagger, u^\dagger)$. In this section, the parameters $R_1$, $R_2$, $R_3$, $R_4$, and $R_5$ are defined by \eqref{Gamma1bd}–\eqref{Delta3bd}, respectively. Suppose $S^\dagger(t)$ denotes the $C_0$-semigroup associated with the infinitesimal generator $A^\dagger$ such that $\Vert S^{A^\dagger}(t) \Vert \leq M^{A^\dagger}_T$, for all $t\in \mathcal{T}$.  According to~\cite[Proposition 4.4]{LF24}, the solution $\Pi^\dagger$ of the operator differential Riccati equation given by \eqref{ODE_Pi} for the reference set of rules satisfies 
\begin{align} 
\begin{aligned} 
& \| \Pi^\dagger(t) \| \leq C^{\Pi^\dagger},\quad \forall t \in \mathcal{T}, 
\label{|Pi|<=CPi} \\ 
& C^{\Pi^\dagger} \eqdef  2 (M_T^{A^\dagger})^2 \exp \big( 8 T (M_T^{A^\dagger})^2 \| D \|^2 \mathrm{tr}(Q) \big) 
\big( \| G \| + T \| M \| \big). \end{aligned} 
\end{align} 
\begin{lemma} 
\label{lem:|q|and|barx|bound}
Suppose $M^{A^\dagger}_T ( C^{\Psi, \bar{x}^\dagger } T +  \| G\| \| \widehat{F}_2 \| ) < 1 $. The mean field $\bar{x}^\dagger$ and the offset term $q^\dagger$ associated with the reference model, given by ~\eqref{ODE_q_MFG} and~\eqref{eq:barx_MFG}, respectively, satisfy 
\begin{align} 
& \big\| \bar{x}^\dagger \big\|_{C(\mathcal{T}; \mathcal{H} ) } 
\leq C^{\bar{x}^\dagger} ,  
 \label{|barx|<=Cbarx} 
 \allowdisplaybreaks\\ 
& \big\| q^\dagger \big\|_{C(\mathcal{T}; \mathcal{H} )} 
\leq C^{q^\dagger} , 
\label{|q|<=Cq} 
\end{align} 
where 
\begin{align}  
 & C^{\bar{x}^\dagger} =  \big[ 1 -  M^{A^\dagger}_T ( C^{\Psi, \bar{x}^\dagger } T +  \| G\| \| \widehat{F}_2 \| ) \big]^{-1}   
  \big[ M^{A^\dagger}_T (  C^{\Phi, c, \dagger } T + |\bar\xi| )  \exp( M^{A^\dagger}_T C^{\Phi, \bar{x}^\dagger } T)  
 \notag \allowdisplaybreaks\\ 
 & \hspace{6.3cm}  +  M^{A^\dagger}_T C^{\Phi, q^\dagger} T \big(  M^{A^\dagger}_T C^{\Psi, c, \dagger } T  \big) \exp\big( M^{A^\dagger}_T 
  ( C^{\Phi, \bar{x}^\dagger } + C^{\Psi, q^\dagger } ) T \big)  \big] ,  
 \notag \allowdisplaybreaks\\ 
 & C^{q^\dagger}  
 =   M^{A^\dagger}_T C^{\Psi, c, \dagger } T  \exp( M^{A^\dagger}_T C^{\Psi, q^\dagger } T) 
 + M^{A^\dagger}_T ( C^{\Psi, \bar{x}^\dagger } T + \| G \| \| \widehat{F}_2 \| ) \exp( M^{A^\dagger}_T C^{\Psi, q^\dagger} T) 
 C^{\bar{x}^\dagger } , 
 \notag \allowdisplaybreaks\\ 
& C^{\Phi, \bar{x}^\dagger } =  \| B \| ( \|B\| + R_3 + R_2 \|F_2^\dagger \| ) C^{\Pi^\dagger }   + \|F_1\|  , 
 \notag \\ 
 &  C^{\Phi, q^\dagger} = \| B^\dagger \|^2 , 
 \quad 
  C^{\Phi, c, \dagger } = \|B^\dagger \| R_2 \|\sigma\| C^{\Pi^\dagger}  ,  
 \notag \allowdisplaybreaks\\ 
& C^{\Psi, q^\dagger } = C^{\Pi^\dagger} (\| B^\dagger  \| + R_3 ) 
  \big\| B^\dagger \big\| , 
  \notag \\ 
& C^{\Psi, \bar{x}^\dagger } 
 = \big[ R_1 C^{\Pi^\dagger} + ( C^{\Pi^\dagger} )^2 ( \| B^\dagger \| + R_3 ) R_2 \big] \big\| F_2^\dagger \big\| 
  + C^{\Pi^\dagger} \big\| F_1 \big\|  
 + \big\| M \big\|  \big\|  \hat{F}_1 \big\| ,   
\notag \allowdisplaybreaks\\ 
& C^{\Psi, c, \dagger } =   R_1 C^{\Pi^\dagger} + (C^{\Pi^\dagger})^2 (\|B^\dagger\| + R_3 ) R_2. 
\notag 
\end{align} 
\end{lemma} 
\begin{proof}
    See \ref{sec:proof:lem:|q|and|barx|bound}.
\end{proof}

\begin{lemma} 
\label{lem:E|x|<=Cx}
The equilibrium state $x^\dagger$ for the reference model as given by~\eqref{x_MFG_eqm} satisfies 
\begin{align} 
\big( \sup_{t\in\mathcal{T}}\mathbb{E} \| x^\dagger (t) \|^2 \big)^{\frac{1}{2}} 
\leq 
C^{x^\dagger },  
\notag 
\end{align} 
where 
\begin{align} 
 C^{x^\dagger} = & \big\{  3(M^{A^\dagger}_T)^2 \mathbb{E}\|\xi\|^2 
 + 3 T (M_T^{A^\dagger})^2 
  \big[ ( \| B^\dagger \|^2 + \| E \|^2 ) \big(  \| B^\dagger \| C^{q^\dagger} 
+ R_2 ( \|F_2^\dagger \|  C^{\bar{x}^\dagger} + \|\sigma\| ) C^{\Pi^\dagger } \big)  \big] 
\notag \\ 
  &\quad+ 3T(M^{A^\dagger}_T)^2 ( \| F_1 \|^2 + \| F_2^\dagger \|^2 ) ( C^{\bar{x}^\dagger } )^2 
 + 3T(M^{A^\dagger}_T)^2 \| \sigma \|^2 
 \big\}^{1/2}  \notag \\ 
&\quad \times \exp\big\{ (M^{A^\dagger}_T)^2 \big[ ( \| B^\dagger \|^2 + \| E \|^2 ) \big( (\|B^\dagger \| + R_3 ) C^{\Pi^\dagger} \big)^2 + \| D \|^2 \big] 3T/2   
\big\}. 
\notag 
\end{align} 
\end{lemma} 
\begin{proof}
   See \ref{sec:proof:lem:E|x|<=Cx}.    
\end{proof}
 
\subsection{Stability of the equilibrium with respect to operator \texorpdfstring{$A$}{A}}
\label{s:StabilityA} 

In this section, we perturb the operator $A^\dagger$ to $A$ and denote by $(\Pi^A, \bar{x}^A, q^A, u^A, x^A)$ the solution to the MFG system, given by \eqref{u_MFG_eqm}-\eqref{x_MFG_eqm}, corresponding to the set of rules $(A$, $B^\dagger$, $D$, $E$, $F_1$, $F_2^\dagger$, $\sigma$, $M$, $\widehat{F}_1$, $\widehat{F}_2$, $G)$.
\begin{lemma}\label{lem:A:|A1-A2|}  
Suppose $A$ and $A^\dagger$ are unbounded linear operators defined on the same domain, and assume that the difference operator $A-A^\dagger$ is a bounded linear operator, i.e. $A-A^\dagger \in \mathcal{L}(H)$. Furthermore, suppose that the perturbed operator $A$ generates a $C_0$-semigroup $S^A(t) \in \mathcal{L}(H)$, with $\big\| S^A(t) \big\|_{\mathcal{L}(H)}\leq M^{A}_T$, for all $t \in \mfT$. Then, for all $t\in\mathcal{T}$,  
  \begin{align} 
 \big\| S^A(t) - S^{A^\dagger}(t) \big\|_{\mathcal{L}(H)} \leq M^{A, A^\dagger}_T 
 \| A - A^\dagger \| , 
 \notag 
 \end{align} 
 where $M^{A, A^\dagger}_T \eqdef M^A_T M^{A^\dagger}_T$. 
\end{lemma} 
\begin{proof}
 See \ref{sec:lem:A:|A1-A2|}.  
 \end{proof}
 
 \begin{lemma}  
\label{lem:A:|Pi1-Pi2|}
For all $t\in \mathcal{T}$, the solution $\Pi^A$ of the operator differential Riccati equation, associated with the perturbed operator $A$, satisfies
\begin{align}
& \big\| \Pi^A(t) \big\| 
\leq C^\Pi_{A} , 
\notag \allowdisplaybreaks\\ 
& \big\| \Pi^A(t) - \Pi^\dagger(t) \big\| \leq 
 C^\Pi_{A, A^\dagger}  \big\| A - A^\dagger \big\| , 
 \label{|Pi1-Pi2|<=|A1-A2|}
\end{align} 
where 
\begin{align} 
C^{\Pi}_A \eqdef & 2 (M_T^{A})^2 \exp \big( 8 T (M_T^{A})^2 \| D \|^2 \mathrm{tr}(Q) \big) 
\big( \| G \| + T \| M \| \big) , 
\notag \allowdisplaybreaks\\ 
 C^\Pi_{A, A^\dagger} \eqdef & 
  \Big\{ 
 \| M \| (T+1) 
+  T (\|B^\dagger \|+R_3)^2 C^{\Pi^\dagger}  
   \Big\} ( 1+C^{\Pi^\dagger} + C^\Pi_{A} ) 
   \notag \allowdisplaybreaks\\ 
& \times 2 \sqrt{2} 
\big[
 M_T^{A^\dagger} \exp\big( 8 T (M_T^{A^\dagger})^2 \| D\|^2 \mathrm{tr}(Q) \big)  
 + 
 M_T^{A} \exp\big( 8 T (M_T^{A})^2 \| D\|^2 \mathrm{tr}(Q) \big)  
\big] 
\notag \allowdisplaybreaks\\ 
& \times \| D \| T^{1/2} 
 \Big[ M_T^{A^\dagger} \exp\big( 8 T (M_T^{A^\dagger})^2 \| D\|^2 \mathrm{tr}(Q) \big) \Big]  
  \exp(  (M_T^{A})^2 \| D \|^2 T )    
  \notag \allowdisplaybreaks\\ 
& \times \exp \Big\{ T(\|B^\dagger \| + R_3)^2 (1+R_5 C^\Pi_{A} ) 
 \Big[ C^\Pi 2 M_T^{A^\dagger} M_T^{A} \exp\big( 8 T \big( (M_T^{A^\dagger })^2 + (M_T^{A })^2 \big) \| D\|^2 \mathrm{tr}(Q) \big)  
\notag \allowdisplaybreaks\\ 
& 
+ C^\Pi_{A} 
\sqrt{2} (M_T^{A})  \exp\big( 8 T (M_T^{A})^2 \| D\|^2 \mathrm{tr}(Q) \big)  
\Big] 
\Big\} 
 M^{A, A^\dagger}_T . 
 \notag 
\end{align} 
\end{lemma}  
\begin{proof}
 See~\ref{sec:lem:A:|Pi1-Pi2|}.  
 \end{proof}

\begin{lemma}  
\label{lem:A:|q1-q2|and|barx1-barx2|}
Suppose that 
\begin{align} 
\label{eq:small_time_condition}
\begin{aligned}
 &  M^{A}_T ( C^{\Psi, \bar{x}}_{A} T +  \| G\| \| \widehat{F}_2 \| ) < 1 , 
\\
& ( M^{A}_T C^{\Psi,\bar{x}}_{A, A^\dagger} T + M^{A}_T \| G \| \| \widehat{F}_2 \| ) 
 M^{A}_T C^{\Phi, q}_{A, A^\dagger }T \exp\big( M^{A}_T 
  ( C^{\Phi, \bar{x}}_{A, A^\dagger} +C^{\Psi, q}_{A, A^\dagger } ) T \big) < 1.
\end{aligned}
\end{align} 
Then, the offset term $q^A$ and the mean field $\bar{x}^A$, corresponding to the perturbed operator $A$, satisfy 
\begin{align} 
& \| \bar{x}^A \|_{C(\mathcal{T}; H)} 
\leq C^{\bar{x}}_{A} , 
\notag \allowdisplaybreaks\\ 
& \| q^A \|_{C(\mathcal{T}; H)} 
\leq C^q_{A} , 
\notag \allowdisplaybreaks\\ 
& \| \bar{x}^A - \bar{x}^\dagger \|_{C(\mathcal{T}; H)} \leq 
C^{\bar{x}}_{A, A^\dagger} \big\| A - A^\dagger \big\| , 
\notag \allowdisplaybreaks\\ 
& \| q^A - q^\dagger \|_{C(\mathcal{T}; H)} \leq 
C^q_{A, A^\dagger} \big\| A - A^\dagger \big\| , 
\notag 
\end{align} 
where 
\allowdisplaybreaks
\begin{align} 
 C^{\bar{x}}_{A, A^\dagger}  
&=  \big[1 - ( M^{A}_T C^{\Psi,\bar{x}}_{A, A^\dagger} T + M^{A}_T \| G \| \| \widehat{F}_2 \| ) 
 M^{A}_T C^{\Phi, q}_{A, A^\dagger }T \exp\big( M^{A}_T 
  ( C^{\Phi, \bar{x}}_{A, A^\dagger} +C^{\Psi, q}_{A, A^\dagger } ) T \big) \big]^{-1} 
\notag \allowdisplaybreaks\\ 
& \times\! \Big\{ \! \Big[ M^{A, A^\dagger}_T      
   \big( C^{\Phi, \bar{x}^\dagger} C^{\bar{x}^\dagger}   
+ C^{\Phi, q^\dagger}  C^{q^\dagger} 
+ C^{\Phi, c, \dagger }  \big)T 
+ M^{A}_T C^{\Phi, \Pi}_{A, A^\dagger} C^\Pi_{A, A^\dagger} T
+ M_T^{A, A^\dagger} \big\| \bar\xi \big\|  \Big] 
\notag \allowdisplaybreaks\\ 
& 
\times \exp( M_T^{A} C^{\Phi, \bar{x}}_{A, A^\dagger} T ) + M^{A}_T C^{\Phi, q}_{A, A^\dagger} T  
\Big[ M^{A, A^\dagger }_T 
 \big( C^{\Psi, q^\dagger }  C^{q^\dagger}  
+C^{\Psi, \bar{x}^\dagger } C^{\bar{x}^\dagger }   
+ C^{\Psi, c, \dagger }  \big) T 
\notag \allowdisplaybreaks\\ 
& 
+ M^{A}_T C^{\Psi, \Pi}_{A, A^\dagger} C^{\Pi}_{A, A^\dagger} T + M^{A, A^\dagger}_T C^{\bar{x}^\dagger } \| G \| \| \widehat{F}_2 \|  \Big]   
 \exp\big( M^{A}_T ( C^{\Phi,\bar{x}}_{A, A^\dagger} + C^{\Psi, q}_{A, A^\dagger} ) T \big)
\Big\} , 
\notag \allowdisplaybreaks\\ 
%\end{align} 
%\begin{align} 
 C^q_{A, A^\dagger } 
=  &  
\Big[ M^{A, A^\dagger }_T 
 \big( C^{\Psi, q^\dagger}  C^{q^\dagger}  
+C^{\Psi, \bar{x}^\dagger } C^{\bar{x}^\dagger}   
+ C^{\Psi, c, \dagger}  \big) T + M^{A}_T C^{\Psi, \Pi}_{A, A^\dagger} C^{\Pi}_{A, A^\dagger} T 
 + M^{A, A^\dagger}_T C^{\bar{x}^\dagger } \| G \| \| \widehat{F}_2 \|  \Big]  
 \notag \allowdisplaybreaks\\ 
& \times\exp\big( M^{A}_T C^{\Psi, \bar{x} }_{A, A^\dagger } T  \big) 
 + \Big[ M^{A}_T C^{\Psi, \bar{x}}_{A, A^\dagger} T + M^{A}_T \| G \| \| \widehat{F}_2 \| \Big] 
\exp\big( M^{A}_T C^{\Psi, \bar{x} }_{A, A^\dagger } T  \big) 
C^{\bar{x}}_{A, A^\dagger} ,   
\notag \allowdisplaybreaks\\ 
C^{\Phi, \bar{x}}_{A, A^\dagger} = & \| B^\dagger \| ( \| B^\dagger \| + R_3 ) C^\Pi_{A} 
+  R_2 C^{\Pi^\dagger}\| B^\dagger \| \| F_2^\dagger \| + \| F_1 \|  , 
\notag \allowdisplaybreaks\\ 
C^{\Phi, q }_{A, A^\dagger} = & \| B^\dagger \|^2 , 
\notag \allowdisplaybreaks\\ 
C^{\Phi, \Pi }_{A, A^\dagger} = & \| B^\dagger \| R_5 (\| B^\dagger \| + R_3 ) C^{\Pi^\dagger} C^{\bar{x}^\dagger } 
+ \| B^\dagger \| (\| B^\dagger \| + R_3) C^{\bar{x}^\dagger } 
+ \| B^\dagger \|^2 R_3  C^{q^\dagger}  
\notag \allowdisplaybreaks\\ 
 & + \| B^\dagger \| R_5 R_2  ( \| F_2^\dagger \| C^{\bar{x}^\dagger} + \| \sigma\| ) C^{\Pi^\dagger}  
 + \| B^\dagger \| R_2  (\| F_2^\dagger \| C^{\bar{x}}_{A} + \| \sigma\| ) 
 + \| F_1\|, 
\notag \allowdisplaybreaks\\ 
C^{\Psi, \bar{x}}_{A, A^\dagger} = & R_1 \| F_2^\dagger \| C^{\Pi^\dagger} 
+ ( \| B^\dagger \| + R_3) C^{\Pi}_{A} R_2 \| F_2^\dagger \| C^{\Pi^\dagger} 
+ C^\Pi_{A} \| F_1 \| + \| M \| \| \widehat{F}_1 \| , 
\notag \allowdisplaybreaks\\ 
C^{\Psi, q }_{A, A^\dagger} = & ( \| B^\dagger \| + R_3 ) C^\Pi_{A} \| B^\dagger \| , 
\notag \allowdisplaybreaks\\ 
C^{\Psi, \Pi }_{A, A^\dagger} = & (\| B^\dagger \| + R_3 ) \| B^\dagger \| C^{q^\dagger}  
 + (\| B^\dagger \| + R_3 ) C^{\Pi}_{A} R_5 \| B^\dagger \| C^{q^\dagger}  
+ R_1 (\| F_2^\dagger \| C^{\bar{x}}_{A} + \| \sigma \| ),
\notag \allowdisplaybreaks\\ 
& + (\| B^\dagger \| + R_3 ) R_2 ( \| F_2^\dagger \| C^{\bar{x}^\dagger} + \| \sigma \| ) C^{\Pi^\dagger}   
+ (\| B^\dagger \| + R_3 ) C^{\Pi}_{A } R_5  R_2 ( \| F_2^\dagger \| C^{\bar{x}^\dagger } + \| \sigma \| ) C^{\Pi^\dagger}   
\notag \\ 
& + (\| B^\dagger \| + R_3 )C^\Pi_{A} R_2 (\| F_2 \| C^{\bar{x}}_{A} + \| \sigma\| ) 
+ \| F_1 \| C^{\bar{x}^\dagger}, 
\notag \\ 
 C^{\bar{x}}_{A} = &  \big[ 1 -  M^{A}_T ( C^{\Psi, \bar{x}}_{A} T +  \| G\| \| \widehat{F}_2 \| ) \big]^{-1}   \notag \\ 
 & \times \big[ M^{A}_T (  C^{\Phi, c}_{A} T + \big\| \bar\xi \big\| )  \exp( M^{A}_T C^{\Phi, \bar{x} }_{A} T)  
 +  M^{A}_T C^{\Phi, q}_{A} T \big(  M^{A}_T C^{\Psi, c}_{A} T  \big) 
  \exp\big( M^{A}_T ( C^{\Phi, \bar{x}}_{A} + C^{\Psi, q}_{A} ) T \big)  \big] ,  
 \notag \\ 
 C^q_{A}  
 = &   M^{A}_T C^{\Psi, c}_{A} T  \exp( M^{A}_T C^{\Psi, q}_{A} T )  
 + M^{A}_T ( C^{\Psi, \bar{x}}_{A} T + \| G \| \| \widehat{F}_2 \| ) \exp( M^{A}_T C^{\Psi, q}_{A} T) 
 C^{\bar{x}}_{A} , 
 \notag \\ 
 C^{\Phi, \bar{x}}_{A} = &  \| B^\dagger \| ( \|B^\dagger \| + R_3 + R_2 \|F_2^\dagger \| ) C^\Pi_{A}   + \|F_1\|  , 
 \notag \\ 
   C^{\Phi, q}_{A} = & \| B^\dagger \|^2,\notag \\ 
   C^{\Phi, c}_{A} = & R_2 \|B^\dagger \| \| \sigma \| C^\Pi_{A}  ,  
 \notag \\ 
 C^{\Psi, q}_{A} = & C^\Pi_{A} (\| B^\dagger \| + R_3 ) 
  \big\| B^\dagger \big\| , 
  \notag \\ 
 C^{\Psi, \bar{x}}_{A} 
 = & \big[ R_1 C^\Pi_{A} + ( C^\Pi_{A} )^2 (\| B^\dagger \| + R_3 ) R_2 \big] \big\| F_2^\dagger \big\| + C^\Pi_{A} \big\| F_1 \big\|  
 + \big\| M \big\| \big\|  \widehat{F}_1 \big\|,   
\notag \\ 
 C^{\Psi, c}_{A} = &   R_1 C^\Pi_{A} + ( C^\Pi_{A} )^2 (\|B^\dagger \| + R_3 ) R_2  . 
\notag 
\end{align} 
\end{lemma}  
\begin{proof}
 See~\ref{sec:lem:A:|q1-q2|and|barx1-barx2|}.  
 \end{proof}

\begin{lemma} 
\label{lem:A:|x1-x2|}
The equilibrium state $x^A$, associated with the perturbed operator $A$, satisfies  
\begin{align} 
& \big( \sup_{t\in\mathcal{T}}\mathbb{E} \| x^A(t) \|^2 \big)^{\frac{1}{2}} \leq C^x_{A}, 
\notag \\ 
& \big( \sup_{t\in\mathcal{T}} \mathbb{E} \| ( x^A - x^\dagger )(t) \|^2 \big)^{\frac{1}{2}}   
\leq C^x_{A, A^\dagger} 
\big\|A - A^\dagger \big\| , 
\notag 
\end{align} 
where 
\allowdisplaybreaks
\begin{align} 
 C^x_{A} = & \big\{  
 3(M^{A}_T)^2 \mathbb{E}\|\xi\|^2 
 + 3 T (M_T^{A})^2 
   \big[ ( \| B^\dagger \|^2 + \| E \|^2 ) \big(  \| B^\dagger \| C^q_{A} 
+ R_2 ( \|F_2^\dagger \|  C^{\bar{x}}_{A} + \|\sigma\| ) C^\Pi_{A } \big)  \big] 
\notag \\ 
&  \hspace{4cm} + 3T(M^{A}_T)^2 ( \| F_1 \|^2 + \| F_2^\dagger \|^2 ) ( C^{\bar{x}}_{A} )^2 
 + 3T(M^{A}_T)^2 \| \sigma \|^2 
 \big\}^{1/2}  \notag \\ 
& \times \exp\big\{ (M^{A}_T)^2 \big[ ( \| B^\dagger \|^2 + \| E \|^2 ) \big( (\|B^\dagger \| + R_3 ) C^\Pi_{A} \big)^2 + \| D \|^2 \big] 3T/2   
\big\}  
 \notag , \notag \\ 
 %%%%%%%%
 C^x_{A, A^\dagger} = &   M^{A, A^\dagger}_T   \big\{  \mathbb{E}\|\xi\|^2 
+ T( C^{\Xi_1^\dagger} +  C^{\Xi_2^\dagger }  )  
\notag \\ 
&  + T  
 \big[ ( C^{\Xi_1, \Pi}_{A, A^\dagger} + C^{\Xi_2, \Pi}_{A, A^\dagger} ) C^\Pi_{A, A^\dagger}  
+ ( C^{\Xi_1, \bar{x}}_{A, A^\dagger} + C^{\Xi_2, \bar{x}}_{A, A^\dagger} ) C^{\bar{x}}_{A, A^\dagger}  
 + ( C^{\Xi_1, q}_{A, A^\dagger } + C^{\Xi_2, q}_{A, A^\dagger } ) C^q_{A, A^\dagger} \big] 
\big\}^{1/2}  
\notag \\ 
& \times \exp \big\{ T ( M^{A}_T )^2 ( C^{\Xi_1, x}_{A, A^\dagger} + C^{\Xi_2, x}_{A, A^\dagger} ) /2  \big\} ,  
\notag \\ 
%%%%% 
C^{\Xi_1^\dagger } = &   3 \| B^\dagger \|^2 (\| B^\dagger \|+R_3)^2 (C^{\Pi^\dagger })^2 (C^{x^\dagger})^2  
  + 3 \| F_1 \|^2 (C^{\bar{x}^\dagger})^2 
 \notag \\ 
 & + 3 \| B^\dagger \|^2 \big[ 2 \| B^\dagger \|^2 (C^{q^\dagger})^2 + 4 (R_2)^2 ( \| F_2^\dagger \|^2 (C^{\bar{x}^\dagger })^2 + \| \sigma \|^2  ) ( C^{\Pi^\dagger} )^2 \big]  , 
 \notag \\ 
 %%%%%%%% 
 C^{\Xi_2^\dagger } = & 4\big[ 2\| D \|^2 + 2\| E\|^2 (\| B^\dagger \| + R_3)^2(C^{\Pi^\dagger})^2 \big] (C^{x^\dagger})^2 
 + 4 \| F_2^\dagger \|^2 (C^{\bar{x}^\dagger })^2 + 4 \| \sigma\|^2
 \notag \\ 
 & + 4 \| E \|^2 \big[ 2 \| B^\dagger \|^2 (C^{q^\dagger})^2 + 4 (R_2)^2 ( \| F_2^\dagger \|^2 (C^{\bar{x}^\dagger})^2 + \| \sigma \|^2  ) ( C^{\Pi^\dagger} )^2 \big] . 
 \notag \\ 
%\end{align} 
%\begin{align} 
 C^{\Xi_1, x}_{A, A^\dagger} = & 4\big[  \| B^\dagger \| (\| B^\dagger \| + R_3) C^{\Pi^\dagger} \big]^2,\notag\\ 
C^{\Xi_2, x}_{A, A^\dagger} =& 4 \big[ \| E \| (\| B^\dagger \| + R_3) C^{\Pi^\dagger} \big]^2 , 
\notag \\ 
 C^{\Xi_1, \Pi}_{A, A^\dagger} = & 4 \| B^\dagger \|^2 (C^{x^\dagger})^2 \big[ 2 ( R_5 )^2 (\| B^\dagger \| + R_3)^2 
  (C^{\Pi^\dagger})^2 
+ 2 (\| B^\dagger \| + R_3 )^2 \big] , 
\notag \\ 
 C^{\Xi_2, \Pi}_{A, A^\dagger} = & 4 \| E \|^2 (C^{x^\dagger})^2 \big[ 2 ( R_5 )^2 (\| B^\dagger \| + R_3)^2 (C^{\Pi^\dagger})^2 
+ 2 (\| B^\dagger \| + R_3 )^2 \big] , 
\notag \\ 
 C^{\Xi_1, \bar{x}}_{A, A^\dagger} = & 4\big[ 2 (R_2)^2 \| F_2^\dagger \|^2 + 2 \| F_1 \|^2 \big] ,\notag \\
C^{\Xi_2, \bar{x}}_{A, A^\dagger} =& 4 \big[ 2 (R_2)^2 \| F_2^\dagger \|^2 +  2 \| F_2^\dagger \|^2 \big] , 
\notag \\ 
 C^{\Xi_1, q}_{A, A^\dagger} 
= & 4 \|B^\dagger \|^4,\\
 C^{\Xi_2, q}_{A, A^\dagger} 
=& 4 \| E \|^2 \|B^\dagger \|^2. 
\notag 
\end{align} 
\end{lemma} 

\begin{proof} 
The bound $C^x_A$ for $x^A$ follows from Lemma~\ref{lem:E|x|<=Cx} with $A^\dagger$ replaced by $A$. 

Denote 
\begin{align} 
& \Xi_1^A(t) = B^\dagger ( (K^A)^{-1} L^A )(T-r) x^A(r) + B^\dagger \tau^A(r) - F_1 \bar{x}^A(r) , \notag \allowdisplaybreaks\\ 
& \Xi_2^A(t) =  \big[ D - E ( (K^A)^{-1} L^A )(T-t) \big] x^A(r) - E \tau^A(t) + F_2^\dagger \bar{x}^A(t) + \sigma,   
\notag 
\end{align} 
where $ \tau^A(t) = (K^A(T-r))^{-1} \big[ B^\star q^A(T-t) + \Gamma_2 \big( (F_2 \bar{x}(t) + \sigma )^\star \Pi^A(T-t) \big) \big]$. By~\eqref{x_MFG_eqm}, we have 
\begin{align} 
  (x^A - x^\dagger )(t) 
 \leq & 
  (S^A - S^\dagger )(t)  \xi  
  + \int_0^t (S^A - S^\dagger )(t-r) \Xi_1^A(r) dr  
  + \int_0^t S^\dagger(t-r) ( \Xi_1^A - \Xi_1^\dagger )(r) dr
 \notag \allowdisplaybreaks\\ 
 &  + \int_0^t (S^A - S^\dagger )(t-r) \Xi_2^A(r) d W(r)  
  + \int_0^t S^\dagger(t-r) ( \Xi_2^A - \Xi_2^\dagger )(r) d W(r) . 
 \notag 
\end{align}

Since 
\begin{align} 
( \Xi_1^A - \Xi_1^\dagger)(r) = & 
 B^\dagger \big(  (K^A)^{-1} L^A - (K^\dagger)^{-1}L^\dagger  \big)(T-r) x^A(r) 
 + B^\dagger (K^\dagger)^{-1}L^\dagger (x^A - x^\dagger)(r) \notag \allowdisplaybreaks\\ 
 & + B^\dagger ( \tau^A - \tau^\dagger) (r) - F_1 (\bar{x} - \bar{x}^\dagger ) , 
 \notag \allowdisplaybreaks\\ 
 %%%%%
 ( \Xi_2^A - \Xi_2^\dagger )(r) = & 
 - E \big( (K^A)^{-1} L^A - (K^\dagger)^{-1}L^\dagger   \big)(T-r) x^A(r) 
 + F_2^\dagger ( \overline{x}^A - \overline{x}^\dagger )(r) 
 \notag \allowdisplaybreaks\\ 
 & + [D - E ( K^\dagger )^{-1} L^\dagger (T-r) ] 
  ( x^A - x^\dagger )(r)  
  - E(\tau^A - \tau^\dagger )(r) , 
 \notag 
\end{align} 
where $\tau(t) = (K(T-r))^{-1} \big[ B^\star q(T-t) + \Gamma_2 \big( (F_2 \bar{x}(t) + \sigma )^\star \Pi(T-t) \big) \big]$, for $i=1, 2$, we have    
\begin{align} 
 \mathbb{E} \| ( \Xi_i^A - \Xi_i^\dagger )(t) \|^2  
\leq & C^{\Xi_i, x}_{A, A^\dagger} \mathbb{E} \| (x^A - x^\dagger)(t) \|^2
+ C^{\Xi_i, \Pi}_{A, A^\dagger} \sup_{r\in \mathcal{T}}\| ( \Pi^A - \Pi^\dagger )(r)  \|^2  
\notag \allowdisplaybreaks\\ 
& 
+ C^{\Xi_i, \bar{x}}_{A, A^\dagger} \| \bar{x}^A - \bar{x}^\dagger \|_{C(\mathcal{T}; H)}^2 
+ C^{\Xi_i, q}_{A, A^\dagger } \| q^A - q^\dagger \|_{C(\mathcal{T}; H)}^2
\notag \allowdisplaybreaks\\ 
%%%%% 
\leq & C^{\Xi_i, x}_{A, A^\dagger} \mathbb{E} \| ( x^A - x^\dagger)(t) \|^2 
+ \big[ C^{\Xi_i, \Pi}_{A, A^\dagger} C^\Pi_{A, A^\dagger}  
+ C^{\Xi_i, \bar{x}}_{A, A^\dagger} C^{\bar{x}}_{A, A^\dagger}  
+ C^{\Xi_i, q}_{A, A^\dagger } C^q_{A, A^\dagger} \big] \| A - A^\dagger \|^2 . 
\notag 
\end{align} 

We also have that 
\begin{align} 
 \mathbb{E} \| \Xi_1^A(r) \|^2 
 \leq & 3 \| B^\dagger ( (K^A)^{-1} L^A )(T-r) \|^2 \mathbb{E} \|x^A(r)\|^2 + 3 \| B^\dagger \|^2 \| \tau^A(r) \|^2 
 + 3 \| F_1 \|^2 \| \bar{x}^A(r) \|^2 
 \notag \allowdisplaybreaks\\ 
 \leq & 3 \| B^\dagger \|^2 (\| B^\dagger \|+R_3)^2 (C^\Pi_A)^2 (C^x_A)^2  
  + 3 \| F_1 \|^2 (C^{\bar{x}}_A)^2 
 \notag \allowdisplaybreaks\\ 
 & + 3 \| B^\dagger \|^2 \big[ 2 \| B^\dagger \|^2 (C^q_A )^2 + 4 (R_2)^2 ( \| F_2^\dagger \|^2 (C^{\bar{x}}_A )^2 + \| \sigma \|^2  ) ( C^\Pi_A )^2 \big]  
 \notag \allowdisplaybreaks\\ 
 \eqdef & C^{\Xi_1}_A , 
 \notag \allowdisplaybreaks\\ 
%\end{align} 
%\begin{align} 
 \mathbb{E}\|\Xi_2^A(r)\|^2 \leq & 
 4\big[ 2\| D \|^2 + 2\| E\|^2 (\| B^\dagger \| + R_3)^2(C^\Pi_A )^2 \big] (C^x_A)^2 
 + 4 \| F_2^\dagger \|^2 (C^{\bar{x}}_A )^2 + 4 \| \sigma\|^2
 \notag \allowdisplaybreaks\\ 
 & + 4 \| E \|^2 \big[ 2 \| B^\dagger \|^2 (C^q_A )^2 + 4 (R_2)^2 ( \| F_2^\dagger \|^2 (C^{\bar{x}}_A )^2 + \| \sigma \|^2  ) ( C^\Pi_A )^2 \big] 
 \notag \allowdisplaybreaks\\ 
 \eqdef & C^{\Xi_2}_A . 
 \notag 
\end{align} 
It follows from Jensen's inequality that 
\begin{align} 
 \mathbb{E} \big\|(x^A - x^\dagger)(t)\big\|^2  
\leq & \| S^A - S^\dagger \|^2 \mathbb{E}\|\xi\|^2 + \int_0^t \| (S^A - S^\dagger)(t-r) \|^2 
( \mathbb{E}\| \Xi_1^A(r) \|^2 + \mathbb{E} \| \Xi_2^A(r) \|^2 ) d r 
\notag \allowdisplaybreaks\\ 
& + \int_0^t \| S^\dagger(t-r) \|^2 \big\{ \mathbb{E} \| (\Xi_1^A - \Xi_1^\dagger)(r) \|^2 
+ \mathbb{E} \| (\Xi_2^A - \Xi_2^\dagger )(r) \|^2 \big\} dr
\notag \allowdisplaybreaks\\ 
%%%%%%% 
\leq & ( M^{A, A^\dagger}_T )^2 \| A - A^\dagger \|^2 \mathbb{E}\|\xi\|^2 
+ T ( M^{A, A^\dagger}_T )^2 \| A - A^\dagger \|^2  ( C^{\Xi_1}_A + C^{\Xi_2}_A )  
\notag \allowdisplaybreaks\\ 
& + ( M^{A^\dagger}_T )^2 \int_0^t ( C^{\Xi_1, x}_{A, A^\dagger} + C^{\Xi_2, x}_{A, A^\dagger} ) 
 \mathbb{E} \big| (x^A - x^\dagger)(r) \big|^2 dr 
\notag \allowdisplaybreaks\\
& + T ( M^{A, A^\dagger}_T )^2 
 \big[ ( C^{\Xi_1, \Pi}_{A, A^\dagger} + C^{\Xi_2, \Pi}_{A, A^\dagger} ) C^\Pi_{A, A^\dagger}  
+ ( C^{\Xi_1, \bar{x}}_{A, A^\dagger} + C^{\Xi_2, \bar{x}}_{A, A^\dagger} ) C^{\bar{x}}_{A, A^\dagger} 
\notag \\ 
& + ( C^{\Xi_1, q}_{A, A^\dagger } + C^{\Xi_2, q}_{A, A^\dagger } ) C^q_{A, A^\dagger} \big] \| A - A^\dagger \|^2 . 
\notag 
\end{align} 
By the Gr\"{o}nwall's inequality, we have 
\begin{align} 
  \mathbb{E} \big\|(x^A - x^\dagger)(t) \big\|^2 
  \leq & ( M^{A, A^\dagger}_T )^2  \big\{  \mathbb{E}\|\xi\|^2 
+ T ( (C^{\Xi_1}_A)^2 + (C^{\Xi_2}_A )^2 )  
\notag \allowdisplaybreaks\\ 
& + T  
 \big[ ( C^{\Xi_1, \Pi}_{A, A^\dagger} + C^{\Xi_2, \Pi}_{A, A^\dagger} ) C^\Pi_{A, A^\dagger}  
+ ( C^{\Xi_1, \bar{x}}_{A, A^\dagger} + C^{\Xi_2, \bar{x}}_{A, A^\dagger} ) C^{\bar{x}}_{A, A^\dagger} 
\notag \allowdisplaybreaks\\ 
& + ( C^{\Xi_1, q}_{A, A^\dagger } + C^{\Xi_2, q}_{A, A^\dagger } ) C^q_{A, A^\dagger} \big] 
\big\} 
\| A - A^\dagger \|^2 
 \exp \big\{ T ( M^{A^\dagger}_T )^2 ( C^{\Xi_1, x}_{A, A^\dagger} + C^{\Xi_2, x}_{A, A^\dagger} )  \big\} . 
  \notag 
\end{align} 

\end{proof}

\begin{proposition} 
 The equilibrium strategy  $u^A$, associated with the perturbed operator $A$, satisfies  
\begin{align} 
& \big( \sup_{t\in \mathcal{T}} \mathbb{E} \big\| u^A(t) - u^\dagger (t) \big\|^2 \big)^{\frac{1}{2}}  
\leq C^u_{A, A^\dagger} \big\| A - A^\dagger \big\|, 
\notag 
\end{align} 
where 
\begin{align} 
 C^u_{A, A^\dagger } 
 = & \sqrt{5}
 \Big\{ 
 ( C^\Pi_{A, A^\dagger} )^2 
  \big[ ( \| B^\dagger \| + R_3 )^2 ( 1+ R_5 C^\Pi_{A} )^2    
 ( C^x_A )^2
 + ( R_2 )^2 ( \| F_2^\dagger \|   
 C^{\bar{x}}_{A} + \| \sigma \| )^2 
 \big] 
 \notag \allowdisplaybreaks\\ 
 & + ( C^x_{A, A^\dagger} )^2 
  (\| B^\dagger \| + R_3)^2 ( C^\Pi_{A} )^2 
  + ( C^{\bar{x}}_{A, A^\dagger} )^2 
  ( R_2 )^2 \|F_2^\dagger \|^2 ( C^\Pi_A )^2   
 + \|B^\dagger \|^2 ( C^q_{A, A^\dagger} )^2 
 \Big\}^{\frac{1}{2}} . 
 \notag 
\end{align} 
\end{proposition} 

\begin{proof} 
By~\eqref{u_MFG_eqm}, the equilibrium strategy for a representative agent corresponding to $A$ is given by 
\begin{align} 
 & u^A(t) =  - (K^A )^{-1}(T-t) 
 \big[ L^A(T-t) x^A(t) 
  + \Gamma_2 \big( ( F_2^\dagger \bar{x}^A(t) + \sigma )^\star \Pi^A(T-t) \big) 
   + (B^\dagger)^\star q^A(T-t) \big], 
\notag \allowdisplaybreaks\\ 
& K^A (t) = I + \Delta_3(\Pi^A (t)), \quad L^A (t) = (B^\dagger)^\star \Pi^A(t) + \Delta_1(\Pi^A (t)) .  
\notag 
\end{align} 

Since 
\begin{align} 
   ( u^A - u^\dagger )(t) 
 = &  ( (K^A)^{-1} L^A - (K^\dagger)^{-1} L^\dagger )(T-t)   x^A(t) 
 + (K^\dagger )^{-1} L^\dagger (T-t)    ( x^A - x^\dagger)(t)  
 \notag \allowdisplaybreaks\\
 & + \Gamma_2 \big(  F_2^\dagger   ( \bar{x}^A - \bar{x}^\dagger )(t) \Pi^A (t)  
 + ( F_2^\dagger \bar{x}^A(t) + \sigma )   (\Pi^A - \Pi^\dagger )(t)  \big)  
  +  B^\dagger  ( q^A - q^\dagger )(t) ,  
 \notag  
\end{align} 
we have that 
\begin{align} 
 & \mathbb{E} \| ( u^A - u^\dagger )(t) \|^2 \notag \allowdisplaybreaks\\ 
 \leq & 5 \| ( (K^A)^{-1} L^A - (K^\dagger)^{-1} L^\dagger )(T-t) \|^2  \mathbb{E} \| x^A(t) \|^2 
 + 5 \| (K^\dagger )^{-1} L^\dagger (T-t) \|^2  \mathbb{E} \| ( x^A - x^\dagger)(t) \|^2 
 \notag \allowdisplaybreaks\\
 & + 5 ( R_2 )^2 \big( \| F_2^\dagger \|^2  \| \bar{x}^A - \bar{x}^\dagger \|_{C(\mathcal{T}; H)}^2 \| \Pi^A (t) \|^2 
 + 5 \| F_2^\dagger \bar{x}^A + \sigma \|^2  \| (\Pi^A - \Pi^\dagger )(t) \|^2 \big)  
 \notag \allowdisplaybreaks\\ 
 &  + 5 \| B^\dagger \|^2  \|q^A - q^\dagger \|_{C(\mathcal{T}; H)}^2
 \notag \allowdisplaybreaks\\
 %%%% 
 \leq & 
  5 ( \| B^\dagger \| + R_3 )^2 ( 1+ R_5 C^\Pi_{A} )^2   \| (\Pi^A - \Pi^\dagger )(t) \|^2  \mathbb{E} \|x^A(t)\|^2 
  + 5 ( \|B^\dagger \| + R_3  )^2 ( C^\Pi_{A} )^2 
   \mathbb{E} \| (x^A - x^\dagger )(t) \|^2 
 \notag \allowdisplaybreaks\\ 
 & + 5 ( R_2 )^2 \big( \| F_2^\dagger \|^2  \| \bar{x}^A - \bar{x}^\dagger \|_{C(\mathcal{T}; H)}^2 ( C^\Pi_A )^2   
 + ( \| F_2^\dagger \|  \| \bar{x}^A \| + \|\sigma\| )^2  \| (\Pi^A - \Pi^\dagger )(t) \|^2 \big)   
 \notag \allowdisplaybreaks\\ 
 &  + 5 \|B^\dagger\|^2  \|q^A - q^\dagger \|_{C(\mathcal{T}; H)}^2 .  
 \notag \allowdisplaybreaks\\ 
 %%%% 
 \leq & 
 5 ( \| B^\dagger \| + R_3 )^2 ( 1+ R_5 C^\Pi_{A^\dagger } )^2  ( C^\Pi_{A, A^\dagger} )^2 \| A - A^\dagger \|^2  ( C^x_A )^2  
  + 5  ( \|B^\dagger \| + R_3  )^2  ( C^\Pi_{A^\dagger } )^2 ( C^x_{A, A^\dagger } )^2 \| A - A^\dagger \|^2 
 \notag \allowdisplaybreaks\\ 
 & + 5 ( R_2 )^2 \big[ \| F_2^\dagger \|^2 ( C^{\bar{x}}_{A, A^\dagger } )^2 ( C^\Pi_A )^2  
 + 5 ( \| F_2^\dagger \|   \| \bar{x}^\dagger \| + \|\sigma\| )^2 ( C^\Pi_{A, A^\dagger} )^2 \big] \| A - A^\dagger \|^2    
 \notag \allowdisplaybreaks\\ 
 &  + 5 \|B^\dagger \|^2 ( C^q_{A, A^\dagger } )^2 \| A - A^\dagger \|^2 .  
 \notag 
\end{align} 
The desired bound for $\big( \sup_{t\in \mathcal{T}} \mathbb{E} \big\| u^A(t) - u^\dagger (t) \big\|^2 \big)^{\frac{1}{2}}$ then follows. 
\end{proof} 

\subsection{Stability of the equilibrium with respect to operator \texorpdfstring{$B$}{B}}
\label{s:StabilityB} 

In this subsection, we perturb the parameter $B^\dagger$ to $B$ and denote by $(\Pi^B, \bar{x}^B, q^B, u^B, x^B)$ the solution to the MFG system, given by \eqref{u_MFG_eqm}-\eqref{x_MFG_eqm}, corresponding to the set of rules $(A^\dagger$, $B$, $D$, $E$, $F_1$, $F_2^\dagger$, $\sigma$, $M$, $\widehat{F}_1$, $\widehat{F}_2$, $G)$. 

\begin{lemma}  
\label{lem:B:|Pi1-Pi2|}
 For all $t\in \mathcal{T}$, the solution $\Pi^B$ of the operator differential Riccati equation, associated with the perturbed operator $B$, satisfies
\begin{align} 
& \big\| \Pi^B(t) \big\| 
\leq C^\Pi_{B} , 
\notag \allowdisplaybreaks\\ 
& \big\| \Pi^B(t) - \Pi^\dagger(t) \big\| \leq C^{\Pi, 1}_{B, B^\dagger}  \big\| B - B^\dagger \big\| 
  + C^{\Pi, 2}_{B, B^\dagger}  \| B - B^\dagger \|^2, 
 \notag 
\end{align} 
where 
\begin{align} 
 C^\Pi_{B} 
= & 
2 (M^{A^\dagger}_T)^2 \exp \big( 8 T (M^{A^\dagger}_T)^2 \| D \|^2 \mathrm{tr}(Q) \big) 
\big( \| G \| + T \| M \| \big) ,  
\notag \allowdisplaybreaks\\ 
C^{\Pi, 1}_{B, B^\dagger} = & \sqrt{T} 
   \exp\Big\{ \frac{1}{2}  (M^{A^\dagger}_T)^2 (1 + \| D\|^2 )T  \Big\}  
  \sqrt{2} M^{A^\dagger}_T  \exp\big(8T (M^{A^\dagger}_T)^2 \| D\|^2 \mathrm{tr}(Q) \big)    
  \notag \allowdisplaybreaks\\ 
 & \times 
  2 ( \| M \| T + \| G \| )   
  \exp \Big\{ ( \|B^\dagger \| + R_3 + R_5 )T 
 \big[ \sqrt{2} M^{A^\dagger}_T  \exp\big(8T (M^{A^\dagger}_T)^2 \| D\|^2 \mathrm{tr}(Q) \big)  \big] \Big\} ,  
\notag \allowdisplaybreaks\\ 
 C^{\Pi,2}_{B, B^\dagger} = & \sqrt{T}
   \exp\Big\{ \frac{1}{2}  (M^{A^\dagger}_T)^2 (1 + \| D\|^2 )T  \Big\}  
  \sqrt{2} M^{A^\dagger}_T  \exp\big(8T (M^{A^\dagger}_T)^2 \| D\|^2 \mathrm{tr}(Q) \big)    
  \notag \allowdisplaybreaks\\ 
 & \times\Big\{ 
   ( \|B \|  + \|B^\dagger \| + 2 R_3 )    
 T (\|B^\dagger \| + R_3) (C^{\Pi^\dagger})^3 \| B - B^\dagger \|  \Big\}  
 \notag \allowdisplaybreaks\\ 
 & \times \exp \Big\{ ( \|B^\dagger \| + R_3 + R_5 )T 
 \big[ \sqrt{2} M^{A^\dagger}_T  \exp\big(8T (M^{A^\dagger}_T)^2 \| D\|^2 \mathrm{tr}(Q) \big)  \big] \Big\}  . \notag 
\end{align} 
\end{lemma}  
\begin{proof}
See \ref{sec:lem:B:|Pi1-Pi2|}.
\end{proof}
\begin{lemma}  
\label{lem:B:|q1-q2|and|barx1-barx2|}
Suppose that 
\begin{align} 
 &  M^{A^\dagger}_T ( C^{\Psi, \bar{x}}_{B} T +  \| G\| \| \widehat{F}_2 \| ) < 1 , 
\notag \allowdisplaybreaks\\ 
 & ( M^{A^\dagger}_T C^{\Psi, \bar{x}}_{B, B^\dagger} T + M^{A^\dagger}_T \| G \| \| \widehat{F}_2 \| )
 M^{A^\dagger}_T C^{\Phi, q}_{B, B^\dagger} T 
 \exp\big( M^{A^\dagger}_T ( C^{\Phi, \bar{x}}_{B, B^\dagger} + C^{\Psi, q}_{B, B^\dagger} ) T \big) < 1. 
 \notag 
\end{align} 
Then, the offset term $q^B$ and the mean field $\bar{x}^B$, corresponding to the perturbed operator $B$, satisfy   
\begin{align} 
& \big\| \bar{x}^B \big\|_{C(\cal{T}; \cal{H})} 
\leq C^{\bar{x}}_{B} ,
\notag \allowdisplaybreaks\\ 
& \big\| q^B \big\|_{C(\cal{T}; \cal{H})} 
\leq C^q_{B} , 
\notag \allowdisplaybreaks\\ 
& \big\| \bar{x}^B - \bar{x}^\dagger \big\|_{C(\mathcal{T}; H)} \leq 
C^{\bar{x},1}_{B, B^\dagger} \big\| B - B^\dagger \big\| 
+ C^{\bar{x},2}_{B, B^\dagger} \big\| B - B^\dagger \big\|^2 , 
\notag \allowdisplaybreaks\\ 
& \big\| q^B - q^\dagger \big\|_{C(\mathcal{T}; H)} \leq 
C^{q,1}_{ B, B^\dagger } \big\| B - B^\dagger \big\| 
+ C^{q,2}_{ B, B^\dagger } \big\| B - B^\dagger \big\|^2 , 
\notag 
\end{align} 
where 
\allowdisplaybreaks
\begin{align} 
C^{\bar{x}, 1}_{B, B^\dagger} 
 = & \big[ 1 - ( M^{A^\dagger}_T C^{\Psi, \bar{x}}_{B, B^\dagger} T + M^{A^\dagger}_T \| G \| \| \widehat{F}_2 \| )
 M^{A^\dagger}_T C^{\Phi, q}_{B, B^\dagger} T  
 \exp\big( M^{A^\dagger}_T ( C^{\Phi, \bar{x}}_{B, B^\dagger} + C^{\Psi, q}_{B, B^\dagger} ) T \big) \big]^{-1} 
 \notag \allowdisplaybreaks\\ 
 & \times \hspace{-0.03cm} \Big[  M^{A^\dagger}_T \big( C^{\Phi, \Pi}_{B, B^\dagger} C^{\Pi, 1}_{B, B^\dagger}   
  + C^{\Phi,c}_{B, B^\dagger}  
 \big) T \exp\big( M^{A^\dagger }_T C^{\Phi, \bar{x}}_{B, B^\dagger} T \big) 
 \notag \\ 
 & + M^{A^\dagger}_T C^{\Phi, q}_{B, B^\dagger} T 
 M^{A^\dagger}_T \big( C^{\Psi, \Pi}_{B, B^\dagger} C^{\Pi, 1}_{B, B^\dagger}  
  + C^{\Psi, c}_{B, B^\dagger}   \big)T 
 \exp\big( M^{A^\dagger }_T ( C^{\Phi, \bar{x}}_{B, B^\dagger} + C^{\Psi, q}_{B, B^\dagger} ) T \big) 
 \Big] , 
\notag \allowdisplaybreaks\\ 
%\end{align} 
%\begin{align} 
C^{\bar{x}, 2}_{B, B^\dagger} 
 = & \big[ 1 - ( M^{A^\dagger}_T C^{\Psi, \bar{x}}_{B, B^\dagger} T + M^{A^\dagger}_T \| G \| \| \widehat{F}_2 \| )
 M^{A^\dagger}_T C^{\Phi, q}_{B, B^\dagger} T 
 \exp\big( M^{A^\dagger}_T ( C^{\Phi, \bar{x}}_{B, B^\dagger} + C^{\Psi, q}_{B, B^\dagger} ) T \big) \big]^{-1}  
 \notag \allowdisplaybreaks\\ 
 &\times  \hspace{-0.03cm} \Big[  M^{A^\dagger}_T \big(  
    C^{\Phi, \Pi}_{B, B^\dagger} C^{\Pi, 2}_{B, B^\dagger}   
 \big) T \exp\big( M^{A^\dagger}_T C^{\Phi, \bar{x}}_{B, B^\dagger} T \big) 
 \notag \allowdisplaybreaks\\ 
 & + M^{A^\dagger}_T C^{\Phi, q}_{B, B^\dagger} T 
 M^{A^\dagger}_T \big(  
  C^{\Psi, \Pi}_{B, B^\dagger} C^{\Pi, 2}_{B, B^\dagger}    \big)T 
 \exp\big( M^{A^\dagger}_T ( C^{\Phi, \bar{x}}_{B, B^\dagger} + C^{\Psi, q}_{B, B^\dagger} ) T \big) 
 \Big] ,  
\notag \allowdisplaybreaks\\ 
C^{q,1}_{B, B^\dagger} = & M^{A^\dagger}_T \big( C^{\Psi, \Pi}_{B, B^\dagger} C^{\Pi, 1}_{B, B^\dagger}  
  + C^{\Psi, c}_{B, B^\dagger}   \big)T \exp\big( M^{A^\dagger}_T C^{\Psi, q}_{B, B^\dagger} T \big)
\notag \allowdisplaybreaks\\ 
& + ( M^{A^\dagger}_T C^{\Psi, \bar{x}}_{B, B^\dagger} T + M^{A^\dagger}_T \| G \| \| \widehat{F}_2 \| ) 
\exp\big( M^{A^\dagger }_T C^{\Psi, q}_{B, B^\dagger} T \big) 
 C^{\bar{x},1}_{B, B^\dagger}  , 
\notag \allowdisplaybreaks\\ 
C^{q,2}_{B, B^\dagger} = & M^{A^\dagger}_T \big(  
  C^{\Psi, \Pi}_{B, B^\dagger} C^{\Pi, 2}_{B, B^\dagger}    \big)T \exp\big( M^{A^\dagger}_T C^{\Psi, q}_{B, B^\dagger} T \big)
\notag \allowdisplaybreaks\\ 
& 
+ ( M^{A^\dagger}_T C^{\Psi, \bar{x}}_{B, B^\dagger} T + M^{A^\dagger}_T \| G \| \| \widehat{F}_2 \| ) 
\exp\big( M^{A^\dagger}_T C^{\Psi, q}_{B, B^\dagger} T \big) 
  C^{\bar{x},2}_{B, B^\dagger},  
\notag \allowdisplaybreaks\\
 C^{\Phi, \bar{x}}_{B, B^\dagger} =& \| B \| (\|B \| + R_3) C^\Pi_{B}
+ \| B \| R_2 \| F_2^\dagger \| C^{\Pi^\dagger} + \| F_1 \| , 
\notag \allowdisplaybreaks\\ 
 C^{\Phi, q}_{B, B^\dagger} =& \| B \|^2 , 
\notag \allowdisplaybreaks\\ 
 C^{\Phi, \Pi}_{B, B^\dagger } =& \| B \| \big( R_5 + \| B \| + R_3 \big) C^{\bar{x}}_B 
+ \| B \| R_2 ( \|F_2^\dagger \| C^{\bar{x}}_{B} + \|\sigma\| ),
\notag \allowdisplaybreaks\\ 
 C^{\Phi, c}_{B, B^\dagger} =& (\| B^\dagger \| + R_3 ) C^{\Pi^\dagger} C^{\bar{x}^\dagger} 
 + \| B \| C^{\Pi^\dagger} C^{\bar{x}^\dagger } 
+ ( \| B \| + \| B^\dagger \| ) C^{q^\dagger}   
+  R_2\big( (\| F_2^\dagger \| C^{\bar{x}^\dagger} + \|\sigma\|) C^{\Pi^\dagger}  \big),
\notag \allowdisplaybreaks\\
C^{\Psi, \bar{x}}_{B, B^\dagger} = & R_1 \| F_2^\dagger \| C^{\Pi^\dagger} + ( \| B \| + R_3 ) R_2 \| F_2^\dagger \| 
  C^{\Pi^\dagger} 
+ ( C^\Pi_{B} \| F_1 \| + \| M \| \|\widehat{F}_1 \| ),
\notag \allowdisplaybreaks\\ 
 C^{\Psi, q}_{B, B^\dagger } = & (\| B \| + R_3 ) C^\Pi_{B} \| B \|, 
\notag \allowdisplaybreaks\\ 
 C^{\Psi, \Pi}_{B, B^\dagger} = & ( \| B \| + R_3 ) \| B^\dagger\| C^{q^\dagger} 
+ (\| B \| + R_3 ) C^\Pi_{B} R_5 \| B^\dagger \| C^{q^\dagger}  
+ R_1 ( \| F_2^\dagger \| C^{\bar{x}^\dagger} + \|\sigma\| )
\notag \allowdisplaybreaks\\ 
& + ( \| B \| + R_3 ) R_2 \big( \| F_2^\dagger \| C^{\bar{x}^\dagger} + \|\sigma\| \big)  
+ (\| B \| + R_3 ) C^\Pi_{B} R_5 R_2 \big( \| F_2^\dagger \| C^{\bar{x}^\dagger} + \|\sigma\| \big) C^{\Pi^\dagger}  
\notag \allowdisplaybreaks\\ 
& + (\| B \| + R_3) C^\Pi_{B} R_2 ( \| F_2^\dagger \| C^{\bar{x}}_{B} + \|\sigma\| ) 
+ \| F_1 \| C^{\bar{x}^\dagger}  
\notag \allowdisplaybreaks\\ 
C^{\Psi, c}_{B, B^\dagger} = & C^{\Pi^\dagger}  \| B^\dagger \| C^{q^\dagger}  
+ (\| B \| + R_3 ) C^{\Pi^\dagger} C^{q^\dagger}   
 + C^{\Pi^\dagger} R_2 ( \| F_2^\dagger \| C^{\bar{x}^\dagger} + \|\sigma\|  ) C^{\Pi^\dagger},   
\notag \allowdisplaybreaks\\
C^{\bar{x}}_{B} =&  \big[ 1 -  M^{A^\dagger}_T ( C^{\Psi, \bar{x}}_{B} T +  \| G\| \| \widehat{F}_2 \| ) \big]^{-1}  \notag \allowdisplaybreaks\\ 
 &\!\!\times\!\! \big[ M^{A^\dagger}_T (  C^{\Phi, c}_{B} T + |\bar\xi| )  \exp( M^{A^\dagger}_T C^{\Phi, \bar{x} }_{B} T )   
 +  M^{A^\dagger}_T C^{\Phi, q}_{B} T \big(  M^{A^\dagger}_T C^{\Psi, c}_{B} T  \big) \exp\big( M^{A^\dagger}_T ( C^{\Phi, \bar{x}}_{B} 
 + C^{\Psi, q}_{B} ) T \big)  \big] ,  
 \notag \allowdisplaybreaks\\ 
 C^q_{B} 
 =&   M^{A^\dagger}_T C^{\Psi, c}_{B} T  \exp( M^{A^\dagger}_T C^{\Psi, q}_{B} T) 
 + M^{A^\dagger}_T ( C^{\Psi, \bar{x}}_{B} T + \| G \| \| \widehat{F}_2 \| ) \exp( M^{A^\dagger}_T C^{\Psi, q}_{B} T) 
 C^{\bar{x}}_{B} , 
 \notag \allowdisplaybreaks\\ 
 C^{\Phi, \bar{x}}_{B} =&  \| B \| ( \|B \| + R_3 + R_2 \|F_2^\dagger\| ) C^\Pi_{B}    + \|F_1\|  , 
 \notag \allowdisplaybreaks\\ 
  C^{\Phi, q}_{B} =& \| B \|^2 , \notag
 \allowdisplaybreaks\\
  C^{\Phi, c}_{B} =& R_2\| B \|  \|\sigma\| C^{\Pi^\dagger}  ,  
 \notag \allowdisplaybreaks\\ 
C^{\Psi, q}_{B} =& C^\Pi_{B} (\| B \| + R_3 ) 
  \big\| B  \big\| , 
  \notag \allowdisplaybreaks\\ 
 C^{\Psi, \bar{x}}_{B}  
 =& \big[ R_1 C^\Pi_{B} + ( C^\Pi_{B} )^2 (\| B \| + R_3 ) R_2 \big] \big\| F_2^\dagger \big\| + C^\Pi_{B}  \big\| F_1 \big\|  
 + \big\| M \big\|  \big\|  \widehat{F}_1 \big\| ,   
\notag \allowdisplaybreaks\\ 
 C^{\Psi, c}_{B} =&   R_1 C^\Pi_{B} + (C^\Pi_{B})^2 (\|B \| + R_3 ) R_2  . 
\notag 
\end{align} 
\end{lemma} 
\begin{proof}
    See \ref{sec:lem:B:|q1-q2|and|barx1-barx2|}.
\end{proof}
\begin{lemma} 
\label{lem:B:|x1-x2|}
The equilibrium state $x^B$, associated with the perturbed operator $B$, satisfies  
\begin{align} 
& \big( \sup_{t\in\mathcal{T}} \mathbb{E} \|x^B(t)\|^2 \big)^{\frac{1}{2}} \leq C^x_{B}, 
\notag \allowdisplaybreaks\\ 
& \big( \sup_{t\in\mathcal{T}} \mathbb{E} \| ( x^B - x^\dagger )(t) \|^2 \big)^{\frac{1}{2}} 
\leq C^{x,1}_{B, B^\dagger } 
\big\|B - B^\dagger \big\| 
+ C^{x,2}_{B, B^\dagger } 
\big\|B - B^\dagger \big\|^2
, 
\notag 
\end{align} 
where 
\allowdisplaybreaks
\begin{align} 
 C^x_{B} = & \big\{  3(M^{A^\dagger}_T)^2 \mathbb{E}\|\xi\|^2 
 + 3 T (M_T^{A^\dagger})^2  \notag \allowdisplaybreaks\\ 
& \times   \big[ ( \| B \|^2 + \| E \|^2 ) \big(  \| B \| C^q_{B} 
+ R_2 ( \|F_2^\dagger \|  C^{\bar{x}}_{B} + \|\sigma\| ) C^{\Pi^\dagger} \big)  \big] 
\notag \allowdisplaybreaks\\ 
& + 3T(M^{A^\dagger}_T)^2 ( \| F_1 \|^2 + \| F_2^\dagger \|^2 ) ( C^{\bar{x}}_{B} )^2 
 + 3T(M^{A^\dagger}_T)^2 \| \sigma \|^2 
 \big\}^{1/2}   \notag \allowdisplaybreaks\\ 
& \times\exp\big\{ (M^{A^\dagger}_T)^2 \big[ ( \| B \|^2 + \| E \|^2 ) \big( (\|B \| + R_3 ) C^{\Pi^\dagger} \big)^2 + \| D \|^2 \big] 3T/2   
\big\}   , \notag \allowdisplaybreaks\\
%%%%%%%
 C_{B, B^\dagger}^{x,1} 
= & T^{1/2}  M^{A^\dagger}_T 
\sqrt{3} 
\Big[  \big( C^{\Xi_1, \bar{x}}_{B, B^\dagger} + C^{\Xi_2, \bar{x}}_{B, B^\dagger} \big)^{1/2} 
C^{\bar{x}, 1}_{B, B^\dagger} 
+ \big( C^{\Xi_1, q}_{B, B^\dagger } + C^{\Xi_2, q}_{B, B^\dagger } \big)^{1/2} 
 C^{q,1}_{B, B^\dagger} 
\notag \allowdisplaybreaks\\ 
& + \big( C^{\Xi_1, c}_{B, B^\dagger} +  C^{\Xi_2, c}_{B, B^\dagger} \big)^{1/2}  \Big] 
\exp\Big\{  T ( M^{A^\dagger}_T )^2  ( C^{\Xi_1, x}_{B, B^\dagger} + C^{\Xi_2, x}_{B, B^\dagger} )/2  \Big\} 
,   
\notag \allowdisplaybreaks\\ 
%%%%%%% 
 C_{B, B^\dagger}^{x,2} = &  
 T^{1/2}  M^{A^\dagger}_T 
\sqrt{3} 
\Big[  \big( C^{\Xi_1, \bar{x}}_{B, B^\dagger} + C^{\Xi_2, \bar{x}}_{B, B^\dagger} \big)^{1/2} 
C^{\bar{x}, 2}_{B, B^\dagger} 
+ \big( C^{\Xi_1, q}_{B, B^\dagger } + C^{\Xi_2, q}_{B, B^\dagger } \big)^{1/2} 
 C^{q,2}_{B, B^\dagger} 
  \Big] 
  \notag \allowdisplaybreaks\\ 
& \times \exp\Big\{  T ( M^{A^\dagger}_T )^2  ( C^{\Xi_1, x}_{B, B^\dagger} + C^{\Xi_2, x}_{B, B^\dagger} )/2  \Big\}, 
\notag \\
\allowdisplaybreaks
 C^{\Xi_1, x}_{B, B^\dagger} = & 
 5 \big[ (\| B \| + R_3 ) C^{\Pi^\dagger} \big]^2  , 
\notag \allowdisplaybreaks\\ 
%%%
C^{\Xi_1, \Pi}_{B, B^\dagger} = & 
5 \big[
\| B \| R_5 (\| B^\dagger \| + R_3)C^{\Pi^\dagger} C^{x^\dagger}  
+ \| B \| (\| B \| + R_3) C^{x^\dagger} 
+ \| B \| R_5 \| B^\dagger\| C^{q^\dagger}  
\notag \allowdisplaybreaks\\ 
&  + \| B \| R_2 (\| F_2^\dagger \| C^{\bar{x}}_{B} + \| \sigma \| ) \big]^2,
\notag \allowdisplaybreaks\\
 %%%
 C^{\Xi_1, \bar{x} }_{B, B^\dagger} = & 
5 \big[ \| B \| R_2 \| F_2^\dagger \| C^{\Pi^\dagger} + \| F_1 \|  \big]^2 ,   
\notag \allowdisplaybreaks\\ 
%%%
 C^{\Xi_1, q}_{B, B^\dagger} = & 5 \| B \|^4, 
\notag \allowdisplaybreaks\\
%%%
 C^{\Xi_1, c}_{B, B^\dagger} = & 
 5 \big[ (\| B^\dagger \| + R_3)C^{\Pi^\dagger} C^{x^\dagger} + \| B \| R_5 (\| B^\dagger \| + R_3) C^{\Pi^\dagger} C^{x^\dagger}  
 + \| B \|^2 C^{x^\dagger}   
\notag \allowdisplaybreaks\\ 
&   + \| B^\dagger \| C^{q^\dagger} + \| B \| R_5 \| B^\dagger \| C^{q^\dagger} + \| B \| C^{q^\dagger} 
+ R_2 ( \| F_2^\dagger \| C^{\bar{x}^\dagger} + \| \sigma \| ) C^{\Pi^\dagger} \big]^2 
\notag ,\\
C^{\Xi_2, x}_{B, B^\dagger } = & 5 \big[ \| D \| + \| E \| ( \| B^\dagger \| + R_3 ) \big]^2, 
\notag \allowdisplaybreaks\\ 
%%% 
C^{\Xi_2, \Pi }_{B, B^\dagger } = & 5 \big[  \| E \| R_5 ( \| B^\dagger \| + R_3 ) C^{\Pi^\dagger} C^{x^\dagger}  
+ \| E \| (\| B \| + R_3 ) C^{x^\dagger}  
\notag \allowdisplaybreaks\\ 
& + \| E \| R_5 \| B^\dagger \| C^{q^\dagger} + \| E \| R_5 R_2 ( \| F_2^\dagger \| C^{\bar{x}^\dagger} + \| \sigma \| ) 
 C^{\Pi^\dagger}  
+ \| E \| R_2 (\| F_2^\dagger \| C^{\bar{x}^\dagger } + \| \sigma \| ) \big]^2 , 
\notag \allowdisplaybreaks\\ 
%%% 
C^{\Xi_2, \bar{x}}_{B, B^\dagger} = & 5 \big[ \| B \| R_2 \| F_2^\dagger \| C^{\Pi^\dagger} \big]^2 ,  
\notag \allowdisplaybreaks\\ 
C^{\Xi_2, q}_{B, B^\dagger} = & 5 \big( \| E \| \| B \| \big)^2 , 
\notag \allowdisplaybreaks\\ 
C^{\Xi_2, c}_{B, B^\dagger } = & 5 ( \| E \| C^{\Pi^\dagger} C^{x^\dagger} + \| E \| C^{q^\dagger} )^2 . 
\notag 
\end{align} 

\end{lemma} 
\begin{proof}
    See \ref{sec:lem:B:|x1-x2|}.
\end{proof}

\begin{proposition} 
The equilibrium strategy $u^B$, associated with the perturbed operator $B$, satisfies    
\begin{align} 
& \big( \sup_{t\in\mathcal{T}} \mathbb{E} \big\|u^B(t) - u^\dagger(t)\big\|^2 \big)^{\frac{1}{2}} 
\leq C^{u,1}_{B, B^\dagger } \big\| B - B^\dagger \big\| + C^{u,2}_{B, B^\dagger } \big\| B - B^\dagger \big\|^2, 
\notag 
\end{align} 
where 
\begin{align} 
 C^{u,1}_{B, B^\dagger } 
 = & 
 6 \big[  ( \| B^\dagger \| + R_3 ) ( 1+ R_5 C^\Pi_{B} )   C^{x^\dagger} 
  + 6 R_2 ( \| F_2^\dagger \| C^{\bar{x}}_{B} + \|\sigma\| )
 \big] C^{\Pi, 1}_{B, B^\dagger }
 \notag \allowdisplaybreaks\\ 
 &  + 6 ( \|B^\dagger \| + R_3  ) C^\Pi_{B } C^{x,1}_{B, B^\dagger } 
 + 6 R_2  \| F_2^\dagger \| C^{\Pi^\dagger} C^{\bar{x}, 1}_{B, B^\dagger} 
 + 6 C^{q^\dagger} 
  + 6 \| B \| C^{q,1}_{B, B^\dagger },
 \notag \allowdisplaybreaks\\ 
 %%%%%% 
 C^{u,2}_{B, B^\dagger } 
 = & 
 6 \big[  ( \| B^\dagger \| + R_3 ) ( 1+ R_5 C^\Pi_{B} )   C^{x^\dagger} 
  + R_2 ( \| F_2^\dagger \| C^{\bar{x}}_{B} + \|\sigma\| )
 \big] C^{\Pi, 2}_{B, B^\dagger}
 \notag \allowdisplaybreaks\\ 
 &  + 6 ( \|B^\dagger \| + R_3  ) C^\Pi_{B} C^{x,2}_{B, B^\dagger} 
 + 6 R_2  \| F_2^\dagger \| C^{\Pi^\dagger} C^{\bar{x}, 2}_{B, B^\dagger } 
  + 6  \| B \| C^{q,2}_{B, B^\dagger }. 
 \notag 
\end{align} 
\end{proposition} 

\begin{proof} 
Since
\begin{align} 
  (u^B - u^\dagger )(t)  
 = &  ( (K^B)^{-1} L^B - (K^\dagger)^{-1} L^\dagger ) (T-t)  x^B(t)  
  + (K^\dagger)^{-1} L^\dagger (T-t) (x^B - x^\dagger )(t)  
 \notag \\ 
 & + \Gamma_2 \big(  F_2^\dagger  ( \bar{x}^B - \bar{x}^\dagger )(t)  \Pi^\dagger (t)  
 + ( F_2^\dagger \bar{x}^\dagger(t) + \sigma )   (\Pi^B - \Pi^\dagger )(T-t)  \big)  
 \notag \\ 
 & + ( B - B^\dagger )  q^\dagger(T-t)  
 + B   ( q - q^\dagger )(T-t),  
 \notag 
 \end{align} 
we have that 
\begin{align} 
 \mathbb{E}\big\| (u^B - u^\dagger )(t) \big\|^2 
 \leq & 
6 \| ( (K^B)^{-1} L^B - (K^\dagger)^{-1} L^\dagger ) (T-t) \|^2 \mathbb{E}\|x^B(t)\|^2 
\notag \allowdisplaybreaks\\ 
 &  + 6 \| (K^\dagger)^{-1} L^\dagger (T-t) \|^2 \,  
  \mathbb{E} \| (x^B - x^\dagger )(t) \|^2 
 \notag \allowdisplaybreaks\\ 
 & + 6 (R_2)^2 \big( \| F_2^\dagger \|^2  \|\bar{x}^B - \bar{x}^\dagger \|_{C(\mathcal{T}; H)}^2 \| \Pi^\dagger (t) \|^2 
 + \| F_2^\dagger \bar{x}^\dagger + \sigma \|^2  \| (\Pi^B - \Pi^\dagger )(t) \|^2 \big)  
 \notag \allowdisplaybreaks\\ 
 & + 6 \| B - B^\dagger \|^2  \|q^\dagger\|_{C(\mathcal{T}; H)}^2  
 + 6 \|B \|^2  \|q - q^\dagger \|_{C(\mathcal{T}; H)}^2.
 \notag 
\end{align} 
Hence, we have
\begin{align} 
\big( \sup_{t\in\mathcal{T}} 
 \mathbb{E}\big\| (u^B - u^\dagger )(t) \big\|^2 \big)^{\frac{1}{2}}
 \notag 
 \leq&
  6 ( \| B^\dagger \| + R_3 ) ( 1+ R_5 C^\Pi_{B} )   \| (\Pi^B - \Pi^\dagger )(t) \|  
  \big( \sup_{t\in \mathcal{T}} \mathbb{E}\| x^B(t) \|^2 \big)^{\frac{1}{2}}
  \notag \\ 
 & + 6 ( \|B^\dagger \| + R_3 )   C^\Pi_{B}  \, 
  \big( \sup_{t\in \mathcal{T}} 
   \mathbb{E} \| (x^B - x^\dagger )(t) \|^2 \big)^{\frac{1}{2}} 
 \notag \\ 
 & + 6 R_2  \big( \| F_2^\dagger \|  
 \|\bar{x}^B - \bar{x}^\dagger \|_{C(\mathcal{T}; H)} C^{\Pi^\dagger}   
 + ( \| F_2^\dagger \| C^{\bar{x}}_{B } + \|\sigma\| )
 \| (\Pi^B - \Pi^\dagger )(t) \| \big)  
 \notag \\ 
 & + 6 \| B - B^\dagger \|  \|q^\dagger\|_{C(\mathcal{T}; H)}  
 + 6 \|B \|  \| q^B - q^\dagger \|_{C(\mathcal{T}; H)}  . 
 \notag 
\end{align} 
The desired estimate then follows. 
\end{proof}

%%%%%%%%%%%%%%% 
\subsection{Stability of the equilibrium with respect to operator \texorpdfstring{$F_2$}{F2}}
\label{s:StabilityF2}

In this subsection, we perturb the parameter  $F_2^\dagger$ to $F_2$ and denote by $(\Pi^{F_2}, \bar{x}^{F_2}, q^{F_2}, u^{F_2}, x^{F_2})$ the solution to the MFG system, given by \eqref{u_MFG_eqm}-\eqref{x_MFG_eqm}, corresponding to the set of rules $(A^\dagger$, $B^\dagger$, $D$, $E$, $F_1$, $F_2$, $\sigma$, $M$, $\widehat{F}_1$, $\widehat{F}_2$, $G)$. 
From~\eqref{ODE_Pi}, the Riccati operator $\Pi(t)$ is independent of $F_2$. Hence, we have
$\Pi^{F_2}(t) = \Pi^{\dagger}(t),\, \forall t \in \mathcal{T}$.

\begin{lemma}  
\label{lem:F2:|q1-q2|and|barx1-barx2|}
Suppose that 
\begin{align} 
&   M^{A^\dagger}_T ( C^{\Psi, \bar{x}}_{F_2} T 
 +  \| G\| \| \widehat{F}_2 \| ) < 1 ,  \notag \\ 
& \big( M^{A^\dagger}_T C^{\Psi, \bar{x}}_{F_2, F_2^\dagger} T + M^{A^\dagger}_T \| G \| \| \widehat{F}_2 \|   \big)  
M^{A^\dagger}_T C^{\Phi, q}_{F_2, F_2^\dagger } T 
\exp\big[ M^{A^\dagger}_T ( C^{\Phi, \bar{x}}_{F_2, F_2^\dagger} + C^{\Psi, q }_{F_2, F_2^\dagger}  )T \big] < 1 , 
\notag 
\end{align} 
Then, the offset term $q^{F_2}$ and the mean field $\bar{x}^{F_2}$, associated with the perturbed operator $F_2$, satisfy   
\begin{align} 
& \big\| \bar{x}^{F_2} \big\|_{C(\cal{T}; \cal{H})} 
\leq C^{\bar{x}}_{F_2}  , 
\notag \\ 
& \big\| q^{F_2} \big\|_{C(\cal{T}; \cal{H})} 
\leq C^q_{F_2} , 
\notag \\ 
& \big\| \bar{x}^{F_2} - \bar{x}^\dagger \big\|_{C(\mathcal{T}; H)} \leq 
C^{\bar{x}}_{F_2, F_2^\dagger } \big\| F_2 - F_2^\dagger \big\|  , 
\notag \\ 
& \big\| q^{F_2} - q^\dagger \big\|_{C(\mathcal{T}; H)} \leq 
C^{q}_{ F_2, F_2^\dagger } \big\| F_2 - F_2^\dagger \big\|  , 
\notag 
\end{align} 
where 
\begin{align} 
C^{\bar{x}}_{F_2, F_2^\dagger} 
= & 
\big[ 1 - ( M^{A^\dagger}_T C^{\Psi, \bar{x}}_{F_2, F_2^\dagger} T + M^{A^\dagger}_T \| G \| \| \widehat{F}_2 \|   ) 
M^{A^\dagger}_T C^{\Phi, q}_{F_2, F_2^\dagger } T 
\exp\big( M^{A^\dagger}_T ( C^{\Phi, \bar{x}}_{F_2, F_2^\dagger} + C^{\Psi, q }_{F_2, F_2^\dagger}  )T \big) \big]^{-1}   
\notag \allowdisplaybreaks\\ 
& \times \big[ (M^{A^\dagger}_T C^{\Phi, c}_{F_2, F_2^\dagger}  T  ) \exp( M^{A^\dagger}_T C^{\Psi, q}_{F_2, F_2^\dagger} T ) 
\notag \allowdisplaybreaks\\ 
&  
+ M^{A^\dagger}_T C^{\Phi, q}_{F_2, F_2^\dagger} T ( M^{A^\dagger}_T C^{\Psi, c}_{F_2, F_2^\dagger} T ) 
 \exp\big( M^{A^\dagger}_T ( C^{\Phi, \bar{x}}_{F_2, F_2^\dagger} + C^{\Psi, q }_{F_2, F_2^\dagger}  )T \big)
\big]
\notag \\ 
C^q_{F_2, F_2^\dagger} = & ( M^{A^\dagger}_T C^{\Psi, c}_{F_2, F_2^\dagger} T ) 
\exp\big( M^{A^\dagger}_T ( C^{\Phi, \bar{x}}_{F_2, F_2^\dagger} + C^{\Psi, q }_{F_2, F_2^\dagger}  ) T \big)
\notag \allowdisplaybreaks\\ 
& + \big( M^{A^\dagger}_T C^{\Psi, \bar{x}}_{F_2, F_2^\dagger} T+ M^{A^\dagger}_T \| G \| \| \widehat{F}_2 \| \big) 
\exp\big( M^{A^\dagger}_T C^{\Psi, q}_{F_2, F_2^\dagger} T \big) 
C^{\bar{x}}_{F_2, F_2^\dagger} 
\notag \allowdisplaybreaks\\ 
C^{\Phi,\bar{x}}_{F_2, F_2^\dagger} 
 = &  \| B^\dagger \| (\| B^\dagger \| + R_3 ) C^{\Pi^\dagger} + \| F_1 \| 
  + \| B^\dagger \| R_2 \| F_2 \| C^{\Pi^\dagger} , 
\notag \\ 
 C^{\Phi,q}_{F_2, F_2^\dagger} = & \| B^\dagger \|,\notag \allowdisplaybreaks\\ 
C^{\Phi,c }_{F_2, F_2^\dagger} 
= &R_2\| B^\dagger \|  C^{\bar{x}^\dagger } , 
\notag \\ 
 C^{\Psi, \bar{x}}_{F_2, F_2^\dagger} = & R_1 \| F_2 \| C^{\Pi^\dagger } 
+ (\| B^\dagger \| + R_3 ) C^{\Pi^\dagger} R_2 \| F_2 \| C^{\Pi^\dagger} 
+ ( C^{\Pi^\dagger} \| F_1 \| + \| M\| \| \widehat{F}_1 \|  ) , 
\notag \allowdisplaybreaks\\ 
 C^{\Psi, q}_{F_2, F_2^\dagger} = & (\| B^\dagger \|  + R_3 ) C^{\Pi^\dagger} \| B^\dagger \|,\notag \allowdisplaybreaks\\ 
 C^{\Psi, c}_{F_2, F_2^\dagger} =& (\| B^\dagger \| + R_3) C^{\Pi^\dagger } R_2 C^{\bar{x}^\dagger } C^{\Pi^\dagger},  
\notag \allowdisplaybreaks\\ 
C^{\bar{x}}_{F_2} = & \big[ 1 -  M^{A^\dagger}_T ( C^{\Psi, \bar{x}}_{F_2} T 
 +  \| G\| \| \widehat{F}_2 \| ) \big]^{-1}  \big[ M^{A^\dagger}_T (  C^{\Phi, c}_{F_2} T + \|\bar\xi\| ) 
 \exp( M^{A^\dagger}_T C^{\Phi, \bar{x} }_{F_2} T)  
 \notag \\&+  M^{A^\dagger}_T C^{\Phi, q}_{F_2} T \big(  M^{A^\dagger}_T C^{\Psi, c}_{F_2} T  \big) 
  \exp\big( M^{A^\dagger}_T ( C^{\Phi, \bar{x}}_{F_2} + C^{\Psi, q}_{F_2} ) T \big)  \big] ,  
 \notag \allowdisplaybreaks\\ 
C^q_{F_2} 
 = &  M^{A^\dagger}_T C^{\Psi, c}_{F_2} T  \exp( M^{A^\dagger}_T C^{\Psi, q}_{F_2} T) 
 + M^{A^\dagger}_T ( C^{\Psi, \bar{x}}_{F_2} T + \| G \| \| \widehat{F}_2 \| ) 
 \exp( M^{A^\dagger}_T C^{\Psi, q}_{F_2} T) 
 C^{\bar{x}}_{F_2} , 
 \notag \allowdisplaybreaks\\ 
 C^{\Phi, \bar{x}}_{F_2} = &  \| B^\dagger \| ( \|B^\dagger \| + R_3 + R_2 \|F_2 \| ) C^{\Pi^\dagger}   + \|F_1\|  , 
 \notag \allowdisplaybreaks\\ 
  C^{\Phi, q}_{F_2} = & \| B^\dagger \|^2 , 
 \notag \\
  C^{\Phi, c}_{F_2} =& \| B^\dagger \| R_2 \|\sigma\| C^{\Pi^\dagger}  ,  
 \notag \allowdisplaybreaks\\ 
 C^{\Psi, q}_{F_2} = & C^\Pi (\| B^\dagger \| + R_3 ) 
  \big\| B^\dagger \big\| , 
  \notag \allowdisplaybreaks\\ 
 C^{\Psi, \bar{x}}_{F_2} 
 = &\big[ R_1 C^{\Pi^\dagger} + ( C^{\Pi^\dagger} )^2 (\| B^\dagger \| + R_3 ) R_2 \big] \big\| F_2 \big\| 
  + C^{\Pi^\dagger} \big\| F_1 \big\|  
 + \big\| M \big\|  \big\|  \widehat{F}_1 \big\| ,   
\notag \allowdisplaybreaks\\ 
C^{\Psi, c}_{F_2} = &  R_1 C^{\Pi^\dagger} + (C^{\Pi^\dagger} )^2 (\|B^\dagger \| + R_3 ) R_2  . 
\notag 
\end{align} 
\end{lemma} 
\begin{proof}
    See~\ref{sec:lem:F2:|q1-q2|and|barx1-barx2|}.
\end{proof}
\begin{lemma} 
\label{lem:F2:|x1-x2|}
The equilibrium state $x^{F_2}$, associated with the perturbed operator $F_2$, satisfies   
\begin{align} 
& \big( \sup_{t\in \mathcal{T}} \mathbb{E} \|x^{F_2}(t)\|^2 \big)^{\frac{1}{2}} \leq C^x_{F_2}, 
\notag \allowdisplaybreaks\\ 
& \big( \sup_{t\in\mathcal{T}} \mathbb{E} \| ( x^{F_2} - x^\dagger )(t) \|^2 \big)^{\frac{1}{2}} 
\leq C^x_{F_2, F_2^\dagger} 
\big\| F_2 - F_2^\dagger \big\| , 
\notag 
\end{align} 
where 
\begin{align} 
 C^x_{F_2} = & \big\{  3(M^{A^\dagger}_T)^2 \mathbb{E}\|\xi\|^2 
 + 3 T (M_T^{A^\dagger} )^2  
   \big[ ( \| B^\dagger \|^2 + \| E \|^2 ) \big(  \| B^\dagger \| C^q_{F_2^\dagger} 
+ R_2 ( \|F_2\|  C^{\bar{x}}_{F_2 } + \|\sigma\| ) C^{\Pi^\dagger} \big)  \big] 
\notag \allowdisplaybreaks\\ 
& \hspace{0cm} + 3T(M^{A^\dagger}_T)^2 ( \| F_1 \|^2 + \| F_2 \|^2 ) ( C^{\bar{x}}_{F_2} )^2 
 + 3T(M^{A^\dagger}_T)^2 \| \sigma \|^2 
 \big\}^{1/2}   \notag \allowdisplaybreaks\\ 
& \times\exp\big\{ (M^{A^\dagger}_T)^2 \big[ ( \| B^\dagger \|^2 + \| E \|^2 ) \big( (\|B^\dagger \| + R_3 ) C^{\Pi^\dagger} \big)^2 + \| D \|^2 \big] 3T/2   
\big\} ,   
 \notag \allowdisplaybreaks\\ 
 %%%%%%%%%
 C^x_{F_2, F_2^\dagger } = & 
 T^{1/2}  M^{A^\dagger}_T  
 \Big[ ( C^{\Xi_1, q}_{F_2, F_2^\dagger} + C^{\Xi_2, q}_{F_2, F_2^\dagger} ) C^q_{F_2, F_2^\dagger}  
+ ( C^{\Xi_1, \bar{x}}_{F_2, F_2^\dagger} + C^{\Xi_2, \bar{x}}_{F_2, F_2^\dagger} ) C^{\bar{x}}_{F_2, F_2^\dagger} 
+ C^{\Xi_1, c}_{F_2, F_2^\dagger} + C^{\Xi_2, c}_{F_2, F_2^\dagger}
\Big]^{1/2} 
\notag \allowdisplaybreaks\\ 
& \times \exp \big[  T ( M^{A^\dagger}_T )^2 ( C^{\Xi_1, x}_{F_2, F_2^\dagger} + C^{\Xi_2, x}_{F_2, F_2^\dagger} )/2 \big] ,  
\notag \allowdisplaybreaks\\ 
 C^{\Xi_1, x}_{F_2, F_2^\dagger} = & \| B^\dagger \|^2 ( \| B^\dagger \| + R_3 )^2 (C^{\Pi^\dagger})^2 , 
\notag \allowdisplaybreaks\\ 
 C^{\Xi_1, \bar{x} }_{F_2, F_2^\dagger} = & \| B^\dagger \|^2 ( R_2 )^2 \| F_2 \|^2 ,
\notag \allowdisplaybreaks\\ 
 C^{\Xi_1, q }_{F_2, F_2^\dagger} = & \| B^\dagger \|^4 ,
\notag \allowdisplaybreaks\\ 
 C^{\Xi_1, c}_{F_2, F_2^\dagger} = & \|B^\dagger \|^2 \big( R_2 C^{\bar{x}^\dagger} C^{\Pi^\dagger} \big)^2 ,
\notag \allowdisplaybreaks\\ 
  C^{\Xi_2, x}_{F_2, F_2^\dagger} =& \big[ \| D \| + \| E \| ( \| B^\dagger \| + R_3) (C^{\Pi^\dagger})  \big]^2 ,
 \notag \allowdisplaybreaks\\ 
  C^{\Xi_2, \bar{x}}_{F_2, F_2^\dagger } =& \|F_2\|^2 + \big( R_2\|E \| \| F_2^\dagger \| C^{\Pi^\dagger} \big)^2 ,
 \notag \allowdisplaybreaks\\ 
  C^{\Xi_2, q}_{F_2, F_2^\dagger } =& \| E \|^2 \|B^\dagger \|^2,
 \notag \allowdisplaybreaks\\ 
  C^{\Xi_2, c}_{F_2, F_2^\dagger } =& \| E \|^2 \big( R_2 C^{\bar{x}^\dagger} C^{\Pi^\dagger} \big)^2. 
 \notag 
\end{align} 

\end{lemma} 
\begin{proof}
    See Section~\ref{sec:lem:F2:|x1-x2|}.
\end{proof}

\begin{proposition} 
The equilibrium strategy $u^{F_2}$, associated with the perturbed operator $F_2$, satisfies    
\begin{equation}
 \mathbb{E} \|u^{F_2}(t) - u^\dagger(t)\| 
\leq C^u_{F_2, F_2^\dagger} \big\| F_2 - F_2^\dagger \big\|, 
\notag 
\end{equation} 
where 
\begin{align} 
 C^u_{F_2, F_2^\dagger } 
 = & 
 \sqrt{6}
 \Big\{ 
  (\| B^\dagger \| + R_3 ) C^x_{F_2, F_2^\dagger}  
  +  R_2 C^{\Pi^\dagger} 
  \big(  (C^{\bar{x}^\dagger})^2 
  + \| F_2 \|^2 
   ( C^{\bar{x}}_{F_2, F_2^\dagger} )^2  \big) 
   +  \| B^\dagger \|
 C^q_{F_2, F_2^\dagger} 
 \Big\}^{\frac{1}{2}} . 
 \notag 
\end{align} 
\end{proposition} 

\begin{proof} 
From~\eqref{u_MFG_eqm}, the equilibrium strategy corresponding to the operator $F_2$ is given by 
\begin{align} 
 & u^{F_2}(t) =  - (K^\dagger)^{-1}(T-t) 
   \big[ L^\dagger(T-t) x^{F_2}(t) 
  + \Gamma_2 \big( ( F_2 \bar{x}^{F_2}(t) + \sigma )^\star \Pi^\dagger(T-t) \big) 
   + B^{\dagger\star}  q^{F_2}(T-t) \big], 
\notag \\ 
& K^\dagger (t) = I + \Delta_3(\Pi^\dagger (t)), \quad L^\dagger (t) = B^{\dagger\star} \Pi^\dagger(t) + \Delta_1(\Pi^\dagger (t)) .   
\notag 
\end{align} 
Since 
\begin{align} 
  (u^{F_2} - u^\dagger )(t)   
 = &   (K^\dagger)^{-1} L^\dagger (T-t)  
  ( x^{F_2} - x^\dagger )(t)  
 +  (K^\dagger)^{-1}(T-t)   \Gamma_2 \big( F_2 \bar{x}^{F_2}(t) - F_2^\dagger \bar{x}^\dagger(t) \big)  \Pi^\dagger(T-t) 
 \notag \\ 
 & +  (K^\dagger)^{-1}(T-t)   B^{\dagger\star}  
 (q^{F_2} - q^\dagger)(T-t) , 
 \notag 
\end{align} 
we have  
\begin{align} 
 \mathbb{E}\| (u^{F_2} - u^\dagger )(t) \|^2  
 \notag 
 &\leq  3 \| (K^\dagger)^{-1} L^\dagger (T-t) \|^2 \, 
  \mathbb{E} \| ( x^{F_2} - x^\dagger )(t) \|^2 
 \notag \\ 
 &\,\,\,\, + 3 \| (K^\dagger)^{-1}(T-t) \|^2 \| \Gamma_2 \|^2 
 \| F_2 \bar{x}^{F_2}(t) - F_2^\dagger \bar{x}^\dagger(t) \|^2 
 \| \Pi^\dagger(T-t) \|^2
 \notag \\ 
 &\,\,\,\, + 3  \| (K^\dagger)^{-1}(T-t) \|^2 \| B^\dagger \|^2 \| (q^{F_2} - q^\dagger)(T-t)\|^2
 \notag \\ 
 %%%%%% 
 &\leq  6 (\| B^\dagger \| + R_3 )^2 
  ( C^x_{F_2, F_2^\dagger} )^2 \| F_2 - F_2^\dagger \|^2 + 6 ( R_2 )^2 ( C^{\Pi^\dagger} )^2 
   \big( \|F_2 - F_2^\dagger \|^2 ( C^{\bar{x}^\dagger} )^2 
    \\&\,\,\,+ \| F_2 \|^2 
     \| F_2 - F_2^\dagger \|^2 \big) ( C^{\bar{x}}_{F_2, F_2^\dagger} )^2 
  + 6 \|B^\dagger \|^2
 \| F_2 - F_2^\dagger \|^2 (C^q_{F_2, F_2^\dagger} )^2 . 
 \notag 
\end{align} 
\end{proof}

\subsection{{Proof of Proposition~\ref{prop:Solution_Operator_continuous}}}\label{sec:prop:Solution_Operator_continuous}
We may now deduce Proposition~\ref{prop:Solution_Operator_continuous}.
\begin{proof}
Direct consequence of Theorem~\ref{thrm:Stability_R2E}, and the inequalities 
\eqref{eq:3HS_2_DirSum} and \eqref{eq:Sup_to_L2} as well as the fact that the Hilbert-Schmidt norm upper-bounds the operator norm.
\end{proof}

\subsection{{Proof of Proposition~\ref{prop:ExtensionExistence}}}\label{sec:prop:ExtensionExistence}
\begin{proof}
By Theorem~\ref{thrm:Stability_R2E}, there exists a constant $\tilde{L}_{\mathcal{K}}\ge 0$, depending only on $\mathcal{K}$, as given by~\eqref{prop:Solution_Operator_continuous}, such that for every $(A,B,F_2),(\tilde{A},\tilde{B},\tilde{F}_2)\in \mathcal{K}$
\begin{equation}
\label{eq:XuweiDenaEstimate}
       \|
      \mathfrak{R}(A,B,F_2)_t
      -
      \mathfrak{R}(\tilde{A},\tilde{B},\tilde{F}_2)_t
    \|_{\Lt}
\le 
    \tilde{L}_{\mathcal{K}}
    \,
        \Big(
        \sum_{i=1}^3
        \,
            \|A-\tilde{A}\|^2 
        + 
            \|B-\tilde{B}\|^2
        +
            \|F_2 - \tilde{F}_2\|^2
        \Big)^{1/2}
.
\end{equation}
Combining~\eqref{eq:XuweiDenaEstimate} with the estimates in~\eqref{eq:sup_to_L2}
and in~\eqref{eq:op_to_HS}, we deduce that $\mathcal{R}:\HS\to \Lt$ is $L_{\mathcal{K}}\eqdef T\,\tilde{L}_{\mathcal{K}}$-Lipschitz when \textit{restricted} to $\mathcal{K}$.  By the Benyamini-Lindenstrauss Theorem, see e.g.~\cite[Theorem 1.12]{BenyaminiLindenstrauss_2000_NonlinearFunctionalAnalysis} there exists a $L_{\mathcal{K}}$-Lipschitz extension $\mathfrak{R}^{\mathcal{K}}:\HS\to\Lt$ of the restricted rules-to-equilibrium map $\mathcal{R}|_{\mathcal{K}}:\mathcal{K}\to \Lt$; i.e. $\mathcal{R}^{\mathcal{K}}$ satisfies~\eqref{eq:Def_Extension_R2E} and is $L_{\mathcal{K}}$-Lipschitz \textit{globally} on all of $\HS$.
\end{proof}

\section{Proof of Sample-Complexity Estimates and Approximation Guarantees}
\label{s:Proofs_PACLearning}

In this section, we generalize $\HS$ (resp.\ $\Lt$) to an arbitrary infinite-dimensional separable Hilbert space $\mathcal{H}$ (resp.\ $\Lt$), each equipped with an orthonormal basis denoted $(e_i)_{i\in I}$ (resp.\ $(\eta_j)_{j\in J}$). This generalization is justified as all results presented here remain valid in this broader setting. Accordingly, we define the projection operators~\eqref{eq:ProjOp} and embedding operators~\eqref{eq:EmbOp} exactly as before, with $\HS$ and $\Lt$ replaced by $\mathcal{H}$ and $\Lt$, respectively. When necessary, we make the domain and codomain explicit by writing $P^{\mathcal{H}}_i$ (resp.\ $E^{\Lt}_j$) instead of $P_i$ (resp.\ $E_j$). Similarly, the RNO model defined previously now refers to maps from $\mathcal{H}$ to $\Lt$, with the only change being that the encoder and decoder maps $P_i$ and $E_j$ are now defined on $\mathcal{H}$ and $\Lt$ rather than on $\HS$ and $\Lt$.

\subsection{Oracle Inequalities}
We are interested in the out-of-sample performance of a neural operator model $\hat{F}:\mathcal{H}\to \Lt$ is measured by its ability to recover the noiseless target operator $f^{\star}$, as quantified by its \textit{reconstruction quality} (sometimes, abusively called excess risk)
\begin{equation}
\label{eq:risk_excess}
\mathcal{R}(\hat{F})\eqdef \mathbb{E}_{X\sim \mathbb{P}_X}\Big[
            \|
            \hat{F}(X)-
            \underbrace{f^{\star}(X)}_{\text{True Operator}}
        \|_{\Lt}
   \Big]
.
\end{equation}
We denote the sample-based version of the risk, called the \textit{empirical risk}, is defined by
\begin{equation}
\label{eq:empirical_risk}
    \hat{\mathcal{R}}_{\mathcal{S}}(\hat{F})
    \eqdef 
    \frac1{N}\sum_{n=1}^N\, \|\hat{F}(X_n)-Y_n\|_{\Lt}
.
\end{equation}
The next result bounds the true risk in terms of the empirical risk, under the assumption that both the target function and the hypothesis class admit the same worst-case Lipschitz constant. 
\begin{lemma}[A General Transport-Theoretic Oracle-Type Inequality]
\label{lem:excess_decomposition}
Let $N\in \mathbb{N}_+$, $L\ge 0$, and define $\bar{L}\eqdef 2\max\{L,1\}$.
Let $f^{\star}:\mathcal{H}\to \Lt$ be $L$-Lipschitz and $\hat{F}\in \mathcal{F}\subseteq \operatorname{Lip}(\mathcal{H},\Lt|L)$ be an empirical risk minimizer
\begin{equation}
\label{eq:ERM}
        \hat{\mathcal{R}}^S(\hat{F})
           =
        \inf_{\tilde{F}\in \mathcal{F}}\,
        \hat{\mathcal{R}}^S(\tilde{F}).
\end{equation}
For any compact $\mathcal{K}\subset\mathcal{H}$ of positive $\mathbb{P}_X$-measure, the following holds
\[
        \mathbb{P}\biggl(
                           \mathcal{R}(\hat{F})
            \le 
                               \underbrace{
                    \inf_{\tilde{F}\in \mathcal{F}}\,
                    \sup_{x\in \mathcal{K}}
                        \|\tilde{F}(x)-f^{\star}(x)\|_{\Lt}
                }_{\term{t:apprx}: \text{Approximation}}
                +
                \underbrace{
                        \bar{L}\,
                        \mathcal{W}_1\big(
                            \mathbb{P}_X
                        ,
                            \mathbb{P}^N_X
                        \big)
                }_{\term{t:stat}: 
                    \text{Concentration Wasserstein}
                }
        \biggr)
    \ge 
        \mathbb{P}_X(\mathcal{K})^N
    >
        0
.
\]
\end{lemma}
\begin{proof}[{Proof of Lemma~\ref{lem:excess_decomposition}}]
Denote the (random) empirical measure induced by our sample set by $\mathbb{P}^N\eqdef \frac1{N}\sum_{n=1}^N\,\delta_{X_N}$.
For every $\tilde{F}\in \mathcal{F}$, the map $\Lambda_{\tilde{F}}:\mathcal{H}\ni x\mapsto \|\tilde{F}(x)-f^{\star}(x)\|_{\Lt}$ is $\bar{L}\eqdef 2\max\{1,L\}$-Lipschitz since $f^{\star}$ and $\tilde{F}\in \mathcal{F}$ are both $L$-Lipschitz and $\|\cdot\|_{\Lt}$ is $1$-Lipschitz.  In particular, this is the case when taking $\tilde{F}=\hat{F}$.
Therefore, the Kantorovich-Rubinstein duality implies that
\allowdisplaybreaks
\begin{align}
\label{eq:ERM_Applied__PRE}
    \mathcal{R}(\hat{F}) 
% & 
= 
    \mathbb{E}_{X\sim \mathbb{P}_X}\big[
        \|
            \hat{F}(X)-f^{\star}(X)
        \|_{\Lt}
    \big]
% \\
% & 
\le  
        \bar{L}\mathcal{W}_1(\mathbb{P},\mathbb{P}_N)
    +
        \mathbb{E}_{X\sim \mathbb{P}_X^N}\big[
        \|
            \hat{F}(X)-f^{\star}(X)
        \|_{\Lt}
    \big]
.
\end{align}
Now, applying the empirical risk minimization property of $\hat{F}$, in~\eqref{eq:ERM}, we may re-express~\eqref{eq:ERM_Applied__PRE} as
\allowdisplaybreaks
\begin{align}
\label{eq:ERM_Applied}
    \mathcal{R}(\hat{F}) 
 \le  
        \bar{L}\mathcal{W}_1(\mathbb{P},\mathbb{P}_N)
    +
        \inf_{\tilde{F}\in \mathcal{F}}\,
            \frac1{N}\sum_{n=1}^N
        \,
            \|
                \tilde{F}(X_n)
                -
                f^{\star}(X_n)
            \|_{\Lt}
.
\end{align}
Now, since $X_1,\dots,X_N$ are i.i.d.\ with law $\mathbb{P}$ and since $\mathcal{K}$ has positive $\mathbb{P}$-measure then 
\[
        \mathbb{P}\big(\forall n=1,\dots,N\, X_n \in \mathcal{K}\big)
    =
        \mathbb{P}(X_1\in \mathcal{K})
    =
        \mathbb{P}_X(\mathcal{K})^N
    >
        0
.
\]
Thus, with (positive) probability at-least $\mathbb{P}_X(\mathcal{K})^N$, for every $\tilde{F}\in \mathcal{F}$ we have 
\begin{equation}
\label{eq:conditional_uniformationzation_bound}
    \frac1{N}\sum_{n=1}^N
        \|
            \tilde{F}(X_n)
            -
            f^{\star}(X_n)
        \|_{\Lt}
    \le 
        \frac1{N}\sum_{n=1}^N
        \sup_{x\in \mathcal{K}}
        \|
            \tilde{F}(x)
            -
            f^{\star}(x)
        \|_{\Lt}
= 
        \sup_{x\in \mathcal{K}}
        \|
            \tilde{F}(x)
            -
            f^{\star}(x)
        \|_{\Lt}
.
\end{equation}
Taking infimal over $\mathcal{F}$ across~\eqref{eq:conditional_uniformationzation_bound} yields
\begin{equation}
\label{eq:conditional_uniformationzation_bound__infimalVerison}
        \inf_{\tilde{F}\in \mathcal{F}}
        \,
            \frac1{N}\sum_{n=1}^N
                \|
                    \tilde{F}(X_n)
                    -
                    f^{\star}(X_n)
                \|_{\Lt}
    \le 
        \inf_{\tilde{F}\in \mathcal{F}}
        \,
            \sup_{x\in \mathcal{K}}
            \|
                \tilde{F}(x)
                -
                f^{\star}(x)
            \|_{\Lt}
\end{equation}
which, again, holds with probability at-least $\mathbb{P}(\mathcal{K})^N$ (on the draw of $X_1,\dots,X_N$).
Now, merging~\eqref{eq:conditional_uniformationzation_bound__infimalVerison} back into the right-hand side of~\eqref{eq:ERM_Applied} implies that
\allowdisplaybreaks
\begin{align*}
\label{eq:ERM_Applied}
    \mathcal{R}(\hat{F}) 
& \le  
        \bar{L}\mathcal{W}_1(\mathbb{P},\mathbb{P}_N)
    +
        \inf_{\tilde{F}\in \mathcal{F}}\,
            \frac1{N}\sum_{n=1}^N
        \,
            \|
                \tilde{F}(X_n)
                -
                f^{\star}(X_n)
            \|_{\Lt}
\\
& \le 
        \bar{L}\mathcal{W}_1(\mathbb{P},\mathbb{P}_N)
    +
        \inf_{\tilde{F}\in \mathcal{F}}
        \,
            \sup_{x\in \mathcal{K}}
            \|
                \tilde{F}(x)
                -
                f^{\star}(x)
            \|_{\Lt}
\end{align*}
holds with probability at-least $\mathbb{P}(\mathcal{K})^N$.
\end{proof}
\begin{lemma}[Infinite-Dimensional Concentration Inequality]
\label{lem:concentrate}
Under Assumption~\ref{ass:Statistical}, for every $0<\delta \le 1$, the following
\[
        \mathcal{W}_1(\mathbb{P}_X, \mathbb{P}_X^N) 
    \lesssim
        e^{-\sqrt{
            \log(N^{2r})
        }}
    +
        \frac{
        \ln\big(
                \frac{2}{\delta}
            \big) 
        }{N}
        +
        \frac{
            \sqrt{
            \ln\big(
                \frac{2}{\delta}
            \big) 
            }
        }{
            \sqrt{N}
        }
\]
holds with probability at-least $1-\delta$.
\end{lemma}
\begin{proof}[{Proof of Lemma~\ref{lem:concentrate}}]
The rapid decay condition in Assumption~\ref{eq:KLDecomposition} implies both that
\[
        \sum_{i=1}^{\infty}\,
            \sigma_i^2
        \lesssim
            \sum_{i=1}^{\infty}\,
                e^{-r 2i}
            =
            \frac1{e^{2r}-1}
            <\infty
.
\]
This, together with~\cite[Proposition 5.3]{LeiConcentrationUnbounded2020Bernoulli}, the sub-Gaussian decay condition $\sup_{i\in \mathbb{N}_+}\,\|Z_i\|_{\psi_2}<\infty$ implies that $X$ has sub-exponential tails; i.e.\ 
$
    \|X\|_{\psi_1} < \infty
$.
As shown directly below~\cite[Equation (14)]{LeiConcentrationUnbounded2020Bernoulli}, the finiteness of $\|X\|_{\psi_1}$ implies that condition~\cite[Equation (13)]{LeiConcentrationUnbounded2020Bernoulli} is satisfied; namely, there are constant $v,V>0$ satisfying
\begin{equation}
\label{eq:rapid_tail_DecCond}
        2\mathbb{E}[\|X\|^k] 
    \le 
        k!
        v^2
        V^{k-2}
\end{equation}
for all integers $k\ge 2$.  In turn, \eqref{eq:rapid_tail_DecCond} implies that~\cite[Corollary 5.2]{LeiConcentrationUnbounded2020Bernoulli} applies; whence, for all $\eta>0$ we have that 
\begin{equation}
\label{eq:Concentration}
\mathbb{P}\Big( 
        \big\|
        \mathcal{W}_1(\mathbb{P}_X, \mathbb{P}_X^N) - 
            \mathbb{E} \big[
                \mathcal{W}_1(\mathbb{P}_X, \mathbb{P}_X^N) 
            \big]
        \big\|
            \geq 
        \eta 
    \Big) 
\leq 
    2 \exp 
    \Big( 
        - \frac{
            N
            \eta^2
            }{        
                8v^2+ 4 V \eta
        } 
    \Big)
.
\end{equation}
It remains to bound the expected Wasserstein distance $\mathbb{E} \big[
\mathcal{W}_1(\mathbb{P}_X, \mathbb{P}_X^N) \big]$.
Now, the exponential decay assumption $\sigma_i\, e^{-r i}$ imply that~\cite[Proposition 4.4 (2)]{LeiConcentrationUnbounded2020Bernoulli} applies; which, in turn, guarantees that~\cite[Theorem 4.2]{LeiConcentrationUnbounded2020Bernoulli} applies with $\gamma= e^r$. Whence
\begin{equation}
\label{eq:mean_bound}
        \mathbb{E} \big[
            \mathcal{W}_1(\mathbb{P}_X, \mathbb{P}_X^N) 
        \big]
    \lesssim
        e^{-\sqrt{
            2\log(\gamma)
            \log(N)
        }}
    =
        e^{-\sqrt{
            \log(N^{2r})
        }}
.
\end{equation}
Combining~\eqref{eq:Concentration} with~\eqref{eq:mean_bound} implies that: for each $\eta>0$ we have that  
\begin{equation}
\label{eq:Concentration_NearComplete}
        \mathcal{W}_1(\mathbb{P}_X, \mathbb{P}_X^N) 
    \leq
        e^{-\sqrt{
            \log(N^{2r})
        }}
    +
        \eta 
\end{equation}
holds with probability at-least $
1-    2 \exp 
    \Big( 
        - \frac{
            N
            \eta^2
            }{        
                8v^2+ 4 V \eta
        } 
    \Big)
$.  Fix $\delta \in (0,1]$.  Setting $\delta = 2 \exp 
    \Big( 
        - \frac{
            N
            \eta^2
            }{        
                8v^2+ 4 V \eta
        } 
    \Big)$ and solving for $\eta$ implies that
\begin{equation}
\label{eq:QuadToSolve}
    N \eta ^2 - 4LV\eta - 8 C_{\delta}v^2 = 0
\end{equation}
where $C_{\delta}\eqdef \ln(2/\delta)$; which is a quadratic polynomial in $\eta>0$.  Its only positive solution is given by
\begin{align*}
% \numberthis
% \label{eq:eta_value}
    \eta 
% = 
%     \frac{
%         2V \ln\left(\frac{2}{\delta}\right) + 2\sqrt{\ln\left(\frac{2}{\delta}\right)\left(\ln\left(\frac{2}{\delta}\right)V^2 + 2Nv^2\right)}
%     }{
%         N
%     }
&
= 
    2V
    \ln\Big(
        \frac{2}{\delta}
    \Big)
    \frac{
        1
        + 
        \sqrt{
            \Big(
                1
                + 
                    2N\Big(\frac{v}{V}\Big)^2
                    /
                    \ln\big(
                    \frac{2}{\delta}
                \big) 
            \Big)
        }
    }{
        N
    }
\\
\numberthis
\label{eq:basic}
& 
\le 
       2V
    \Biggl(
        \frac{2
        \ln\big(
                \frac{2}{\delta}
            \big) 
        }{N}
        +
        \frac{
            \frac{v\sqrt{2}}{V}
            \sqrt{
            \ln\big(
                \frac{2}{\delta}
            \big) 
            }
        }{
            \sqrt{N}
        }
    \Biggr)
\\
& \le
    2V\max\Big\{2,\frac{\sqrt{2}v}{V}\Big\}
    \biggl(
        \frac{
        \ln\big(
                \frac{2}{\delta}
            \big) 
        }{N}
        +
        \frac{
            \sqrt{
            \ln\big(
                \frac{2}{\delta}
            \big) 
            }
        }{
            \sqrt{N}
        }
    \biggr)
\end{align*}
where~\eqref{eq:basic} held by the elementary inequality: $\sqrt{a+b}\le \sqrt{a}+\sqrt{b}$ for all $a,b\ge 0$.  Substituting this upper-bound for $\eta$ back into the right-hand side of~\eqref{eq:Concentration_NearComplete} yields the desired conclusion.
\end{proof}
To summarize the discussion to this point, we can express the results as a single transport-theoretic oracle inequality. Note that this inequality holds only under the assumption, soon to be rigorously proven below, that the hypothesis class can be taken to be Lipschitz with the same constant as the target nonlinear operator being learned.
We emphasize that the following oracle inequality fundamentally depends on us establishing the (Lipschitz) regularity of the approximating class (below).
\begin{proposition}[Regularity-Based Oracle Inequality]
\label{prop:Oracle_main}
In the setting of Lemmata~\ref{lem:excess_decomposition} and~\ref{lem:concentrate}:
For any $\hat{F}\in \mathcal{F}\subseteq \operatorname{Lip}(\mathcal{H},\Lt|L)$ which is an empirical risk minimizer, i.e. satisfies~\eqref{eq:ERM}, any compact $\mathcal{K}\subset\mathcal{H}$ of positive $\mathbb{P}_X$-measure, the following 
\[
    \mathcal{R}(\hat{F})
\le 
    \underbrace{
        \inf_{\tilde{F}\in \mathcal{F}}\,
        \sup_{x\in \mathcal{K}}
            \|\tilde{F}(x)-f^{\star}(x)\|_{\Lt}
    }_{\eqref{t:apprx}: \text{Approx.\ Err.}}
    +
    \bar{L}
    \underbrace{
         e^{-\sqrt{
            \log(N^{2r})
        }}
    +
        \frac{
        \ln\big(
                \frac{2}{\delta}
            \big) 
        }{N}
        +
        \frac{
            \sqrt{
            \ln\big(
                \frac{2}{\delta}
            \big) 
            }
        }{
            \sqrt{N}
        }
    }_{\term{t:stat_II}: \text{ Estimation}}
\]
holds with probability at-least $
1-(1-\mathbb{P}_X(\mathcal{K})^N)-\delta
$; where $\bar{L}\eqdef 2\max\{L,1\}$.
\end{proposition}
\begin{proof}
Directly follows from Lemma~\ref{lem:excess_decomposition} and~\ref{lem:concentrate}.
\end{proof}

\subsection{Approximation Guarantees}\label{sec:Approx-Guarantees}
To bound the approximation error in~\eqref{t:apprx}, we will leverage an infinite-dimensional extension of the main result in~\cite{hong2024bridging}, which provides an (optimal) Lipschitz regular version of the result in~\cite{galimberti2022designing}. The latter, in turn, is an infinite-dimensional extension of the irregular (optimal) approximation result from~\cite{shen2022optimal}. We expand our approximation around a solution $f(x^{\star})$, which is available in closed-form through classical (likely non-deep learning) methods. This idea is inspired by the numerical approach of~\cite{shi2023deep}, where a realistic hedging strategy was learned by perturbing a closed-form hedging strategy under idealized market conditions. Instead, our motivation is approximation rates, and using this idea, we can show that this residual method enables us to scale the approximation error favourably for small perturbation sets.
Fix some reference point $x^{\star}\in X$. For any $x\in X$, we henceforth write $\Delta x \eqdef x-x^{\star}$.
Let $W_0$ be the principal branch of the Lambert $W$ function; which we recall is well-defined on $(-1/e,\infty)$ and is \textit{negative-valued} on $(-1/e,0)$.
\begin{proposition}[General Version of Proposition~\ref{prop:theorem_universality__regular___specialcase}: Quantitative Approximation]
\label{prop:theorem_universality__regular}
Under Assumption~\ref{ass:Statistical}, suppose that $0\in \mathcal{K}$, $0<\alpha\le 1$, $L>0$, $f^{\star}:\mathcal{K}\to \Lt$ be an $(\alpha,L)$-H\"{o}lder map, and fix an error $\varepsilon>0$ and a $x^{\star}\in \HRules$.
There is an $(\mathbf{\alpha},{L})$-H\"{o}lder  
RNO $\hat{F}:\HRules\to \Lt$ satisfying 
\[
    \sup_{x\in \mathcal{K}}\,
        \big\|
        \hat{F}(\Delta x)
            -
            f^\star(x)
        \big \|_{\Lt}
    \le 
        \varepsilon
.
\]
Furthermore, $\hat{F}$ is base-point preserving; in that $\hat{F}(0)=\hat{y}^{\star}=f(x^{\star})$.
\hfill\\
If $\mathcal{K}$ is $(r,\varepsilon)$-exponentially ellipsoidal%
\footnote{Cf.\ Definition~\ref{defn:Ellispoids}.}~%
for some $r>0$, then there is an $\hat{F}$ of depth $\mathcal{O}(1)$ and whose width and number of non-zero parameters are $\hat{F}$ are $\mathcal{O}\Big(
        \varepsilon^{-c_{r,\alpha}}
        \log(\varepsilon^{-1})
    \Big)$; where $
    c_{r,\alpha}\eqdef
    \frac1{\alpha}
    \Big\lceil
        \log\big(
            \tfrac{(4L)^{\frac1{r\alpha}}}{r^{1/r}}
        \big)
    \Big\rceil$.
If $0<L<4e^{\alpha}$ and%
\footnote{$W_0$ denotes the principal branch of the Lambert $W$ function; i.e.\ the inverse function of $x\mapsto xe^x$ on $[-1/e,\infty)$.}~%
$r=
% -
1/W_0\Big(
% -\,
(4L)^{-\tfrac{1}{\alpha}}\Big)$ then, its width and connectivity are
$\mathcal{O}(\varepsilon^{-1}\log(\varepsilon^{-1}))$.
\end{proposition}
\begin{proof}[{Proof of Proposition~\ref{prop:theorem_universality__regular}}]
\hfill\\
\noindent\textbf{Step 1 - Extension:}
Since $\mathcal{H}$ and $\Lt$ are both separable Hilbert spaces and $f$ is $(\alpha,L)$-H\"{o}lder continuous then~\cite[Theorem 1.12]{BenyaminiLindenstrauss_2000_NonlinearFunctionalAnalysis} applies; whence, there exists an $(\alpha,L)$-H\"{o}lder continuous extension $f^{\uparrow}:\HRules\to \Lt$ of $f$; meaning that: for each $x\in \mathcal{K}$ we have
\begin{equation}
\label{eq:extension_def}
    f(x)=f^{\uparrow}(x) 
    \mbox{ and }f^{\uparrow}
    \mbox{ is $(\alpha,L)$-H\"{o}lder}
.
\end{equation}
Since the case where $f$ is constant is trivial, we assume without loss of generality that $f$ (and hence $f^{\uparrow}$) is non-constant; thus, $L>0$.
Next, We define the \textit{residual} target function $r_f:\HRules\to \Lt$ for each $x\in \HRules$ by 
\[
r_f(x)\eqdef f^{\uparrow}(x+x^{\star}) - f^{\uparrow}(x^{\star})
=
f^{\uparrow}(x+x^{\star})-y^{\star}
\] 
where $y^{\star}\eqdef f^{\uparrow}(x^{\star})$.
By construction, the following identity holds for each $x\in \mathcal{K}$
\begin{equation}
\label{eq:residual_function}
    \underbrace{
        r_f(
        \overbrace{x-x^{\star}}^{\Delta x}
        ) + y^{\star} 
    }_{\Delta y}
    =
        f^{\uparrow}((x-x^{\star})+x^{\star}) 
        -
        f^{\uparrow}(x^{\star})
        + y^{\star} 
    =
        f^{\uparrow}(x) 
    =
        f(x)
.
\end{equation}
Since $\HRules$ and $\Lt$ are Banach spaces then, translation is an isometry; thus $\operatorname{Lip}_{\alpha}(r_f)\le \operatorname{Lip}_{\alpha}(f)$. 
Observe $\Delta x^{\star}=x^{\star}-x^{\star}=0$; therefore~\eqref{eq:residual_function} implies that $r_f(0)=0$.
We will approximate the residual function $r_f$ on the following compact set of ``residuals'' $\Delta \mathcal{K}\eqdef 
\{u \in \HRules:\, (\exists x \in\mathcal{K})\,u=x-x^{\star}\}
$. Note that, the compactness of $\Delta \mathcal{K}$ is due to that of $\mathcal{K}$ and since the map $\HRules:x\mapsto x-x^{\star}\in \HRules$ is continuous ($\HRules$ is a TVS).

\noindent\textbf{Step 2 - Controlling the Error of Dimension Reduction:}
Since $\HRules$ and $\Lt$ are both separable Hilbert spaces, then they have the metric approximation property (MAP), i.e.\ the $1$-bounded approximation property; specifically, the projection operators $\{P_i^{\mathcal{Z}}\}_{i\in I}$ approximate the identity map on $\mathcal{Z}\in \{\HRules,\Lt\}$ uniformly on compact subsets thereof and each $P_i^{\mathcal{Z}}$ has operator norm at-most $1$.  
This implies two things: first, since $\HRules$ and $\Lt$ are both Hilbert spaces then, the embeddings $\{I_i^{\mathcal{Z}}\}_{i\in I}$ are isometric embeddings; again for $\mathcal{Z}\in \{\HRules,\Lt\}$, and therefore, for $\mathcal{Z}\in \{\HRules,\Lt\}$ and $i\in I$, the maps 
\begin{equation}
\label{eq:embeddings}
    A_i^{\mathcal{Z}}\eqdef I_i^{\mathcal{Z}}\circ P_i^{\mathcal{Z}}:\mathcal{Z}\to \mathcal{Z}
\end{equation}
are all $1$-Lipschitz linear operators of (finite) rank $i$.  
Second, since $\mathcal{K}$ is compact then the MAP implies that: for every ``dimension reduction error'' $\varepsilon_D>0$ there exists some $I\in \mathbb{N}_+$, depending only on $\varepsilon_D$ and on the compact set $\mathcal{K}$, satisfying
\begin{equation}
\label{eq:finite_dimensional_estimates}
\underbrace{
    \sup_{x\in \mathcal{K}}\,
        \|A_I^{\HRules}(x-x^{\star})-(x-x^{\star})\|_{\HRules}
    \le 
        \frac{\varepsilon_D^{1/\alpha}}{(2L)^{1/\alpha}}
}_{\term{DRIn}: \,\text{Dimension-Reduction: Domain}}
\mbox{ and }
\underbrace{
    \sup_{y\in f(\mathcal{K})\cup\{y^{\star}\}}\,
        \|A_I^{\Lt}(y
        % -y^{\star}
        )-
        y
        % (y-y^{\star})
        \|_{\HRules}
    \le 
        \frac{\varepsilon_D}{2}
}_{\term{DROut}: \text{Dimension-Reduction: Range}}
.
\end{equation}
In the special case where $\mathcal{K}$ and $f(\mathcal{K})$  are both $(\rho,r)$-exponentially ellipsoidal with respect to the respective bases $(e_i)_{i\in I}$ and $(\eta_j)_{j\in J}$ (cf.\ Definition~\eqref{eq:Exp_Ellipsoidal}) then, we may \textit{quantify} the dependence of $I$ and on $\varepsilon_D,r,\rho$ as follows: for each $x\in \mathcal{K}$
\allowdisplaybreaks
\begin{align}
\label{eq:exp_align__estimate_Begin}
        \|A_I^{\HRules}(x-x^{\star})-(x-x^{\star})\|_{\HRules}
    & =
        \big\|
            \sum_{i>I}\,\langle x-x^{\star},e_i\rangle e_i
        \big\|
\\
\nonumber
    & \le 
        \sum_{i>I}\,
        \big\|
            \langle x-x^{\star},e_i\rangle 
        \big\|
        \,
        \|e_i\|
\\
\label{eq:ONorm}
    & =
        \sum_{i>I}\,
        \big\|
            \langle x-x^{\star},e_i\rangle 
        \big\|
\\
\label{eq:exp_alignement}
    & \le
        \sum_{i>I}\,
            \rho\, e^{-r\,i}
\\
% \nonumber
    & \le
        \rho\,
        \int_I^{\infty}
            e^{-r\,u}
            \,du
% \\
\label{eq:exp_align__estimate_end}
%     & 
    =
    \frac{\rho\,e^{-rI}}{r}
,
\end{align}
where~\eqref{eq:ONorm} follows since each $\{e_i\}_{i\in I}$ is a unit vector and~\eqref{eq:exp_alignement} follows from the exponentially ellipsoidal hypothesis on $\mathcal{K}$.  Thus,~\eqref{eq:exp_align__estimate_Begin}-\eqref{eq:exp_align__estimate_end} imply that any choice of $I\in \mathbb{N}_+$ satisfying
\begin{equation}
\label{eq:estimate_on_I__In}
        I
    \ge
        % \Biggl\lceil
            \log\biggl(
                \frac{(2L\rho)^{1/(\alpha r)}}{r^{1/r} \varepsilon_D^{1/(\alpha r)}}
            \biggr)
        % \Biggr\rceil
\end{equation}
ensures that~\eqref{DRIn} holds.  A nearly identical computation shows that~\eqref{DROut} holds if $I$ satisfies
\begin{equation}
\label{eq:estimate_on_I__Out}
        I
    \ge
        % \Biggl\lceil
            \log\biggl(
                \frac{(\rho2)^{1/r}}{(r\varepsilon_D)^{1/r}}
            \biggr)
        % \Biggr\rceil
.
\end{equation}
In view of combining~\eqref{eq:estimate_on_I__In} and~\eqref{eq:estimate_on_I__Out} and since $1/\alpha\ge 1$ and $\log$ is monotonically increasing; then, it is enough to set
\begin{equation}
\label{eq:estimate_on_I__In2}
        I
    \eqdef 
        \Biggl\lceil
            \log\biggl(
                \frac{\rho^{1/(\alpha r)}(2L)^{1/(\alpha r)}}{r^{1/r}\varepsilon_D^{1/(\alpha r)}}
            \biggr)
        \Biggr\rceil
    \in 
        \mathcal{O}\Biggl(
            \log\biggl(
                \frac{
                \rho^{1/(\alpha r)}
                L^{1/(\alpha r)}}{
                r^{1/r}
                \varepsilon_D^{1/(\alpha r)}}
            \biggr)
        \Biggr)
\end{equation}
in order to ensure that both conditions in~\eqref{eq:finite_dimensional_estimates} are simultaneously satisfied.  
\nonumber
\textit{If we couple the radius, $\rho$, of the ellipsoidal set $\mathcal{K}$ to the dimension reduction error $\varepsilon_D$ via}
\begin{equation}
\label{eq:magic_scaling}
\rho\in \mathcal{O}(\varepsilon_D)
\end{equation} 
\textit{then $I$ becomes independent of the dimension reduction error since }
\begin{equation}
\label{eq:constant_dimension}
I
=
\lceil
    r\log(r^{-1}(2L)^{1/(\alpha r)})
\rceil
\in 
    \mathcal{O}(1).
\end{equation}
We will return to the observation in~\eqref{eq:constant_dimension} at the end of the proof.

\noindent\textbf{Step 3 - Approximation Error Decomposition:}
For any $\operatorname{ReLU}$ MLP $\hat{f}:\mathbb{R}^I\to \mathbb{R}^I$, to be specified retroactively, consider the induced residual NO (RNO) given by
\begin{equation}
\label{eq:RNO_form__proof}
    \hat{F}(u)\eqdef I_I^{\Lt}\circ \hat{f}\circ P_I^{\HRules}(u) + y^{\star}
\end{equation}
where $u\in \HRules$; note we have chosen our base-points $x^{\star}$ and $y^{\star}$ in Definition~\ref{defn:RNONO} according to~\eqref{eq:residual_function}.
Thus,~\eqref{eq:residual_function} implies that for all $\Delta x\in \Delta \mathcal{K}$ write $x\eqdef \Delta x +x^{\star}$
\allowdisplaybreaks
\begin{align*}
    \term{t:Main_Error}
    % & 
    \eqdef 
    \big\|
        \hat{F}(\Delta x)
            -
        f(x)
    \big\|_{\Lt}
    % \\
    % &
    =
        \|
                I_I^{\Lt}\circ 
            \hat{f}\circ P_I^{\HRules}
            (x) 
            - 
                r_f(\Delta x)
        \|_{\Lt}
.
\end{align*}
Consequently,~\eqref{t:Main_Error} may be bounded above, for each $\Delta x\in \Delta \mathcal{K}$ by
\allowdisplaybreaks
\begin{align*}
    \eqref{t:Main_Error}
    & 
    =
    \big\|
        I_I^{\Lt}\circ 
        \hat{f}\circ P_I^{\HRules}
        (x)
            -
        r_f(x)
    \big\|_{\Lt}
%%%
\\
& \le
    \underbrace{
        \big\|
            I^{\Lt}_I\circ \big(
                \hat{f}\circ P_I^{\HRules}(x)
            \big) 
                -
            \big(
                (I_I^{\Lt}\circ P_I^{\Lt})
                \circ
                r_f
                \circ 
                (I_I^{\HRules}\circ P_I^{\HRules})
                (x)
            \big)
        \big\|_{\Lt}
    }_{\term{t:Approx}}
\\
    &
    +
    \underbrace{
        \big\| 
            \big(
                (I_I^{\Lt}\circ P_I^{\Lt})
                \circ
                r_f\circ (I_I^{\HRules}\circ P_I^{\HRules})(x)
            \big)
                -
            \big(
                (I_I^{\Lt}\circ P_I^{\Lt})
                \circ r_f(x)
            \big)
        \big\|_{\Lt}
    }_{\term{t:DR_X}}
\\
    &
    +
    \underbrace{
        \big\|
            \big(
                (I_I^{\Lt}\circ P_I^{\Lt})
                \circ r_f(x)
            \big)
                -
                r_f(x)
        \big\|_{\Lt}
    }_{\term{t:DR_Y}}
\end{align*}
Our objective is to bound the approximation error~\eqref{t:Main_Error}, which we will accomplish by bounding each of the individual errors~\eqref{t:Approx},~\eqref{t:DR_X}, and~\eqref{t:DR_Y}; which we now do.

\noindent \textbf{Step 4: Bounding the Individual Errors:}
First, by~\eqref{DROut}, we obtain
\allowdisplaybreaks
\begin{align}
\label{eq:BoundOn__DR_Y}
        \eqref{t:DR_Y}
    & =
       \big\|
            \big(
                (I_I^{\Lt}\circ P_I^{\Lt})
                \circ r_f(x)
            \big)
                -
            r_f(x)
        \big\|_{\Lt}
\\
    & 
    \le
    \sup_{y \in f(\mathcal{K})}
       \big\|
            \big(
                (I_I^{\Lt}\circ P_I^{\Lt})(y)
            \big)
                -
            y
        \big\|_{\HRules}^{\alpha}
    \le \frac{\varepsilon_D}{2}
.
\end{align}
Next, using~\eqref{DRIn}, we control~\eqref{t:DR_X} as follows
\allowdisplaybreaks
\begin{align}
\nonumber
        \eqref{t:DR_X}
    &
    \le
        \operatorname{Lip}(I_I^{\Lt})
        \operatorname{Lip}(P_I^{\Lt})
        \operatorname{Lip}_{\alpha}(r_f)
        \big\| 
            \big(
                (I_I^{\HRules}\circ P_I^{\HRules})(x)
            \big)
                -
            x
        \big\|_{\Lt}^{\alpha}
\\
& 
    \le
        \operatorname{Lip}(I_I^{\Lt})
        \operatorname{Lip}(P_I^{\Lt})
        \operatorname{Lip}_{\alpha}(f)
        \big\| 
            \big(
                (I_I^{\HRules}\circ P_I^{\HRules})(x)
            \big)
                -
            x
        \big\|_{\Lt}^{\alpha}
\\
\nonumber
    &
    \le
    \sup_{x\in \mathcal{K}}\,
        L
        \big\| 
            \big(
                (I_I^{\HRules}\circ P_I^{\HRules})(x)
            \big)
                -
            x
        \big\|_{\Lt}^{\alpha}
\\
\nonumber
    &
    \le
    \sup_{x\in \mathcal{K}}\,
        L
        \Big( 
            \frac{\varepsilon_D^{1/\alpha}}{(2L)^{1/\alpha}}
        \Big)^{\alpha}
\\
\label{eq:BoundOn__DR_X}
    &
    =
        \frac{\varepsilon_D}{2}
.
\end{align}
Consequently, all the dimension reduction errors in~\eqref{t:Main_Error} have been bounded above leading to
\allowdisplaybreaks
\begin{align}
\nonumber
    \eqref{t:Main_Error}
    & 
    \le
        \varepsilon_D
        +
        \operatorname{Lip}(I^{\Lt}_I)
        \big\|
            \big(
                \hat{f}\circ P_I^{\HRules}(x)
            \big) 
                -
            \big(
                P_I^{\Lt}
                \circ
                r_f\circ 
                (I_I^{\HRules}\circ P_I^{\HRules})
                (x)
            \big)
        \big\|_{\mathbb{R}^I}
\\
\label{eq:finalbound}
    & 
    \le
        \varepsilon_D
        +
    \underbrace{
        \sup_{u\in \mathcal{K}_I}
        \big\|
                \hat{f}(u) 
                -
            \big(
                P_I^{\Lt}
                \circ
                r_f\circ 
                (I_I^{\HRules}(u)
            \big)
        \big\|_{\mathbb{R}^I}
    }_{\term{t:Approx_prime}}
\end{align}
where $\mathcal{K}_I\eqdef P_I^{\HRules}(\mathcal{K})$ is compact (by to the compactness of $\mathcal{K}$ and the continuity of $P_I^{\HRules}$).  
Remark that the map $f_{I}:\mathbb{R}^I\to \mathbb{R}^I$ given for any $u\in \mathcal{K}_I$ by
$
    f_I(u)\eqdef P_I^{\Lt}
                \circ
                r_f\circ 
                (I_I^{\HRules}(u)
$ is $(L,\alpha)$-H\"{o}lder continuous since $P_I^{\Lt}$ and $I_I^{\HRules}$ are both $1$-Lipschitz and since $r_f$ is $(L,\alpha)$-H\"{o}lder continuous.  Thus, for every $C\in \mathbb{N}_+$ (to be set retroactively), upon applying \cite[Corollary 5.1]{hong2024bridging} we may \textit{retroactively} pick $\hat{f}$ to be a (globally) $(L,\alpha)$-H\"{o}lder continuous ReLU MLP with depth $\lceil \log_2(I)\rceil + 5$, width $
8 \,I (1+ C)^I
$, with at-most $
18\, I (1+ C)^I
$ non-zero (trainable) parameters such that
\begin{equation}
\max_{u \in \mathcal{K}_I}\,
        \|
            \hat{f}(u) - f_I(u)
        \|_2
    \le 
        L(I/2C)^{\alpha}
.
\end{equation}
Retroactively, setting the connectivity parameter $C$ to be 
\begin{align}
\label{eq:How_much_connex}
        C
    \eqdef 
        \biggl\lceil
            \frac{I}{2}\biggl(\frac{L}{\varepsilon_A}\biggr)^{1/\alpha}
        \biggr\rceil
\end{align}
implies that~\eqref{t:Approx_prime} is at most $\varepsilon_A$.  Further, this determines the ReLU MLP $\hat{f}$.  Our bound on~\eqref{t:Main_Error} is complete since the inequality in~\eqref{eq:finalbound} has further reduced to
\begin{align*}
    \eqref{t:Main_Error}
    & 
    \le
        \varepsilon_D
        +
        \varepsilon_A
.
\end{align*}
Let $\varepsilon>0$ and retroactively couple $\varepsilon_D=\varepsilon_A\eqdef \varepsilon/2$.
It only remains to tally parameters.   
Combining our estimate for the latent dimension $I$ in~\eqref{eq:estimate_on_I__In2} with the estimate in~\eqref{eq:How_much_connex} yields
\begin{align*}
\begin{aligned}
        C
   &  =
        \big\lceil
            IL^{1/\alpha}
            r^{-r}
            % \varepsilon_A
            (2^{-1}\varepsilon)
            ^{-1/\alpha}
        \big\rceil
    \\
    & \in 
        \mathcal{O}\biggl(
            % This is I
            \log\biggl(
                (CL)^{1/(\alpha r)}
                \,
                r^{-1/r} 
                % \varepsilon_D
                (2^{-1}\varepsilon)
                ^{-1/(\alpha r)}
            \biggr)
            \,
            %% This is from Approx
            L^{1/\alpha}
            \,
            % \varepsilon_A
            (2^{-1}\varepsilon)
            ^{-1/\alpha}
        \biggr)
.
\end{aligned}
\end{align*}
Similarly, $\hat{f}$ has a depth of at-most
\[
    \lceil
        \log_2(I)
    \rceil
    +5
=
    \Biggl\lceil
        \log_2\Biggl(
            % I 
            \Biggl\lceil
            \log\biggl( 
                    \rho^{1/(\alpha r)}(2L)^{1/( \alpha r)}
                \,
                r^{-1/r}
                (\varepsilon/2)^{-1/(\alpha r)}
            \biggr)
        \Biggr\rceil
        \Biggr)
    \Biggr\rceil
    +
    5
\]
and a width of the order of
\begin{equation}
\label{eq:width_tooBigStill}
\resizebox{0.95\hsize}{!}{$
\begin{aligned}
    W
    \eqdef
    &
            8
        \,
        \log\big((CL)^{r/\alpha}\,r^{-r} 
        % \varepsilon_D
        (2^{-1}\varepsilon_A)
        ^{-1/(\alpha r)}\big)
    \\
    & \times
        \biggl(
            1
            +
            \biggl\lceil
                \frac{
                \big(
                    \log\big((CL)^{1/(\alpha r)}\,r^{-1/r} 
                        (2^{-1}\varepsilon_A)
                        ^{-1/(\alpha r)}\big)
                \big)
                }{2}\biggl(\frac{L}{
                % \varepsilon_A
                (2^{-1}\varepsilon)
                }\biggr)^{1/\alpha}
            \biggr\rceil
        \biggr)^{
            \Biggl\lceil
            \log\biggl(
                \rho^{r/\alpha}(2L)^{1/(\alpha r)}
                \,
                r^{-r}
                % \varepsilon_D
                (2^{-1}\varepsilon_D)
            \biggr)
            \Biggr\rceil
        }
\end{aligned}
$}
\end{equation} 
and the number of non-zero parameters are the same up to a (constant) multiple of $W$, namely $9/4$.

\noindent \textbf{The Special Case of Exponentially Ellipsoidal Sets with Scaling Radii}
\hfill\\
\noindent If we scale the ``ellipsoidal radius'' $\rho$ of the relevant compact set by \textit{retroactively} coupling 
$\rho \eqdef \varepsilon
$, mirroring~\eqref{eq:magic_scaling}.  
Then, the exponent in~\eqref{eq:width_tooBigStill} becomes
\[
        \tilde{c}_{r,\alpha}
    \eqdef 
        \Biggl\lceil
        \log\biggl(
            (2L)^{r/\alpha}
            2^{r/\alpha}
            \,
            \,
            r^{-1/r}
            \Big(
                \rho/\varepsilon
            \Big)^{1/(\alpha r)}
        \biggr)
        \Biggr\rceil
       =
        \Biggl\lceil
            \log\biggl(
                \frac{(4L)^{\frac{1}{(\alpha r)}}}{r^{1/r}}
            \biggr)
        \Biggr\rceil
    .
\]
Setting $c_{r,\alpha}\eqdef \tilde{c}_{r,\alpha}/\alpha$; then,~\eqref{eq:width_tooBigStill} ameliorates to

$ W
    \in \mathcal{O}\Big(
        \varepsilon^{-c_{r,\alpha}}
        \log(\varepsilon^{-1})
    \Big)
.
$
Similarly, the depth $\Delta$ of $\hat{f}$ ameliorates to
$%$\[
    \Delta
=
        \Big\lceil
        \log_2\Big(
            \Big\lceil
            \log\big( 
                (4L)^{r/\alpha} r^{-r}
            \big)
        \Big\rceil
        \Big)
        \Big\rceil
    + 
        5
    \in \mathcal{O}(1)
$.

This completes the main portion of the proof.  It only remains to verify the ``base-point preserving property'' of $\hat{F}$; i.e.\ that $\hat{F}(0)=y^{\star}$.  

\noindent
\textbf{Verifying the Pointedness of $\hat{F}-y^{\star}$:}
Since $0\in \mathcal{K}$ and by the linearity of $P_I^{\HRules}$ we have
\[
    P_I^{\HRules}(0)=0 \in \mathbb{R}^I.
\]
Note also that, since $r_f(0)$ then again the linearity of $I_I^{\mathcal{Z}}$ and $P_I^{\mathcal{Z}}$ for $\mathcal{Z}\in \{\HRules,\Lt\}$ implies that the finite-dimensional version of our target residual function $(I_I^{\Lt}\circ P_I^{\Lt})
                \circ
                r_f
                \circ 
                (I_I^{\HRules}\circ P_I^{\HRules})
                (x)
            \big):\mathbb{R}^I\to \mathbb{R}^I$ in~\eqref{t:Approx} fixes $0$.  Since $\hat{f}$ was constructed using~\cite[Corollary 5.1]{hong2024bridging}; then the ``ample-interpolation'' property of the ReLU MLP $\hat{f}$ guaranteed by~\cite[Theorem 4.1]{hong2024bridging} implies that $\hat{f}(0)=r_f(0)=0$.  Again appealing to the linearity of $I_I^{\Lt}$ we have that 
\[
    \hat{F}(0)-y^{\star}=I_I^{\Lt}\circ \hat{f}\circ P_I^{\HRules}(0)+\hat{y}^{\star}-\hat{y}^{\star}
    =
        0
.
\]
This yields the last statement and concludes our proof.
\end{proof}

\subsection{Completing the Proof of the Main Learning Guarantee}

Equipped with the regular universal approximation guarantee from Proposition~\ref{prop:theorem_universality__regular} and the oracle-type inequality in Proposition~\ref{prop:Oracle_main}, we are now in a position to present our main abstract learning guarantee.
We first work out the general case, then specialize to the exponentially ellipsoidal case.

\begin{proposition}[Learnability]
\label{prop:LearninGuarantee}
Consider the setting of Proposition~\ref{prop:theorem_universality__regular} 
{and suppose that~Assumption \ref{ass:Statistical} (ii) is generalized to $Y_n=f(X_n)$ for $n=1,\dots,N$, holds.}
%%%
For every approximation error $\varepsilon>0$, and a failure probability $0<\delta\le 1$.
For compact every $\mathcal{K}\subset \mathcal{H}$ of positive $\mathbb{P}_X$-measure containing%
\footnote{I.e.\ $\mathbb{P}_X(\mathcal{K})>0$ and $0\in \mathcal{K}$.}~%
$0$ and each $L$-Lipschitz target function $f:\mathcal{K}\to \Lt$: 
there is a connectivity parameter $C\eqdef C(\epsilon,\delta,\mathcal{K})\in \mathbb{N}_+$ such that for any empirical risk minimizer  $\hat{F}\in \mathcal{RNO}^{1,L}_C(e_{\cdot},\eta_{\cdot})$; i.e.\
\begin{equation}
\label{eq:ERM_CONDITION}
        \hat{\mathcal{R}}^S(\hat{F})
    =
        \inf_{
                \tilde{F}
            \in 
                \mathcal{RNO}^{1,L}_C(e_{\cdot},\eta_{\cdot})
        }\,
        \hat{\mathcal{R}}^S(\tilde{F})
.
\end{equation}
and the following holds 
\begin{equation}
\label{eq:LearninGuarantee__Statement}
\begin{aligned}
\mathbb{P}
\biggl(
        \mathcal{R}(\hat{F})
    \le 
            \varepsilon
        +
            \bar{L}
                 e^{-\sqrt{
                    \log(N^{2r})
                }}
            +
                \frac{
                \ln\big(
                        \frac{2}{\delta}
                    \big) 
                }{N}
                +
                \frac{
                    \sqrt{
                    \ln\big(
                        \frac{2}{\delta}
                    \big) 
                    }
                }{
                    \sqrt{N}
                }
\bigg)
% \ge 
%     1-(1-p)^N-\delta 
\ge 
    p^N-\delta
,
\end{aligned}
\end{equation}
where 
$p\eqdef \mathbb{P}_X(\mathcal{K})$ and (
$\bar{L}\eqdef 2\max\{L,1\}$, where $r>0$ is as in Assumption~\ref{ass:Statistical} (iii).
\hfill\\
If, additionally, there exists some $\bar{r}>0$ such that $\mathcal{K}$ is $(\varepsilon,\bar{r})$-exponentially ellipsoidal%
\footnote{Cf.\ Definition~\ref{defn:Ellispoids}.}~%
then, $\hat{F}$ has depth $\mathcal{O}(1)$ and both the width and the number of non-zero parameters of $\hat{F}$ are $\mathcal{O}\big(
        \varepsilon^{-c_{\bar{r},\alpha}}
        \log(\varepsilon^{-1})
    \big)$; where 
$
    c_{\bar{r},\alpha}
\eqdef
    \frac1{\alpha}
    \Big\lceil
        \tfrac{\log(4\bar{L})}{\bar{r}\alpha}
        -
        \tfrac{\log(\bar{r})}{\bar{r}}
    \Big\rceil
$.
\end{proposition}
\begin{proof}[{Proof of Proposition~\ref{prop:LearninGuarantee}}]
Suppose that $\mathcal{K}$ has positive $\mathbb{P}_X$ measure.  Since each $\tilde{F}\in \mathcal{RNO}_C^{1,L}(e_{\cdot},\eta_{\cdot})$ and the target function $f$ are $L$-Lipschitz, and since every element of $\operatorname{RNO}_C^{1,L}(e_{\cdot},\eta_{\cdot})$ fixes the origin--i.e.\ maps the zero vector to the zero vector in their respective spaces--then, Proposition~\ref{prop:Oracle_main} applies with $\mathcal{F}\eqdef \operatorname{RNO}_C^{1,L}(e_{\cdot},\eta_{\cdot})$; whence, for every empirical risk minimizer $\hat{F}\in \operatorname{RNO}_C^{1,L}(e_{\cdot},\eta_{\cdot})$ (i.e.\ satisfying~\eqref{eq:ERM} and thus satisfying~\eqref{eq:ERM_CONDITION}) the following holds
\begin{equation}
\label{eq:Bound_in_main_proof}
\begin{aligned}
        \mathcal{R}(\hat{F})
    \le 
% &
        \underbrace{
            \inf_{
                \tilde{F}
                \in 
                \operatorname{RNO}_C^{1,L}(e_{\cdot},\eta_{\cdot})
            }\,
            \sup_{x\in \mathcal{K}}
                \|\tilde{F}(x)-F(x)\|_{\Lt}
        }_{\eqref{t:apprx}}
% \\
%     &
        +
        \bar{L}
        \underbrace{
             e^{-\sqrt{
                \log(N^{2r})
            }}
        +
            \frac{
            \ln\big(
                    \frac{2}{\delta}
                \big) 
            }{N}
            +
            \frac{
                \sqrt{
                \ln\big(
                    \frac{2}{\delta}
                \big) 
                }
            }{
                \sqrt{N}
            }
        }_{\eqref{t:stat_II}}
\end{aligned}
\end{equation}
with probability at-least $1-(1-\mathbb{P}_X(\mathcal{K}))^N-\delta$.

As term~\eqref{t:stat_II} converges to $0$ as the sample size $N$ tends to infinity, it only remains to control the approximation error--expressed by term~\eqref{t:apprx}.  
Since the target function $f$ is $L$-Lipschitz, i.e.\ $(1,L)$-H\"{o}lder, and since $\mathcal{K}$ is compact then Proposition~\ref{prop:theorem_universality__regular} applies upon
setting the depth, width, and connectivity parameters large enough to guarantee the existence of some 
$\hat{F}\in \operatorname{RNO}_C^{1,L}(e_{\cdot},\eta_{\cdot})$ satisfying~\eqref{eq:ERM_CONDITION} and such that
\begin{align}
\label{eq:existence_uniform_bound}
        \sup_{x\in \mathcal{K}}
        \,
            \big\|
                \hat{F}(\Delta x)
                -
                f(x)
            \big \|_{\Lt}
    \le 
        \varepsilon
.
\end{align}
If, additionally, there exists some {$\bar{r}>0$} such that $\mathcal{K}$ is $(\varepsilon,\bar{r})$-exponentially ellipsoidal then 
%%%%%
there $\hat{F}$ can be guaranteed to have $\mathcal{O}(1)$, width and number of non-zero parameters ($C$) at-most $\mathcal{O}\big(
        \varepsilon^{-c_{\bar{r},\alpha}}
        \log(\varepsilon^{-1})
    \big)
$; where $
c_{\bar{r},\alpha}\eqdef
\frac1{\alpha}
\Big\lceil
    \log\big(
        \tfrac{(4\bar{L})^{\frac1{\bar{r}\alpha}}}{\bar{r}^{1/\bar{r}}}
    \big)
\Big\rceil
$.
%%%%%
%%%%%
Consequently, the uniform bound in~\eqref{eq:existence_uniform_bound} implies that
\begin{align}
\label{eq:Uniform_Bound_Achieed}
        \eqref{t:apprx}
    =
        \inf_{\tilde{F}\in \operatorname{RNO}_C^{1,L}(e_{\cdot},\eta_{\cdot})}
        \sup_{x\in \mathcal{K}}
        \,
            \big\|
                \hat{F}^{\star}(\Delta x)
                -
                f^\star(x)
            \big \|_{\Lt}
    \le
        \sup_{x\in \mathcal{K}}
        \,
            \big\|
                \hat{F}^{\star}(\Delta x)
                -
                f^\star(x)
            \big \|_{\Lt}
    \le 
        \varepsilon
.
\end{align}
Incorporating the estimate in~\eqref{eq:Uniform_Bound_Achieed} back into~\eqref{eq:Uniform_Bound_Achieed} we find that
\begin{equation}
\label{eq:Bound_in_main_proof2}
\begin{aligned}
        \mathcal{R}(\hat{F})
    \le 
            \varepsilon
        +
            \bar{L}
            \underbrace{
                 e^{-\sqrt{
                    \log(N^{2r})
                }}
            +
                \frac{
                \ln\big(
                        \frac{2}{\delta}
                    \big) 
                }{N}
                +
                \frac{
                    \sqrt{
                    \ln\big(
                        \frac{2}{\delta}
                    \big) 
                    }
                }{
                    \sqrt{N}
                }
            }_{\eqref{t:stat_II}}
\end{aligned}
\end{equation}
holds with probability at-least $1-(1-\mathbb{P}_X(\mathcal{K})^N)-\delta
=
\mathbb{P}_X(\mathcal{K})^N-\delta
$.  
\end{proof}

\subsection{{Proof of Theorem~\ref{thrm:MainLearning__Faster}}}
\label{a:PROOF_thrm:MainLearning__Faster}
It remains to control the probability of obtaining all $N$ samples on the exponentially ellipsoidal set given in the previous proposition.  Indeed, under Assumption~\ref{ass:high_tempirature}, we have the following guarantee.
\begin{lemma}
\label{lem:high_tempirature__likelyhits}
Let $\rho\ge 0$, $r>0$, and let $\mathcal{K}\subseteq \mathcal{H}$ be exponentially $(\rho,r)$-exponentially ellipsoidal and contain the origin $0\in \mathcal{H}$.
Fix a ``sampling failure probability'' $0<\delta_X \le 1$, and assume that $\mathbb{P}_X$ further satisfies Assumption~\ref{ass:high_tempirature}.  Then
\begin{equation}
\label{eq:effective_sampling_guarantee}
    \mathbb{P}(X_1\in \mathcal{K})\ge 1-\delta_X
.
\end{equation}
\end{lemma}
\begin{proof}[{Proof of Lemma~\ref{lem:high_tempirature__likelyhits}}]
Since we have assumed that, for all $i\in \mathbb{N}_+$ and every $t\ge 0$ $\mathbb{P}(|Z_i|\ge t)\le 2e^{-t^2/2}$ then the independence of the $\{Z_i\}_{i=1}^{\infty}$, the relationship between $X_1\sim X\sim \mathbb{P}_X$ in~\eqref{eq:KLDecomposition}, Assumption~\ref{ass:Statistical}, and the definition of $(\rho,r)$-exponentially ellipsoidal sets, implies that
\begin{equation}
\label{eq:Assumption_panning_out}
    \mathbb{P}(X\in \mathcal{K}) 
= 
    \mathbb{P}((\forall i\in \mathbb{N}_+)|\sigma_i \,Z_i|\le \rho e^{-ri})
=
    \prod_{i=1}^{\infty}\,
        \mathbb{P}(|\sigma_i \,Z_i|\le \rho e^{-ri})
\ge 
    \prod_{i=1}^{\infty}\,
        \big(
            1
            -
            2e^{-t_i^2/2\sigma_i^2}
        \big)
.
\end{equation}
Now, the constant on $\sigma_i$ in Assumption~\ref{ass:high_tempirature}implies that: for every $i\in \mathbb{N}_+$ we have $
    1
    -
    2e^{-t_i^2/2\sigma_i^2}
\ge 
    1-\Delta^i
$.  Consequently, the left-hand side of~\eqref{eq:Assumption_panning_out} can be bounded below by
\begin{equation}
\label{eq:Assumption_panning_out__2}
    \mathbb{P}(X\in \mathcal{K}) 
\ge 
    \prod_{i=1}^{\infty}\,
        \big(
            1
            -
            2e^{-t_i^2/2\sigma_i^2}
        \big)
\ge 
    \prod_{i=1}^{\infty}\,
        \big(
            1
            -
            \Delta^i
        \big)
\ge 
    \exp\Big(
        -2\sum_{i=1}^{\infty}\,
        \Delta^i
    \Big)
=
    \exp\big(
        \tfrac{
        -2\Delta
        }{1-\Delta}
    \big)
.
\end{equation}
The definition of $\Delta$, as coupled to $\delta_X$, given in Assumption~\ref{ass:high_tempirature} implies that $\exp\big(
        \tfrac{
        -2\Delta
        }{1-\Delta}
    \big)=1-\delta_X$ which concludes our proof.
\end{proof}

\begin{proof}[{Proof of Theorem~\ref{thrm:MainLearning__Faster}}]
In the setting of Proposition~\ref{prop:LearninGuarantee} if, additionally Assumption~\ref{ass:high_tempirature} holds with $\delta_X=e^{-\delta/N}$ then the result following upon applying Lemma~\ref{lem:high_tempirature__likelyhits}, since, in this setting, $p=\mathbb{P}_X(\mathcal{K})\ge \delta_N =e^{-\delta/N}$; whence $p^N-\delta\ge e^{-\delta}-\delta$.
\end{proof}

\section*{Acknowledgments}
A.\ Kratsios and X.\ Yang acknowledge financial support from an NSERC Discovery Grant No.\ RGPIN-2023-04482 and No.\ DGECR-2023-00230, and they acknowledge that resources used in preparing this research were provided, in part, by the Province of Ontario, the Government of Canada through CIFAR, and companies sponsoring the Vector Institute\footnote{\href{https://vectorinstitute.ai/partnerships/current-partners/}{https://vectorinstitute.ai/partnerships/current-partners/}}.
Dena Firoozi would like to acknowledge the support of the Natural Sciences and Engineering Research Council of Canada (NSERC), grants RGPIN-2022-05337 and DGECR-2022-00468.

\bibliographystyle{elsarticle-num-names}
\bibliography{0_references,0_MFG_HS_Ref}

\appendix 

\section{Additional Background}\label{App1}
This appendix contains additional background supporting a self-contained reading of our manuscript.  

\subsection{Mean Field Game Equilibrium Strategy}
\label{s:MFGRecalEquilibrium}
By~\cite[Theorem 4.1]{LF24}, the mappings given by~\eqref{def:Delta123}-\eqref{def:Gamma12} satisfy the following bounds: 
\begin{align} 
 \Gamma_1 \in \mathcal{L}( \mathcal{L}(H); H) \quad & \mbox{and} \quad \| \Gamma_1 \| \leq R_1 
\quad \mbox{with} \quad R_1 = \tr(Q) \| D \| , 
\label{Gamma1bd} \\ 
 \Gamma_2 \in \mathcal{L}( \mathcal{L}(H); U) \quad & \mbox{and} \quad \| \Gamma_2 \| \leq R_2  
\quad \mbox{with} \quad R_2 = \tr(Q) \| E \| , 
\label{Gamma2bd} \\ 
 \Delta_1 \in \mathcal{L}( \mathcal{L}(H); \mathcal{L}(H; U )) \quad & \mbox{and} \quad \| \Delta_1 \| \leq R_3 
\quad \mbox{with} \quad R_3 = \tr(Q) \| D \| \|E\| , 
\label{Delta1bd} \\ 
\Delta_2 \in \mathcal{L}( \mathcal{L}(H); \mathcal{L}(H) ) \quad & \mbox{and} \quad \| \Delta_2 \| \leq R_4 
\quad \mbox{with} \quad R_4 = \tr(Q) \| D \|^2 , 
\label{Delta2bd} \\ 
\Delta_3 \in \mathcal{L}( \mathcal{L}(H); \mathcal{L}( U )) \quad & \mbox{and} \quad \| \Delta_3 \| \leq R_5  
\quad \mbox{with} \quad R_5 = \tr(Q) \|E\|^2 .  
\label{Delta3bd} 
\end{align} 
Below, we present the equilibrium strategy associated with the set of operators $(A, B, D, E, F_1, F_2, \sigma,\\ M, \widehat{F}_1, \widehat{F}_2, G)$ for the MFG system \eqref{dx_LQ_MFG}-\eqref{Jinfty}. The strategy corresponding to variations in these operators--for instance, in the reference model $(A^\dagger, B^\dagger, D, E, F_1, F_2^\dagger, \sigma, M, \widehat{F}_1, \widehat{F}_2, G)$ or in perturbed cases--can be obtained by direct substitution into the relevant equations.
 \begin{theorem}[MFG Equilibrium Strategy]\cite[Theorem 4.9]{LF24}.\label{Nash-eq} Consider the Hilbert space-valued limiting system with the set of operators ($A$, $B$, $D$, $E$, $F_1$, $F_2$, $\sigma$, $M$, $\widehat{F}_1$, $\widehat{F}_2$, $G$) described by \eqref{dx_LQ_MFG}-\eqref{Jinfty}.
 Suppose the relevant Riesz mappings $\Delta_k, k=1,2,3,$ $\Gamma_k, k=1,2,$ given by~\eqref{def:Delta123}-\eqref{def:Gamma12}, respectively. 
 Suppose Assumptions~\ref{assm:xiL2}-\ref{assm:FiltrationControl} and the following contraction condition hold: 
\begin{equation} 
 C_4 \exp \big( T M^A_T \| B \| C_1 (\| B \| + R_3 ) \big) < 1 , \notag  
\end{equation} 
where 
\begin{align} 
&  C_4 = T M^A_T \Big[ M^A_T \| B \|^2 \big( T C_2 + \| G \| \big\| \widehat{F}_2 \big\| \big)
  \exp(M^A_T T C_3 ) 
   + C_1 R_2 \| B \| \| F_2 \| + \|F_1 \|  
  \Big], 
 \notag \\ 
 & C_2 = C_1 \big( R_1 \| F_2 \| + C_1 R_6 R_2 \| F_2 \| + \| F_1 \| \big) 
 + \| M \| \big\| \widehat{F}_1 \big\|,
 \notag \\ 
 & C_3 = C_1 R_6 \| B \|, 
 \notag \\ 
 & C_1 = C^{\Pi}, 
 \notag 
\end{align}
with $C^{\Pi}$ defined in~\eqref{|Pi|<=CPi}.  
Then, the the MFG equilibrium strategy is given by
\begin{align} 
% \hspace{.5em}
\label{u_MFG_eqm}
\begin{aligned} 
 u(t) & =  - K^{-1}(T-t) 
 % \cdot  \\  
 % & \hspace{1.5cm} 
 \big[ L(T-t) x(t) 
  + \Gamma_2 \big( ( F_2 \bar{x}(t) + \sigma )^\star \Pi(T-t) \big) 
   + B^\star q(T-t) \big], 
\\
K(t) & = I + \Delta_3(\Pi(t)), \quad L(t) = B^\star \Pi(t) 
+ \Delta_1(\Pi(t)) ,  
\end{aligned}
\end{align} 
where the mean field 
$\bar{x}(t) \in H$, the operator $\Pi(t) \in \mc{L}(H)$ and the offset term $q(t)\in H$, are characterized by the unique fixed point of the following set of consistency equations 
\allowdisplaybreaks 
\begin{align} 
\label{ODE_Pi} 
 &\frac{d}{dt} \langle \Pi(t)x, x \rangle 
 =  2 \langle \Pi(t)x, A x \rangle - \langle L^\star(t) (K(t))^{-1}L(t)x, x \rangle 
 + \langle \Delta_2(\Pi(t))x, x \rangle  + \langle Mx, x \rangle, 
 \,\,\, x \in \mathcal{D}(A), 
\\
\label{ODE_q_MFG}
&\dot{q}(t) =  \big( A^\star - L^\star(t) (K(t))^{-1} B^\star \big) q(t) 
 + \Gamma_1 \Big( \big( F_2 \bar{x}(T-t) + \sigma \big)^\star \Pi(t) \Big)    
 \\ 
& \hspace{3cm}- L^\star(t) K^{-1}(t) \Gamma_2 \left( \big( F_2 \bar{x}(T-t) 
 + \sigma \big)^\star \Pi(t) \right) 
 + \big( \Pi(t) F_1  - M \hat{F}_1 \big) \bar{x} (T - t),
  \notag\\ 
  &\bar{x}(t) =  S(t) \bar{\xi} - \int_0^t S(t-r) 
 \Big[  
 B (K(T-r))^{-1}   
 \big( L(T-r) \bar{x}(r) + B^\star q(T-r) 
  \label{eq:barx_MFG}\\ 
&\hspace{3cm}  + \Gamma_2\big( ( F_2 \bar{x}(r) + \sigma )^\star \Pi(T-r)  \big) 
 \big)
  - F_1 \bar{x}(r) 
 \Big]  dr, \notag
\end{align} 
with $\Pi(0) = G$, $q(0) =  -G F_2 \bar{x}(T)$.
\end{theorem} 
The equilibrium state of the MFG system \eqref{dx_LQ_MFG}-\eqref{Jinfty} under the strategy~\eqref{u_MFG_eqm} is given by 
\begin{align} 
\begin{aligned} 
 x(t) = & S(t) \xi - \int_0^t S(t-r) \big[ B (K(T-r))^{-1} L(T-r) x(t) 
  + B \tau(r) - F_1 \bar{x}(r)  \big] dr 
  \\ 
 & + \int_0^t S(t-r) \Big[ \big( D - E (K(T-r))^{-1} L(T-r) \big) x(r) 
 - E \tau(r) + F_2 \bar{x}(r) + \sigma \Big] d W(r) ,  
\end{aligned} 
\label{x_MFG_eqm}
\end{align} 
where 
\begin{align} 
 \tau(t) = (K(T-r))^{-1} \big[ B^\star q(T-t) + \Gamma_2 \big( (F_2 \bar{x}(t) + \sigma )^\star \Pi(T-t) \big) \big] . 
 \label{pi}
\end{align} 
\subsection{Estimates for continuous mappings}\label{sec:lem:|q|and|barx|bound}
\input{proof_regularity}

\section{A Brief Overview of Structures over Hilbert Spaces}
\label{s:HilbertStructures}

This section provides a brief overview of some functional analytic background to help keep our results as self-contained as possible. 

\subsection{Output Space: Bochner-Lebesgue Spaces}
\label{s:HilbertStructures_ss:BLTensor}
Next, we are interested in working with the Bochner-Lebesgue space $\mathcal{M}^2(\mathcal{T}, U )=L^2(\operatorname{Prog},U)$ of all $U$-valued strongly-measurable and square-integrable functions; where $\operatorname{Prog}$ is the $\sigma$-algebra of progressively measurable processes.  
The space $\mathcal{M}^2(\mathcal{T}, U)$ is a Hilbert space with inner-product defined for any $f,g\in \mathcal{M}^2(\mathcal{T}, U)$ by 
$
\langle f,g\rangle_{\mathcal{M}^2(\mathcal{T}, U) }
\eqdef \mathbb{E}\left[ \int_{t=0}^T\, \langle f(t),g(t)\rangle_{U}\,dt\right]
$. It is easy, but useful, to note that if $f\in \Lt$ is continuous and has finite supremum norm then
\begin{equation}
\label{eq:sup_to_L2}
        \tfrac{1}{T}\,
            \|f\|_{ \mathcal{M}^2(\mathcal{T}, U)}^2
    = 
        \tfrac{1} {T}\, 
            \mathbb{E}\biggl[\int_0^T\,
                \|f(t)\|_U^2 
                \,dt
            \biggr]
    \le 
    %     \tfrac{1}{T}\,
    %         \sup_{0\le t\le T}\|f(t)\|^2
    %         \,
    %             \int_0^T\,
    %             1
    %             \,dt
    % = 
        \sup_{0\le t\le T}\|f(t)\|_U^2
\end{equation}
where $\|f\|_{ \mathcal{M}^2(\mathcal{T}, U) }$ is the natural norm on the Bochner-Lebesgue space $\mathcal{M}^2(\mathcal{T}, U)$.

With the obvious modifications of~\cite[Example 2.19]{RyanTensor_2002} (which treats the $L^1$ and not $L^2$ case as an example) the extension $\mathfrak{I}$ of the linear map $\tilde{\mathfrak{I}}:
L^2(\operatorname{Prog})
% L^2(\mathcal{T})
\otimes U \to \Lt$ defined on elementary tensors $f\otimes y$ by $\tilde{\mathfrak{I}}(f\otimes y)\eqdef f(\cdot)y$ defines a (linear) isomorphism of Hilbert spaces.  
As discussed in~\cite[Chapter 2.6]{kadison1986fundamentalsOlAlgs}, an orthonormal basis for a tensor product of Hilbert spaces is given by the tensor product of the basis elements of its factors, up to normalizing constants.  Consequently, when passing through $\mathfrak{I}$ they yield an orthonormal basis of $\Lt$.

\subsection{Hilbert-Schmidt Operators Between Different Spaces}
\label{s:HilbertStructures_ss:HS___sss:DifferentSpaces}
Recall the square-summable sequence space $\ell^2\eqdef \{x\eqdef (x_n)_{n=0}^{\infty},\, x_n \in \mathbb{R}:\, \langle x,x\rangle_{\ell^2} <\infty\}$ where 
$\langle x,y\rangle_{\ell^2}\eqdef \sum_{n=0}^{\infty}\, x_ny_n$ for any $x
\eqdef (x_n)_{n=0},y\eqdef (y_n)_{n=0}\in \mathbb{R}^{\mathbb{N}}$. 
Fix separable Hilbert spaces $\mathcal{Y}$ and $\mathcal{Z}$ 
which we, without loss of generality, assume are infinite-dimensional (if not, one may embed them in larger infinite-dimensional spaces and all operators become finite-rank and therefore trivially Hilbert-Schmidt).
 Then, there exists isometric (linear) isomorphisms $\phi: \mathcal{Y} \hookrightarrow \ell^2$ and $\psi: \mathcal{Z} \hookrightarrow \ell^2$; which can, for instance, be constructed by fixing an orthonormal basis $(y_n)_{n=0}^{\infty}$ of $\mathcal{Y}$ and  then defining
\[
\begin{aligned}
    \phi: \mathcal{Y} & \rightarrow \ell^2
\\
    y & \mapsto (\langle y,y_n\rangle_{\mathcal{Y}})_{n=0}^{\infty}
\end{aligned}
\]
where $\langle\cdot,\cdot\rangle_{\mathcal{Y}}$ denotes the inner-product on $\mathcal{Y}$; with a similar (non-canonical) construction for $\psi$.
We define \textbf{subclass} $\mathcal{HS}(\mathcal{Y},\mathcal{Z}) \subseteq \mathcal{L}(\mathcal{Y}, \mathcal{Z})$ consisting of operators that, under this identification, correspond to \textit{Hilbert-Schmidt} operators; that is, $A \in \mathcal{HS}(\mathcal{Y},\mathcal{Z})$ if and only if the conjugated operator $\mathbf{A} \eqdef \psi \circ A \circ \phi^{-1}$ is Hilbert-Schmidt; note, this definition is ``canonical'' in that it is independent of the choice of $\phi$ and of $\psi$ (since all bases are chosen to be orthonormal).
We equip the (vector) space $\mathcal{HS}(\mathcal{Y},\mathcal{Z})$ with the inner product 
\begin{equation}
\label{eq:HS_Beyond}
        \langle A, B \rangle_{
            \mathcal{HS}(\mathcal{Y},\mathcal{Z})
            } 
    \eqdef 
        \sum_{n=0}^\infty \,
            \langle \mathbf{A} \lambda_n, \mathbf{B} \lambda_n \rangle_{\ell^2}
,
\end{equation}
where $\{\lambda_n\}_{n=0}^{\infty}$
and 
for any $n\in \mathbb{N}$, $\lambda_n\eqdef (\lambda_{n:i})_{i=0}^{\infty}$ where $\lambda_i=1$ if $i=n $ and equals to zero otherwise. 
From~\eqref{eq:HS_Beyond}, it is relatively straightforward to see that the norm induced by the above inner product, called the \textit{Hilbert-Schmidt} norm and denoted by $\|\cdot\|_{\mathcal{HS}(\mathcal{Y},\mathcal{Z})}$, dominates the operator norm%
% \footnote{Simply considering unit vectors in {$\mathcal{Y}$}.}%
\begin{equation}
\label{eq:op_to_HS}
    \|A\|_{op}
\le 
    \|A\|_{\mathcal{HS}(\mathcal{Y},\mathcal{Z})}
.
\end{equation}
%%%
%%%
%%%
The linear space $\mathcal{HS}(\mathcal{Y},\mathcal{Z})$ is a Hilbert space with inner-product in~\eqref{eq:HS_Beyond} and with orthonormal basis $\{E_{(i,j)}^{\mathcal{Y}\to\mathcal{Z}}\}_{(i,j)\in \mathbb{N}^2}$ given by for each $(i,j)\in \mathbb{N}^2$
and every $y\in \mathcal{Y}$ by
\begin{equation}
\label{eq:ONB_HS}
        E_{(i,j)}^{\mathcal{Y}\to\mathcal{Z}}(y)
    \eqdef 
        \langle y , y_i \rangle_{\mathcal{Y}}\, z_j
\end{equation}
where $(z_j)_{j\in \mathbb{N}}$ is choice of an orthonormal basis of $\mathcal{Z}$.
%%%
\begin{example}
\label{ex:Iterated_Construction}
Let $\mathcal{X}$ be an infinite-dimensional separable Hilbert space with inner-product $\langle\cdot,\cdot\rangle_{\mathcal{X}}$ and orthonormal basis $(x_k)_{k=0}^{\infty}$, and consider the orthonormal basis of $\mathcal{HS}(\mathcal{Y},\mathcal{Z})$ defined in~\eqref{eq:ONB_HS}.  
%%%
Then, we obtain an orthonormal basis of $\mathcal{HS}(\mathcal{X},
\mathcal{HS}(\mathcal{Y},\mathcal{Z}))
$ by iterating the construction in~\eqref{eq:ONB_HS}.  
Namely, let 
$\{
E_{(k,i,j)}^{
    \mathcal{X}\to\mathcal{HS}(\mathcal{Y},\mathcal{Z})
}
\}_{(k,i,j)\in \mathbb{N}^3}$ be defined for each $(k,i,j)\in\mathbb{N}^3$ as mapping any $x\in \mathcal{X}$ to the 
Hilbert-Schmidt operator in $\mathcal{HS}(\mathcal{Y},\mathcal{Z})$ given by
\[
    E_{(k,i,j)}^{
        \mathcal{X}\to\mathcal{HS}(\mathcal{Y},\mathcal{Z})
    }
    (x)
\eqdef 
    \langle x, x_k\rangle_{\mathcal{X}}\,   E_{(i,j)}^{\mathcal{Y}\to\mathcal{Z}}
.
\]
\end{example}

\subsection{Direct Sums of Spaces of Hilbert Spaces}
\label{s:HilbertStructures_ss:DirectSums}

Fix an integer $I\ge 2$ and let $\mathcal{X}_1,\dots,\mathcal{X}_I$ be separable infinite-dimensional Hilbert spaces.  For each $i=1,\dots,I$, let $\{x^{(i)}_n\}_{n=0}^{\infty}$ be an orthonormal basis of $\mathcal{X}^{(i)}$ and let $\langle \cdot, \cdot \rangle_{\mathcal{X}^{(i)}}$ denote the inner-product thereon.  Then, the direct sum of these Hilbert spaces, denoted by $\bigoplus_{i=1}^I\,\mathcal{X}^{(i)}$ or by $\mathcal{X}^{(1)}\oplus \dots\oplus \mathcal{X}^{(I)}$, is the vector space whose underlying set is the Cartesian product $\prod_{i=1}^I\, \mathcal{X}^{(i)}$ and whose inner-product 
$\langle \cdot,\cdot\rangle_{\bigoplus_{i=1}^I\,\mathcal{X}^{(i)}}$
is defined by
\begin{equation}
\label{eq:directprod_inner}
        \langle 
            (x^{i})_{i=1}^I,(\tilde{x}^{i})_{i=1}^I
        \rangle_{\bigoplus_{i=1}^I\,\mathcal{X}^{(i)}}
    \eqdef 
        \sum_{i=1}^I\,
        \langle x^{i},\tilde{x}^{i}\rangle_{\mathcal{X}^{(i)}}
\end{equation}
for any $(x^{i})_{i=1}^I,(\tilde{x}^{i})_{i=1}^I\in \bigoplus_{i=1}^I\,\mathcal{X}^{(i)}$.  Moreover, an orthonormal basis of $\bigoplus_{i=1}^I\,\mathcal{X}^{(i)}$ is given by
\begin{equation}
\label{eq:direct_sum}
    \bigcup_{j=1}^I
    \,
    \big\{
        (\delta_{j=i} x^{(i)}_n)_{j=1}^I
    \big\}_{n=1}^{\infty}
.
\end{equation}

\end{document}

%% file: Comments.tex
\definecolor{britishracinggreen}{rgb}{0.0, 0.26, 0.15}
\definecolor{darkcyan}{rgb}{0.0, 0.55, 0.55}
\definecolor{MidnightBlue}{RGB}{25,25,112}
\definecolor{MidnightBlueComplementingGreen}{RGB}{25,112,25}
\definecolor{MidnightBlueComplementingPurple}{RGB}{112,25,112}
\definecolor{MidnightBlueComplementingRed}{RGB}{112,25,69}
\definecolor{WowColor}{rgb}{.75,0,.75}
\definecolor{MildlyAlarming}{rgb}{0.85,0.25,0.1}
\definecolor{SubtleColor}{rgb}{0,0,.50}
\definecolor{antiquefuchsia}{rgb}{0.57, 0.36, 0.51}
\definecolor{fashionfuchsia}{rgb}{0.96, 0.0, 0.63}
\definecolor{jade}{rgb}{0.0, 0.66, 0.42}
\definecolor{caribbeangreen}{rgb}{0.0, 0.8, 0.6}
\definecolor{aquamarine}{rgb}{0.5, 0.8, 0.85}
\definecolor{lightseagreen}{rgb}{0.13, 0.7, 0.67}
\definecolor{darkgreen}{rgb}{0.0, 0.2, 0.13}
\definecolor{darkspringgreen}{rgb}{0.09, 0.45, 0.27}
\definecolor{attentioncolor}{RGB}{152,90,81}
\definecolor{burgred}{RGB}{40,3,22}
\definecolor{AnnieGreen}{RGB}{17,123,92}
\definecolor{Turquoise}{RGB}{64,224,208}

\definecolor{darkjade}{RGB}{0,122,84}
\definecolor{Window1}{RGB}{92,150,31}%
    \definecolor{Window1dark}{RGB}{41,67,13}%
\definecolor{Window2}{RGB}{255,168,28}
    \definecolor{Window2dark}{RGB}{114,75,12}
\definecolor{Window3}{RGB}{255,96,33}
    \definecolor{Window3dark}{RGB}{97,36,12}
\definecolor{InputColor}{RGB}{20,255,177}
    \definecolor{InputColorlight}{RGB}{222,237,229}

\definecolor{RedAlizarin}{rgb}{0.82, 0.1, 0.26}
\usepackage{soul}

\usepackage[colorinlistoftodos]{todonotes}

\NewDocumentCommand{\Dena}{mo}{
    \IfValueF{#2}{
                        {{
                            \textcolor{magenta}{ 
                            [\textbf{Dena:}
                            \textit{{#1}}]
                            }
                        }}
        }
    \IfValueT{#2}{
                        \marginnote{{\scriptsize
                            \textcolor{magenta}{ 
                            \textbf{T:}
                            \textit{{#1}}
                            }
                        }}
        }
                    }

\NewDocumentCommand{\AK}{mo}{
    \IfValueF{#2}{\hfill\\
                        {{
                            \textcolor{darkmidnightblue}{ 
                            \textbf{AK:}
                            \textit{{#1}}
                            }
                        }}
        }
    \IfValueT{#2}{
                        \marginnote{{\scriptsize
                            \textcolor{darkmidnightblue}{ 
                            \textbf{AK:}
                            \textit{{#1}}
                            }
                        }}
        }
                    }

\NewDocumentCommand{\Xuwei}{mo}{
    \IfValueF{#2}{
                        {{
                            \textcolor{brown}{ 
                            [\textbf{Xuwei:}
                            \textit{{#1}}]
                            }
                        }}
        }
    \IfValueT{#2}{
                        \marginnote{{\scriptsize
                            \textcolor{brown}{ 
                            \textbf{T:}
                            \textit{{#1}}
                            }
                        }}
        }
                    }

%% file: proof_regularity.tex
\begin{lemma} 
\label{lem:FBGronwall} 
Suppose $f\in C(\mathcal{T}; H)$ and $g\in C(\mathcal{T}; H)$ satisfy the following system of inequalities on $\mathcal{T}$   
\begin{align} 
& \|f(t)\| \leq  \int_0^t \big( a_f\|f(r)\|  + b_f \|g(T-r)\| + c_f \big) dr + d_f  , 
\label{|f(r)|<=int|f|+|g(T-r)|}
\\ 
& \|g(t)\| \leq d_{g,f} \|f(T)\| + d_g +  \int_0^t \big( a_g \|g(r)\|  + b_g \|f(T-r)\| + c_g \big) dr  , 
\label{|g(r)|<=int|g|+|f(T-r)|}
\end{align} 
where $(a_f, b_f, c_f, d_f)$ and $(a_g, b_g, c_g, d_g, d_{g,f})$ are positive constants and satisfy
\begin{align} 
 (b_g T + d_{g, f} ) b_f T \exp((a_f + a_g)T) <1 . 
 \label{bfbfT2exp((af+ag)T)<1}
\end{align} 
Then $f$ and $g$ have the estimates 
\begin{align} 
\label{|f|<=bdd}
 \| f \|_{C(\mathcal{T}; H)} \leq & 
  \big[ 1 - (b_g T + d_{g,f} ) b_f T \exp((a_f + a_g)T)  \big]^{-1} 
   \\ 
&  \times \big[ ( c_f T + d_f )  \exp(a_f T)  
 +  b_f T \big(  c_g T + d_g \big) \exp\big( (a_f+a_g ) T \big)     \big],  
\notag  \\
%%%%%% 
 \|g\|_{C(\mathcal{T}; H)} 
 \leq & (c_g T + d_g) \exp(a_g T) 
  \label{|g|<=bdd}  \\ 
& + ( b_g T + d_{g, f} ) \exp(a_g T) 
 \big[ 1 - (b_g T + d_{g,f} ) b_f T \exp((a_f + a_g)T)  \big]^{-1}  
\notag \\ 
&  \times \big[ ( c_f T + d_f )  \exp(a_f T)  
 +  b_f T \big(  c_g T + d_g \big) \exp\big( (a_f+a_g ) T \big)     \big] . 
\notag 
\end{align} 
\end{lemma} 

\begin{proof} 
From~\eqref{|f(r)|<=int|f|+|g(T-r)|} and~\eqref{|g(r)|<=int|g|+|f(T-r)|}, we have 
\begin{align} 
& \|f(t)\| \leq  \int_0^t a_f \|f(r)\| dr + b_f t \|g\|_{C(\mathcal{T}; H)} + c_f t + d_f , 
\label{|f|<=int|f|+|g|}\\ 
& \|g(t)\| \leq  \int_0^t a_g \|g(r)\| dr + ( b_g t + d_{g,f} ) \|f\|_{C(\mathcal{T}; H)} + c_g t + d_g . 
\label{|g|<=int|g|+|f|}
\end{align} 
With the Gr\"{o}nwall's inequality applied to~\eqref{|f|<=int|f|+|g|} and~\eqref{|g|<=int|g|+|f|}, we have  
\begin{align} 
& \|f\|_{C(\mathcal{T}; H)} \leq \big( b_f T \|g\|_{C(\mathcal{T}; H)} + c_f T + d_f \big) \exp(a_f T) , 
\label{|f|<=|g|} \\ 
& \|g\|_{C(\mathcal{T}; H)} \leq \big[ (b_g T + d_{g,f} ) \|f\|_{C(\mathcal{T}; H)} + c_g T + d_g \big] \exp(a_g T)  . 
\label{|g|<=|f|} 
\end{align} 
Substituting~\eqref{|g|<=|f|} into~\eqref{|f|<=|g|}, we have 
\begin{align} 
 \|f\|_{C(\mathcal{T}; H)} \leq &   ( c_f T + d_f )  \exp(a_f T)  
 +  b_f T \big[ (b_g T + d_{g,f} ) \|f\|_{C(\mathcal{T}; H)} + c_g T + d_g \big] \exp(a_g T)   \exp(a_f T)  
 \notag 
\end{align} 
and furthermore 
\begin{align} 
 & \|f\|_{C(\mathcal{T}; H)} \big[ 1 - (b_g T + d_{g, f} ) b_f T \exp((a_f + a_g)T)  \big] 
 \notag \\ 
 \leq  &
 ( c_f T + d_f )  \exp(a_f T)  
 +  b_f T \big(  c_g T + d_g \big) \exp\big( (a_f+a_g ) T \big) .    
 \notag 
\end{align} 
Since~\eqref{bfbfT2exp((af+ag)T)<1} is satisfied, we immediately have the upper bound~\eqref{|f|<=bdd} for $f$.  
Substituting~\eqref{|f|<=bdd} into~\eqref{|g|<=|f|}, we obtain the estimate~\eqref{|g|<=bdd} for $g$. 
\end{proof} 
\subsection{Proof of Regularity Results for the Reference MFG Model}
\subsubsection{Proof of Lemma~\ref{lem:|q|and|barx|bound}} \label{sec:proof:lem:|q|and|barx|bound}
\begin{proof}
Denote 
\allowdisplaybreaks 
\begin{align} 
 \Phi^\dagger(r) 
 = &   
 B^\dagger (K^\dagger)^{-1}(T-r) 
  \big[ L^\dagger(T-r) \bar{x}^\dagger(r) + (B^\dagger)^\star q^\dagger(T-r) 
+ \Gamma_2\big( ( F_2^\dagger \bar{x}^\dagger(r) + \sigma )^\star \Pi^\dagger(T-r)  \big) 
 \big] 
 \notag \\ 
&  - F_1 \bar{x}^\dagger(r)  , \label{Phi-ref}
\\ 
 %%%
 \Psi^\dagger(r) = &  
 - (L^\dagger)^{\star}(r) (K^\dagger)^{-1}(r) (B^\dagger)^\star q(r) 
  + \Gamma_1\big((F_2^\dagger \bar{x}^\dagger(T-r)+\sigma)^\star \Pi^\dagger(r) \big) 
\notag \\ 
& \quad - (L^\dagger)^\star(r) (K^\dagger)^{-1}(r)\Gamma_2\big(( F_2^\dagger 
 \bar{x}^\dagger(T-r)+\sigma)^\star \Pi^\dagger(r)   \big) 
%\notag \\
%& \hspace{4cm} 
+  \big(\Pi^\dagger(r)F_1- M\widehat{F}_1 \big)\bar{x}^\dagger(T-r).\label{Psi-ref} 
\end{align} 
By~\eqref{ODE_q_MFG} and~\eqref{eq:barx_MFG}, the terms $q^\dagger \in \mathcal{H}$ and 
the limiting mean field terms $\bar{x}^\dagger \in \mathcal{H}$ have the following integral forms 
\begin{align} 
 \bar{x}^\dagger(t) = & S^\dagger(t) \bar{\xi} 
 - \int_0^t S^\dagger(t-r) 
 \Phi^\dagger(r) dr , 
 \label{eq:barxi_MFG_int} 
 \\  
 q^\dagger(t) = & - S^\dagger(t) G\widehat{F}_2 \bar{x}^\dagger(T)   
 + \int_0^t S^\dagger(t-r) \Psi^\dagger(r) dr .  
\label{eq:qi_MFG_int}
\end{align} 

\begin{align} 
& \big\| \Phi^\dagger(r) \big\| \leq C^{\Phi, \bar{x}^\dagger } \| \bar{x}^\dagger(r) \|   
+ C^{\Phi, q^\dagger }  \| q^\dagger(T-r) \| 
+ C^{\Phi, c, \dagger} ,  
\label{|Phi|<=barxq}  \\ 
& \big\| \Psi^\dagger(r) \big\| \leq 
 C^{\Psi, q^\dagger }  \| q^\dagger(r) \| 
 + C^{\Psi, \bar{x}^\dagger } \| \bar{x}^\dagger(T-r) \|    
+ C^{\Psi, c, \dagger } ,  
\label{eq:|Psi|<=barxq}
\end{align} 
where 
\begin{align} 
& C^{\Phi, \bar{x}^\dagger } =  \| B^\dagger \| ( \|B^\dagger \| + R_3 + R_2 \|F_2^\dagger \| ) C^{\Pi^\dagger}   + \|F_1\|  , 
 \notag \\ 
 &  C^{\Phi, q^\dagger } = \| B^\dagger \|^2 , 
 \quad 
  C^{\Phi, c, \dagger } = \|B^\dagger \| R_2 \|\sigma\| C^{\Pi^\dagger} ,  
 \notag \\ 
%\end{align} 
%\begin{align} 
& C^{\Psi, q^\dagger } = C^{\Pi^\dagger } (\| B^\dagger \| + R_3 ) 
  \big\| B^\dagger \big\| , 
  \notag \\ 
& C^{\Psi, \bar{x}^\dagger } 
 = \big[ R_1 C^{\Pi^\dagger} + ( C^{\Pi^\dagger} )^2 (\| B^\dagger \| + R_3 ) R_2 \big] \big\| F_2^\dagger \big\| 
  + C^{\Pi^\dagger} \big\| F_1 \big\|  
 + \big\| M \big\| \big\|  \hat{F}_1 \big\| ,    
\notag \\ 
& C^{\Psi, c, \dagger } =   R_1 C^{\Pi^\dagger } + (C^{\Pi^\dagger})^2 (\|B^\dagger\| + R_3 ) R_2  . 
\notag 
\end{align} 
By~\eqref{eq:barxi_MFG_int}, \eqref{eq:qi_MFG_int}, \eqref{|Phi|<=barxq} and \eqref{eq:|Psi|<=barxq}, we have 
\begin{align} 
 \|\bar{x}^\dagger(t) \| \leq & 
 \| S^\dagger(t) \|  \| \bar{\xi} \| 
 + \int_0^t \| S^\dagger(t-r) \|   \big\| \Phi^\dagger(r)  \big\| dr
\notag  \allowdisplaybreaks\\ 
 \leq & M^{A^\dagger}_T  \| \bar{\xi} \| 
 + M^{A^\dagger}_T \int_0^t \big( C^{\Phi, \bar{x}^\dagger } \| \bar{x}^\dagger(r) \|   
+ C^{\Phi, q^\dagger }  \| q^\dagger(T-r) \| 
+ C^{\Phi, c, \dagger }   \big) dr 
  \label{|barx|<=int|barx|+|q(T-r)|} 
\end{align} 
and 
\begin{align} 
 \big\| q^\dagger(t) \big\| 
  \leq &  \big\| S^\dagger(t) \big\|  \| G \|  \| \widehat{F}_2 \|  
  \| \bar{x}^\dagger(T) \|    
 + \int_0^t \big\| S^\dagger(t-r) \big\|  \big\| \Psi^\dagger(r) \big\|  dr  
 \notag  \allowdisplaybreaks\\ 
 \leq & M^{A^\dagger}_T  \| G \|  \| \widehat{F}_2 \|  
  \| \bar{x}^\dagger(T) \|     
  + M^{A^\dagger}_T   
   \int_0^t  C^{\Psi, q^\dagger }  \| q^\dagger(r) \| 
 + C^{\Psi, \bar{x}^\dagger } \|\bar{x}^\dagger(T-r) \|   
+ C^{\Psi, c, \dagger } dr  . 
 \label{|q|<=int|q|+|barx(T-r)|} 
\end{align} 
Applying Lemma~\ref{lem:FBGronwall} to \eqref{|barx|<=int|barx|+|q(T-r)|} and~\eqref{|q|<=int|q|+|barx(T-r)|}, we have 
\eqref{|barx|<=Cbarx} and~\eqref{|q|<=Cq}. 
\end{proof}

\subsubsection{Proof of Lemma~\ref{lem:E|x|<=Cx}}\label{sec:proof:lem:E|x|<=Cx}
\begin{proof}
By~\eqref{x_MFG_eqm}, we have 
\allowdisplaybreaks
\begin{align} 
 & \mathbb{E} \|x^\dagger(t)\|^2 
 \notag \allowdisplaybreaks\\ 
 \leq 
 & 3 \| S^\dagger(t) \|^2   \mathbb{E}\|\xi\|^2 
 + 3 \mathbb{E} \Big\| \int_0^t S^\dagger(t-r) \big[ B^\dagger (K^\dagger)^{-1}(T-r) 
  L^\dagger(T-r) x^\dagger(r)  
 \notag \allowdisplaybreaks\\ 
& \hspace{7cm} + B^\dagger \tau^\dagger(r) - F_1 \bar{x}^\dagger(r)  \big] dr \Big\|^2
 \notag \allowdisplaybreaks\\ 
 & + 3 \mathbb{E} \Big\| \int_0^t S^\dagger(t-r) \big[ \big( D - E (K^\dagger)^{-1}(T-r) 
  L^\dagger(T-r) \big) x^\dagger(r) 
 %\notag \\ 
 %& \hspace{4.2cm} 
 - E \tau^\dagger(r) + F_2^\dagger \bar{x}^\dagger(r) + \sigma \big] d W(r) 
 \Big\|^2
 \notag \allowdisplaybreaks\\ 
 %%% 
 \leq & 3 ( M^{A^\dagger}_T )^2  \mathbb{E}\|\xi\|^2 
 + 3  \int_0^t ( M^{A^\dagger}_T )^2 \big[ \| B^\dagger \|^2  \| L^\dagger(T-r) \|^2 \mathbb{E} \| x^\dagger(t) \|^2 
% \notag \\ 
%& \hspace{5cm} 
+ \| B^\dagger \|^2  \| \tau^\dagger(r) \|^2 
  + \| F_1 \|^2  \| \bar{x}^\dagger(r) \|^2  \big] dr 
 \notag \allowdisplaybreaks\\ 
 & + 3 \int_0^t (M^{A^\dagger}_T)^2 
 \big[ \big( \| D \|^2 + \|E\|^2  \| L^\dagger(T-r) \|^2  \big) \mathbb{E}\|x^\dagger(r)\|^2 
 %\notag \\ 
 %& \hspace{5cm}  
 + \|E\|^2  \|\tau^\dagger(r) \|^2  + \|F_2^\dagger \|^2  \| \bar{x}^\dagger \|^2
  + \|\sigma\|^2
 \big] dr . 
 \notag 
\end{align} 
By~\eqref{pi}, we have 
\begin{align} 
\|  \tau(t) \|  \leq & \| K^{-1}(T-t) \| \big[ \| B \| | q(T-t) | 
+ \| \Gamma_2 \| \big( ( \|F_2^\dagger \|  |\bar{x}^\dagger(t)| + \|\sigma\| ) 
 \| \Pi^\dagger(T-t) \| \big) \big]  
 \notag \allowdisplaybreaks\\ 
 \leq &   \| B^\dagger \| C^{q^\dagger}  
+ R_2 ( \|F_2^\dagger \|  C^{\bar{x}^\dagger } + \|\sigma\| ) C^{\Pi^\dagger} . 
\notag 
\end{align} 
By the Gr\"{o}nwall's inequality, we have 
\begin{align} 
 \mathbb{E} \| x^\dagger(t) \|^2 
 \leq & 
 \big\{  3(M^{A^\dagger}_T)^2 \mathbb{E} \|\xi\|^2 
 + 3 T (M_T^{A^\dagger} )^2  
   \big[ ( \| B^\dagger \|^2 + \| E \|^2 ) \big(  \| B^\dagger \| C^{q^\dagger}  
+ R_2 ( \|F_2^\dagger\|  C^{\bar{x}^\dagger} + \|\sigma\| ) C^{\Pi^\dagger} \big)  \big] 
\notag \allowdisplaybreaks\\ 
& \hspace{5cm} + 3 T(M^{A^\dagger}_T)^2 ( \| F_1 \|^2 + \| F_2^\dagger \|^2 ) 
 ( C^{\bar{x}^\dagger} )^2 
 + 3T(M^{A^\dagger}_T)^2 \| \sigma \|^2 
 \big\}\cdot  \notag \allowdisplaybreaks\\ 
& \exp\big\{ 3(M^{A^\dagger}_T)^2 \big[ ( \| B^\dagger \|^2 + \| E \|^2 ) \big( (\|B^\dagger \| + R_3 ) C^{\Pi^\dagger} \big)^2 + \| D \|^2 \big] T  
\big\} .  
 \notag 
\end{align} 
The desired estimate then follows.  
\end{proof} 

\subsection{Proofs of Lipschitz Stability with respect to \texorpdfstring{$A$}{A}} 
\label{sec:Proof_Stability_A}
\subsubsection{Proof of Lemma \ref{lem:A:|A1-A2|}}\label{sec:lem:A:|A1-A2|}
\begin{proof}
Denote $K=A-A^\dagger$. According to \cite{goldstein2017semigroups}, we have 
\begin{align}
 \frac{d}{dt}S^A(t)&=AS^A(t) = (K+A^\dagger)S^A(t),&S^A(0)&= I,\\  \frac{d}{dt}S^{A^\dagger}(t)&=A^\dagger S^{A^\dagger}(t),  &S^{A^\dagger}(0)&= I.
\end{align}
Then, we obtain 
 \begin{equation}
 \frac{d}{dt}(S^{A}(t)-S^{A^\dagger}(t))=A^\dagger (S^{A}(t)+S^{A^\dagger}(t))-KS^A(t),  
 \end{equation}
which results in 
\begin{equation}
S^{A^\dagger}(t)-S^{A}(t)=\int_0^t S^{A^\dagger}(t-s)KS^A(s)ds.
\end{equation}
From the above equation, we conclude that
\begin{equation}
\Vert S^{A^\dagger}(t)-S^{A}(t)\Vert\leq T M^{A^\dagger}_TM^{A}_T\Vert A-A^\dagger \Vert.
\end{equation}
\end{proof}
\subsubsection{{Proof of Lemma~\ref{lem:A:|Pi1-Pi2|}}}\label{sec:lem:A:|Pi1-Pi2|}
\begin{proof} 
The proof is based on \cite[Proposition 4.3]{LF24}. 
We introduce the processes  
\begin{equation} 
\begin{aligned}
 & d y^A(t) = (A y^A(t) + B^\dagger u(t)) dt + (D y^A(t) + E u(t)) d W(t) , \quad  y^A(0) = \xi ,  \\ 
 & d y^\dagger(t) = (A^\dagger y^\dagger(t) + B^\dagger u(t)) dt + (D y^\dagger(t) + E u(t)) d W(t) , 
 \quad  y^\dagger(0) = \xi , 
 \end{aligned} 
 \label{dy:A&Adagger}
 \end{equation} 
and cost functionals 
\begin{align} 
 & \mathbb{E}\int_0^T \big( \big\| M^{1/2} y^A(t) \big\|^2 + \|u(t)\|^2 \big) dt 
   + \mathbb{E} \big\| G^{1/2} y^A(T) \big\|^2 ,  
\notag \\ 
& \mathbb{E}\int_0^T \big( \big\| M^{1/2} y^\dagger(t) \big\|^2 + \|u(t)\|^2 \big) dt 
   + \mathbb{E} \big\| G^{1/2} y^\dagger(T) \big\|^2 . 
\notag 
\end{align} 

Let $S^\dagger\in \mathcal{L}(H)$ and $S^A\in \mathcal{L}(H)$ be the semigroup corresponding to the infinite generator $A^\dagger$ and $A$, respectively. The mild solutions of~\eqref{dy:A&Adagger} are given by 
\begin{align} 
 y^A(t) = & S^A(t) \xi + \int_0^t S^A(t-r) B^\dagger u(r) dr   
  + \int_0^t S^A(t-r) (D y^A(r) + E u(r)) d W(r) ,  
 \notag \\ 
  y^\dagger(t) = & S^\dagger(t) \xi + \int_0^t S^\dagger(t-r) B^\dagger u(r) dr   
  + \int_0^t S^\dagger(t-r) (D y^\dagger(r) + E u(r)) d W(r) ,  
 \notag 
\end{align} 
and satisfy 
\allowdisplaybreaks
\begin{align} 
  y^A(t) - y^\dagger(t) 
 = & (S^A - S^\dagger)(t) \xi 
  + \int_0^t (S^A - S^\dagger)(t-r) B^\dagger u(r) dr  \notag \\ 
 & + \int_0^t \big[ (S^A - S^\dagger)(t-r) (D y^A(r) + E u(r))   
  + S^\dagger(t-r)D(y^A(r) -  y^\dagger(r) ) \big]  d W(r) . 
 \notag 
\end{align} 
By the It\^{o}'s formula and Cauchy-Schwarz inequality, we have 
\begin{align} 
   \mathbb{E} \big( \| ( y^A - y^\dagger )(t)\|^2 \big)   
 = &
 \mathbb{E} \int_0^t 2 \langle (y^A - y^\dagger)(r) , (S^A - S^\dagger)(t-r) B^\dagger u(r)  \rangle dr
 \notag \\ 
 & + \mathbb{E}\int_0^t \big\| ( S^A - S^\dagger )(t-r) (D y^A(r) + E u(r)) 
 + S^\dagger(t-r)D( y^A - y^\dagger)(r) \big\|^2 dr
 \notag \\ 
 \leq & \mathbb{E}\int_0^t 2 \big\| ( y^A - y^\dagger)(r) \big\| 
  \big\| ( S^A - S^\dagger )(t-r) \big\|  \| B^\dagger \|  \| u(r)\|  dr \notag \\ 
 & + \mathbb{E} \int_0^t 2 \big\| ( S^A - S^\dagger )(t-r) (D y^A(r) + E u(r))  \big\|^2 d r 
 \notag \\ 
 &  
 + \mathbb{E}\int_0^t 2 \big\| S^\dagger(t-r)D( y^A -  y^\dagger )(r) \big\|^2 dr . 
 \label{CauchySchwarz|y1-y2|}
\end{align}

As in the proof of~\cite[Proposition 4.3]{LF24}, we have the following bounds for $y^A$ and $y^\dagger$
\begin{align} 
\begin{aligned} 
& \mathbb{E} \|y^A(t)\|^2 \leq 2 ( M^A_T )^2 \|\xi\|^2 \exp\big(16T (M^A_T )^2 \| D\|^2 \mathrm{tr}(Q) \big) , \\
& \mathbb{E} \|y^\dagger(t)\|^2 \leq 2 (M_T^{A^\dagger} )^2 \|\xi\|^2 \exp\big(16T (M_T^{A^\dagger} )^2 \| D\|^2 \mathrm{tr}(Q) \big) . 
 \end{aligned} 
 \label{E|yi0|^2bound}
\end{align} 
where each $M^A_T$ and $M^{A^\dagger}_T$ are the upper bounds for $S^A$ and $S^\dagger$ such that 
$\| S (t) \| \leq M^A_T$ and 
$\|S^\dagger(t)\|\leq M^{A^\dagger}_T$, for all $t\in \mathcal{T}$. 

From~\eqref{CauchySchwarz|y1-y2|} with $u(t)=0$ for all $t\in \mathcal{T}$, we have 
\begin{align} 
 & \mathbb{E} \big( \| (y^A - y^\dagger)(t)\|^2 \big)   \notag \\ 
 \leq &   \mathbb{E} \int_0^t 2 \big\| (S^A-S^\dagger)(t-r) D y(r)   \big\|^2 dr  
    + \mathbb{E}\int_0^t 2 \big\| S^\dagger(t-r)D( y^A -  y^\dagger )(r) \big\|^2 dr 
  \notag \\ 
  \leq &  
   \mathbb{E} \int_0^t 2 \big\| (S^A - S^\dagger )(t-r) \big\|^2  \| D \|^2  
    \| y^A(r) \|^2    dr   
%    \notag \\ 
%& \qquad 
+ \mathbb{E}\int_0^t 2 \big\| S^\dagger (t-r) \big\|^2  
   \| D \|^2  \| ( y^A -  y^\dagger )(r)  \|^2 dr  
   \notag \\ 
 \leq &  \mathbb{E} \int_0^t 2 \big\| ( S^A - S^\dagger )(t-r) \big\|^2  \| D \|^2  
    \| y^A(r) \|^2    dr 
        + \mathbb{E}\int_0^t 2 (M_T^{A^\dagger})^2  
   \| D \|^2  \| ( y^A -  y^\dagger )(r) \|^2 dr . 
 \label{CauchySchwarz|y1-y2|u=0}
\end{align} 
By the Gr\"{o}nwall's inequality and \eqref{E|yi0|^2bound}, we further have that   
\begin{align} 
 & \mathbb{E} \big( \|y^A(t) - y^\dagger(t)\|^2 \big) \notag \\ 
 \leq & 
  \mathbb{E} \int_0^t 2 \big\| (S^A - S^\dagger )(t-r) \big\|^2  \| D \|^2  
    \| y^A(r) \|^2    dr \notag \\ 
& + \int_0^t  \Big[ \mathbb{E} \int_0^s 2 \big\| (S^A - S^\dagger )(s-r) \big\|^2  \| D \|^2  
    \| y^A(r) \|^2    dr \Big]  
    %\cdot \notag \\  
%& \hspace{4cm} 
2 (M_T^{A^\dagger} )^2  \| D \|^2 \exp\big( 2 (M_T^{A^\dagger} )^2 \| D \|^2 (t-s) \big) 
 ds 
 \notag \\ 
 \leq & 
 \mathbb{E}\int_0^t 2 \| (S^A - S^\dagger )(t-r) \|^2  \| D \|^2  \|y^A(r)\|^2 dr 
 \notag \\ 
 & - \Big[ \mathbb{E} \int_0^t 2 \big\| ( S^A - S^\dagger )(t-r) \big\|^2  \| D \|^2  
   \| y^A(r) \|^2    dr \Big]    
 \int_0^t 2 ( M_T^{A^\dagger} )^2  \|D\|^2 \exp\big( 2 (M_T^{A^\dagger} )^2 \|D\|^2 (t-s) \big) ds  \notag \\ 
= & 
 \Big[ \mathbb{E}\int_0^t 2 \| ( S^A - S^\dagger )(t-r) \|^2  \| D \|^2  \|y^A(r)\|^2 dr  \Big]   
% \notag \\ 
%& \hspace{4cm} 
\Big[ 1 - 1 + \exp(2 (M_T^{A^\dagger} )^2 \| D \|^2 t ) \Big] 
 \notag \\ 
 \leq & 
     2 \sup_{t\in \mathcal{T}} \| ( S^A - S^\dagger )(t) \|^2 \| D \|^2 T  
    \Big[ 2 (M_T^A)^2 \|\xi\|^2 \exp\big( 16 T (M_T^A )^2 \| D\|^2 \mathrm{tr}(Q) \big) \Big]     
  \exp( 2 (M_T^{A^\dagger} )^2 \| D \|^2 T ). 
\label{E|y10-y20|^2bound} 
\end{align} 
We have from~\eqref{E|yi0|^2bound} and 
\eqref{E|y10-y20|^2bound} that 
\begin{align} 
\begin{aligned}
 &  \big( \sup_{t\in \mathcal{T}} \mathbb{E} \big\| y^A(t) \big\|^2 \big)^{1/2} 
 \leq \sqrt{2} (M_T^A ) \|\xi\| \exp\big( 8 T (M_T^A )^2 \| D\|^2 \mathrm{tr}(Q) \big)  , \\ 
 & \big( \sup_{t\in \mathcal{T}} \mathbb{E} \big\| y^\dagger(t) \big\|^2 \big)^{1/2} 
 \leq \sqrt{2} (M_T^{A^\dagger} ) \|\xi\| \exp\big( 8 T (M_T^{A^\dagger} )^2 \| D\|^2 \mathrm{tr}(Q) \big) ,   
 \end{aligned} 
 \label{E|yi0|bound} 
\end{align} 
and 
\begin{align} 
    \big( \sup_{t\in \mathcal{T}} \mathbb{E} & \big\| y^A(t) - y^\dagger(t) \big\|^2 \big)^{1/2} 
 \notag \\ 
  &  \leq  2  \sup_{t\in \mathcal{T}} \| ( S^A - S^\dagger )(t) \|  \| D \| T^{1/2} 
 %\cdot 
% \notag \\ 
%& 
\big[ M_T^A \|\xi\| \exp\big( 8 T (M_T^A )^2 \| D\|^2 \mathrm{tr}(Q) \big) \big]     
%\notag \\ 
%&    
\exp(  (M_T^{A^\dagger} )^2 \| D \|^2 T )  . 
\label{E|y10-y20|bound}
\end{align}

By the proof of~\cite[Proposition 4.3]{LF24}, we have for any $u(t)$, $t\in \mathcal{T}$ that 
\begin{align} 
 \langle \Pi^A(t) \xi, \xi \rangle 
  = & \mathbb{E}\int_0^t \big( \big\| M^{1/2} y^A(r) \big\|^2 + \| u(r) \|^2 \big) dr 
   + \mathbb{E}\big\| G^{1/2} y^A(t) \big\|^2 
   \notag \\ 
 &   
 - \mathbb{E} \int_0^t \big\| u(r) + (K^A)^{-1}L^A(T-r) y^A(r) \big\|^2 dr ,  
\notag \\ 
%%%%%
\langle \Pi^\dagger(t) \xi, \xi \rangle 
  = & \mathbb{E}\int_0^t \big( \big\| M^{1/2} y^\dagger(r) \big\|^2 + \| u(r) \|^2 \big) dr 
   + \mathbb{E}\big\| G^{1/2} y^\dagger(t) \big\|^2 
   \notag \\ 
 &   
 - \mathbb{E} \int_0^t \big\| u(r) + (K^\dagger)^{-1}L^\dagger(T-r) y^\dagger(r) \big\|^2 dr . 
 \notag 
\end{align} 
It then follows that 
\begin{align} 
 & \langle ( \Pi^A(t) - \Pi^\dagger(t) ) \xi, \xi \rangle \notag \\ 
 = & \mathbb{E} \int_0^t \big\| M^{1/2} y^A(r) \big\|^2 - \big\| M^{1/2} y^\dagger(r) \big\|^2 dr 
 + \mathbb{E} \big( \big\| G^{1/2} y^A(t) \big\|^2 - \big\| G^{1/2} y^\dagger(t) \big\|^2 \big)   
 \notag \\ 
 & - \mathbb{E}\int_0^t \big( \big\| u(r) + (K^A)^{-1} L^A(T-r) y^A(r) \big\|^2 
%  \notag \\ 
% &\hspace{2cm}  
 - \big\| u(r) + (K^\dagger)^{-1}L^\dagger(T-r) y^\dagger (r) \big\|^2 \big) dr  
  \notag \\ 
  %%%%% 
  = &  \mathbb{E} \int_0^t \langle M^{1/2}(y^A(r) - y^\dagger(r)) ,  M^{1/2}(y^A(r) + y^\dagger(r)) \rangle dr 
  %\notag \\ 
  %&\hspace{2cm} 
  + \mathbb{E}  \langle G^{1/2}(y^A(t) - y^\dagger(t)) ,  G^{1/2}(y^A(t) + y^\dagger(t)) \rangle
  \notag \\ 
  & - \mathbb{E} \int_0^t \bigl\langle (K^A)^{-1}L^A(T-r) y^A(r) - (K^\dagger)^{-1}L^\dagger(T-r) y^\dagger(r) , \notag \\ 
 &\hspace{2cm}  (K^A)^{-1} L^A(T-r) y^A(r) + (K^\dagger)^{-1}L^\dagger(T-r) y^\dagger(r) + 2 u(r)  \bigr\rangle dr . 
  \notag 
\end{align} 
We take $u(t)=0$, for all $t \in \mathcal{T}$, and let $y^A$ and $y^\dagger$ be the states corresponding to $u(t)=0$, $t\in \mathcal{T}$. 
It then follows that  
\begin{align} 
  \langle ( \Pi^A(t) & -  \Pi^\dagger(t) )  \xi, \xi \rangle 
 \notag  \\ 
= &  \mathbb{E} \int_0^t \langle M^{1/2}(y^A(r) - y^\dagger(r)) ,  M^{1/2}(y^A(r) + y^\dagger(r)) \rangle dr 
 \notag \\ 
& \hspace{1cm}   
+ \mathbb{E}  \langle G^{1/2}(y^A(t) - y^\dagger(t)) ,  G^{1/2}(y^A(t) + y^\dagger(t)) \rangle
  \notag \\ 
  & \hspace{1cm}- \mathbb{E} \int_0^t \bigl\langle (K^A)^{-1} L^A(T-r) y^A(r) - (K^\dagger)^{-1}L^\dagger(T-r) y^\dagger(r) , 
  \notag \\ 
 &\hspace{2cm} 
 (K^A)^{-1} L^A(T-r) y^A(r) + (K^\dagger)^{-1}L^\dagger(T-r) y^\dagger(r)   \bigr\rangle dr
  \notag \\ 
%%%%%
&\leq  
\mathbb{E}\int_0^t \|M\|  | y^A(r) - y^\dagger(r) |  |y^A(r) + y^\dagger(r)| dr 
 + \mathbb{E} \big( \| G \|  | y^A(t) - y^\dagger(t) |  |y^A(t) + y^\dagger(t)| \big)   
\notag \\ 
&\hspace{1cm}  + \mathbb{E} \int_0^t 
\big\| (K^A)^{-1} L^A(T-r) y^A(r) - (K^\dagger)^{-1}L^\dagger(T-r) y^\dagger(r) \big\| 
\notag \\ 
& \hspace{3cm}\times
 \big\| (K^A )^{-1} L^A(T-r) y^A(r) + (K^\dagger)^{-1}L^\dagger(T-r) y^\dagger(r) \big\| 
 dr . 
\label{<(PiA1-PiA2)xi,xi>bound|y10-y20|} 
\end{align} 
The following estimates of $\Pi^A$ and $\Pi^\dagger$ are given by~\cite[Proposition 4.3]{LF24}
\begin{align} 
& \| \Pi^A(t) \| \leq C^\Pi_{A} , 
\quad  \| \Pi^\dagger(t) \| \leq C^{\Pi^\dagger} 
\quad \forall t\in \mathcal{T}, \label{Pibound} \\ 
& C^\Pi_{A} = 2 (M_T^{A})^2 \exp \big( 8 T (M_T^{A})^2 \| D \|^2 \mathrm{tr}(Q) \big) 
\big( \| G \| + T \| M \| \big) , 
\notag \\ 
& C^{\Pi^\dagger} = 2 (M_T^{A^\dagger})^2 \exp \big( 8 T (M_T^{A^\dagger})^2 \| D \|^2 \mathrm{tr}(Q) \big) 
\big( \| G \| + T \| M \| \big) . 
\notag 
\end{align} 
By the estimates for $\Delta_1$ and $\Delta_2$ in \cite[Theorem 4.1]{LF24}, the estimate for $\Pi$ in \cite[Proposition 4.3]{LF24}, and the bounds $\|(K^A)^{-1}(t)\|\leq 1$ and $\|(K^\dagger)^{-1}(t)\|\leq 1$, $\forall t\in \mathcal{T}$, we have 
\begin{align} 
  \big\| (K^A)^{-1} L^A(t) -& (K^\dagger)^{-1} L^\dagger(t) \big\| 
  \notag \\ 
 = & \big\| (K^A)^{-1} \big( L^A - L^\dagger \big)(t) + \big( (K^A)^{-1} - (K^\dagger)^{-1} \big) L^\dagger(t) \big\| \notag \\ 
 \leq & 
  \big\| (K^A)^{-1}(t) \big\|  \big\| \big( L^A - L^\dagger \big)(t)  \big\|
  + \big\|  (K^A)^{-1}(t) \big\|  \big\| ( K^\dagger  - K^A  )(t) \big\|  
  \big\| (K^\dagger)^{-1} (t) \big\|  \big\| L^\dagger(t) \big\| \notag
 \notag \\ 
 \leq & 
  ( \| B^\dagger \| + \| \Delta_1 \| ) \| (\Pi^A - \Pi^\dagger )(t) \| + \| \Delta_3 \| 
   \| ( \Pi^A - \Pi^\dagger )(t) \|  (\|B^\dagger \| + \| \Delta_1 \| ) C^{\Pi^\dagger} 
  \notag \\ 
 \leq & 
  ( \| B^\dagger \| + R_3 ) ( 1+ R_5 C^{\Pi^\dagger}  )  \| ( \Pi^A - \Pi^\dagger )(t) \| ,   
  \notag 
\end{align} 
\begin{align} 
 \big\| (K^A)^{-1}& L^A(T-r) y(r) - (K^\dagger)^{-1}L^\dagger(T-r) y^\dagger(r) \big\| \notag \\ 
\leq & 
\big\| (K^A)^{-1} \big\|  \big\| L^A(T-r)  \big\|  \big\| y^A(r) - y^\dagger(r) \big\| 
%\notag \\ 
% & 
 + \big\| \big( (K^A)^{-1} L^A - (K^\dagger)^{-1} L^\dagger \big) (t-r) \big\|  \big\| y^\dagger(r) \big\| 
 \notag \\ 
 \leq & (\|B^\dagger \|+R_3 ) C^\Pi_A \big\| y^A(r) - y^\dagger(r) \big\|    
   + ( \| B^\dagger \| + \| \Delta_1 \| ) ( 1 + \| \Delta_3 \| C^{\Pi^\dagger} )   \| ( \Pi^A - \Pi^\dagger)(t) \|   
   \big\| y^\dagger(r) \big\| 
  \notag \\ 
   \leq & (\|B^\dagger \|+R_3 ) C^\Pi_A \big\| y^A(r) - y^\dagger(r) \big\|    
   + ( \| B^\dagger \| + R_3 ) ( 1 + R_5 C^{\Pi^\dagger} )    \| (\Pi^A - \Pi^\dagger )(t) \|   
   \big\| y^\dagger(r) \big\| , 
 \label{|KA1invLA1yA1-KA2invLA2yA2|<=|yA1-yA2|+|PiA1-PiA2|} 
\end{align}  
and 
\begin{align} 
\big\| (K^A)^{-1} L^A(T-r) y^A(r) & + (K^\dagger)^{-1}L^\dagger(T-r) y^\dagger(r) \big\|  \notag \\ 
\leq & 
 \big\| (K^A)^{-1} \big\|  \big\| L^A(T-r) \big\|  \| y^A(r) \| 
   + \big\| (K^\dagger)^{-1} \big\|  \big\| L^\dagger(T-r) \big\|  \big\| y^\dagger(r) \big\| 
 \notag \\ 
 \leq & 
   (\| B^\dagger \| + R_3 ) ( C^\Pi_A \| y^A(r) \| + C^{\Pi^\dagger} \| y^\dagger(r) \| ) . 
 \label{|KA1invLA1yA1+KA2invLA2yA2|<=|yA1|+|yA2|} 
\end{align} 

From \eqref{<(PiA1-PiA2)xi,xi>bound|y10-y20|}, \eqref{|KA1invLA1yA1-KA2invLA2yA2|<=|yA1-yA2|+|PiA1-PiA2|} 
and~\eqref{|KA1invLA1yA1+KA2invLA2yA2|<=|yA1|+|yA2|}, we have 
\begin{align} 
 & \frac{1}{\|\xi\|^2}\langle (\Pi^A - \Pi^\dagger)(t) \xi , \xi \rangle  
 \notag \\ 
 \leq & 
 \frac{1}{\|\xi\|^2}
 \| M \| (T+1) 
  \big( \sup_{r\in \mathcal{T}}\mathbb{E}\big\| y^A(r) - y^\dagger(r) \big\|^2 \big)^{1/2} 
 \Big[ 
    \big( \sup_{r\in \mathcal{T}}\mathbb{E} \| y^A(r) \|^2 \big)^{1/2} 
  +  \big( \sup_{r \in \mathcal{T}} \mathbb{E} \|y^\dagger(r)\|^2 \big)^{1/2}  \Big] 
  \notag \\ 
  & 
  + \frac{1}{\|\xi\|^2} T (\|B^\dagger \|+R_3)^2 C^{\Pi^\dagger} 
   \Big[ C^\Pi_A \big( \sup_{r\in \mathcal{T}} \mathbb{E} \|y^A(r)\|^2 \big)^{1/2} 
   + C^{\Pi^\dagger} \big( \sup_{r\in \mathcal{T}} \mathbb{E}\|y^\dagger(r)\|^2 \big)^{1/2} \Big] 
   \notag \\ 
 & \times   \big( \sup_{t\in \mathcal{T}} \mathbb{E} \|y^A(r) - y^\dagger(r)\|^2 \big)^{1/2} 
  + 
 \frac{1}{\|\xi\|^2}
 ( \| B^\dagger \| + R_3 )^2 ( 1 + R_5 C^{\Pi^\dagger} ) 
 \notag \\ 
& \times   \Big[ C^\Pi_A \big( \sup_{t\in\mathcal{T}} \mathbb{E} \big\| y^A(r) \big\|^2 \big)^{1/2}  
   \big( \sup_{r\in\mathcal{T}} \mathbb{E} \| y^\dagger(r) \|^2 \big)^{1/2} + C^{\Pi^\dagger}  \big( \sup_{r\in\mathcal{T}} \mathbb{E}\| y^\dagger(r) \|^2 \big)^{1/2}  \Big] 
\int_0^t   \| (\Pi^A - \Pi^\dagger)(r) \|   dr . 
\notag 
\end{align} 
Then, from Gr\"{o}nwall's inequality and the estimates \eqref{E|yi0|^2bound} and \eqref{E|y10-y20|^2bound} it follows that 
\begin{align} 
  \big\| (\Pi^A - \Pi^\dagger)(t) \big\| 
% \notag \\ 
 \leq & 
 \frac{1}{\|\xi\|^2}
 \Big\{ 
 \| M \| (T+1) 
+  T (\|B^\dagger\|+R_3)^2 C^\Pi_A  
   \Big\} ( 1+C^\Pi_A + C^{\Pi^\dagger} ) 
  \notag \\ 
  & 
    \times \Big[ \big( \sup_{r\in \mathcal{T}} \mathbb{E} \|y^A(r)\|^2 \big)^{1/2} 
   +  \big( \sup_{r\in \mathcal{T}} \mathbb{E}\|y^\dagger(r)\|^2 \big)^{1/2} \Big]
   \big( \sup_{r\in \mathcal{T}} \mathbb{E} \| y^A(r) - y^\dagger(r)\|^2 \big)^{1/2}  
   \notag \\ 
&   \times \exp\Big\{ \frac{1}{\|\xi\|^2} T
 ( \| B^\dagger \| + R_3 )^2 ( 1 + R_5 C^{\Pi^\dagger} )    
 \notag \\ 
 &\hspace{1.1cm}\times  
 \Big[ C^\Pi_A \big( \sup_{r\in\mathcal{T}} \mathbb{E} \big\| y^A(r) \big\|^2 \big)^{1/2}  
 \big( \sup_{r\in\mathcal{T}} \mathbb{E} \| y^\dagger(r) \|^2 \big)^{1/2} 
  + C^{\Pi^\dagger}  \big( \sup_{r\in\mathcal{T}} \mathbb{E}\| y^\dagger(r) \|^2 \big)^{1/2}  \Big]  \Big\}.  
 \label{|PiA-Pidagger|<=E|yA-ydagger|}
\end{align} 
By \eqref{E|yi0|bound} and \eqref{E|y10-y20|bound}, we have 
\begin{align} 
 &  \Big[  \big( \sup_{r\in \mathcal{T}} \mathbb{E} \|y^A(r)\|^2 \big)^{1/2}  
   +  \big( \sup_{r\in \mathcal{T}} \mathbb{E}\|y^\dagger(r)\|^2 \big)^{1/2} \Big] 
   \big( \sup_{t\in \mathcal{T}} \mathbb{E} \|y^A(r) - y^\dagger(r)\|^2 \big)^{1/2} 
   \notag \\ 
\leq & 
\sqrt{2} \|\xi\|
\big[
 (M_T^A) \exp\big( 8 T (M_T^A)^2 \| D\|^2 \mathrm{tr}(Q) \big)  
+ 
 (M_T^{A^\dagger}) \exp\big( 8 T (M_T^{A^\dagger})^2 \| D\|^2 \mathrm{tr}(Q) \big)  
\big]  
\notag \\ 
& 
\times \sqrt{2}  \sup_{t\in \mathcal{T}} \| (S - S^\dagger )(t) \|  \| D \| T^{1/2} 
 \Big[ \sqrt{2} M_T^A \|\xi\| \exp\big( 8 T (M_T^A)^2 \| D\|^2 \mathrm{tr}(Q) \big) \Big]  
% \notag \\ 
      \exp(  (M_T^{A^\dagger} )^2 \| D \|^2 T )  
\label{(E|y10|+E|y20|)(E|y10-y20|)bound}
\end{align} 
and 
\begin{align} 
 \big( \sup_{t\in\mathcal{T}} \mathbb{E} \big\| y^A(r) \big\|^2 \big)^{1/2}  
& \big( \sup_{t\in\mathcal{T}} \mathbb{E} \| y^\dagger(r) \|^2 \big)^{1/2} 
 \leq  
  2 M_T^A M_T^{A^\dagger} \|\xi\|^2 \exp\big[ 8 T \big( (M_T^A)^2 + (M_T^{A^\dagger})^2 \big) \| D\|^2 \mathrm{tr}(Q) \big]  . 
\label{(E|y10|)(E|y20|)bound} 
\end{align} 
Substituting \eqref{(E|y10|+E|y20|)(E|y10-y20|)bound} and 
\eqref{(E|y10|)(E|y20|)bound} into~\eqref{|PiA-Pidagger|<=E|yA-ydagger|}, 
we obtain
\eqref{|Pi1-Pi2|<=|A1-A2|}. 
\end{proof} 

\subsubsection{{Proof of Lemma~\ref{lem:A:|q1-q2|and|barx1-barx2|}}}\label{sec:lem:A:|q1-q2|and|barx1-barx2|}
\begin{proof}
From~\eqref{eq:barxi_MFG_int} and~\eqref{eq:qi_MFG_int}, we have 
\begin{align} 
 \big\| (\bar{x}^A - \bar{x}^\dagger )(t) \big\| 
 \leq & \big\| (S^A - S^\dagger )(t) \big\| \| \bar{\xi} \| 
 \! +\!\! \int_0^t \!\! \big\| ( S^A - S^\dagger )(t-r) \big\|   
   \big\| \Phi^A(r) \big\|   dr 
\notag \\
 &+ \!\! \int_0^t \!\! \big\| S^\dagger (t-r) \big\|   
 \big\| \Phi^A(r) - \Phi^\dagger (r) \big\| dr , 
 \notag \\
%\end{align} 
%\begin{align} 
 \|(q^A - q^\dagger )(t)\| 
 \leq & 
 \big( \|(S^A - S^\dagger )(t)\|  \|x^\dagger \| + \| S^\dagger \|  \| (\bar{x}^A - \bar{x}^\dagger\|)(T) \big) \| G \|  \|\widehat{F}_2\| 
 \notag \\ 
 & + \int_0^t \left(\|(S^A - S^\dagger )(t-r)\|  \|\Psi^A (r)\|  
 + \|S^\dagger(t-r)\|  \|(\Psi^A - \Psi^\dagger )(r)\|\right) dr, 
 \notag 
\end{align} 
where $\Phi^\dagger$ and $\Psi^\dagger$ are given by \eqref{Phi-ref} and \eqref{Psi-ref}, respectively, and $\Phi^A$ and $\Psi^A$ are defined analogously for the perturbed operator $A$. Since 
\begin{align} 
 &  \big\| \Phi^A(r) - \Phi^\dagger(r) \big\|  
  \leq  C^{\Phi, \bar{x} }_{A, A^\dagger} \| ( \bar{x}^A - \bar{x}^\dagger )(r)\| 
  + C^{\Phi, q }_{A, A^\dagger} \|(q^A - q^\dagger )(T-r)\| 
  + C^{\Phi, \Pi}_{A, A^\dagger} \big\| ( \Pi^A - \Pi^\dagger )(T-r)\big\|   , 
  \notag \\ 
%\end{align} 
%\begin{align} 
 &  \big\| \Psi^A(r) - \Psi^\dagger(r) \big\|  
  \leq  C^{\Psi,\bar{x} }_{A, A^\dagger} \| ( \bar{x}^A - \bar{x}^\dagger )(T-r)\| 
  + C^{\Psi, q }_{A, A^\dagger} \|( q^A - q^\dagger )(r)\| 
  + C^{\Psi, \Pi }_{A, A^\dagger} \big\| ( \Pi^A - \Pi^\dagger )(r) \big\| , 
  \notag 
\end{align} 
where the constants $C^{\Phi, \bar{x} }_{A, A^\dagger}$, $C^{\Phi, q }_{A, A^\dagger}$, $C^{\Phi, \Pi }_{A, A^\dagger}$, 
$C^{\Psi, \bar{x} }_{A, A^\dagger}$, $C^{\Psi, q }_{A, A^\dagger}$, and $C^{\Psi, \Pi }_{A, A^\dagger}$ are defined in the statement of the lemma, 
it then follows that 
\begin{align} 
 \big\| (\bar{x}^A& - \bar{x}^\dagger )(t) \big\| \notag \\ 
 \leq & M^{A, A^\dagger }_T \|A - A^\dagger \| \| \bar{\xi} \| 
  + \int_0^t M^{A, A^\dagger}_T \| A - A^\dagger \|     
   \big( C^{\Phi, \bar{x}}_{A, A^\dagger} C^{\bar{x}}_A   
+ C^{\Phi, q}_{A, A^\dagger}  C^q_A 
+ C^{\Phi, c}_{A, A^\dagger}  \big)  dr
   \notag \\ 
&   + \int_0^t M^{A^\dagger}_T  
\big( C^{\Phi, \bar{x} }_{A,A^\dagger} \| (\bar{x}^A - \bar{x}^\dagger )(r)\| 
  + C^{\Phi, q }_{A, A^\dagger} \| ( q^A - q^\dagger)(T-r) \| 
  + C^{\Phi, \Pi }_{A, A^\dagger } C^{\Pi}_{A, A^\dagger} \| A - A^\dagger \| 
  \big) dr
   \notag 
\end{align} 
and 
\begin{align} 
\big\|( q^A &- q^\dagger )(t)\big\| \notag \\
\leq & \big( M^{A, A^\dagger }_T \| A - A^\dagger \| C^{\bar{x}} 
+ M^{A^\dagger }_T \| (\bar{x}^A - \bar{x}^\dagger)(T) \| \big) \|G\| \|\widehat{F}_2 \| 
\notag \\ 
& + \int_0^t M^{A, A^\dagger }_T \|A - A^\dagger \|
 \big( C^{\Psi, q}_{A, A^\dagger}  C^q_A 
+C^{\Psi, \bar{x}}_{A, A^\dagger} C^{\bar{x}}_A   
+ C^{\Psi, c}_{A, A^\dagger}  \big) dr
\notag \\ 
& + \int_0^t M^{A^\dagger }_T  \big( 
 C^{\Psi, q}_{A, A^\dagger } \|( q^A - q^\dagger )(r)\| 
 + C^{\Psi, \bar{x}}_{A, A^\dagger } \| (\bar{x}^A - \bar{x}^\dagger) (T-r) \|  
  + C^{\Psi, \Pi}_{A, A^\dagger } C^\Pi_{A, A^\dagger} \| A - A^\dagger \| \big) dr . \notag 
\end{align} 
Applying Lemma~\ref{lem:FBGronwall} to the above two inequalities, we obtain 
the desired estimates for 
$\|\bar{x}^A - \bar{x}^\dagger\|_{C(\mathcal{T}; H )} $ 
and 
$\| q^A - q^\dagger \|_{ C(\mathcal{T}; H) }$. 
\end{proof}

%%%%%%%%%%%%%%%%%%%%%%%%%%%%%%%% 
\subsection{Proof of Lipschitz Stability with respect to \texorpdfstring{$B$}{B}}
\label{sec:Proof_Stability_B}

\subsubsection{{Proofs of Lemma~\ref{lem:B:|Pi1-Pi2|}}}\label{sec:lem:B:|Pi1-Pi2|}
\begin{proof}
The proof is carried out in a similar manner as in~\cite[Proposition 4.3]{LF24}. 
We introduce the processes corresponding to $B$ and $B^\dagger$
\begin{equation} 
\begin{aligned}
 & d y^B(t) = (A^\dagger y^B(t) + B u(t)) dt + (D y^B(t) + E u(t)) d W(t) ,  \quad 
 y^B(0) = \xi , \\ 
 & d y^\dagger(t) = (A^\dagger y^\dagger(t) + B^\dagger u(t)) dt + (D y^\dagger(t) + E u(t)) d W(t) ,  \quad 
 y^\dagger(0) = \xi ,  
 \end{aligned} 
 \label{dyi}
 \end{equation} 
and the cost functionals 
\begin{align} 
 & \mathbb{E}\int_0^T \big( \big\| M^{1/2} y^B(t) \big\|^2 + \|u(t)\|^2 \big) dt 
   + \mathbb{E} \big\| G^{1/2} y^B(T) \big\|^2 ,   
   \notag \\ 
 & \mathbb{E}\int_0^T \big( \big\| M^{1/2} y^\dagger(t) \big\|^2 + \|u(t)\|^2 \big) dt 
   + \mathbb{E} \big\| G^{1/2} y^\dagger(T) \big\|^2 . 
\notag 
\end{align} 
Let $S^\dagger \in \mathcal{L}(H)$ be the semigroup corresponding to the infinite generator $A^\dagger$.  The mild solutions of \eqref{dyi} for $B$ and $B^\dagger$ are given by 
\begin{align} 
\begin{aligned} 
y^B(t) = & S^\dagger(t) \xi + \int_0^t S^\dagger(t-r) B u(r) dr  
 + \int_0^t S^\dagger(t-r) (D y^B (r) + E u(r)) d W(r) , 
  \\ 
 y^\dagger (t) = & S^\dagger(t) \xi + \int_0^t S^\dagger(t-r) B^\dagger u(r) dr  
 + \int_0^t S^\dagger(t-r) (D y^\dagger (r) + E u(r)) d W(r) , \label{B:yi} 
 \end{aligned} 
 \end{align} 
and satisfy
\allowdisplaybreaks
\begin{equation}
 y^B(t) - y^\dagger(t) 
 = 
   \int_0^t S^\dagger(t-r) (B - B^\dagger )u(r) dr  
 + \int_0^t  
   S^\dagger(t-r)D( y^B(r) -  y^\dagger(r) )   d W(r) . 
 \notag 
\end{equation} 
By the It\^{o}'s formula and Cauchy-Schwarz inequality, we have 
\begin{align} 
  \mathbb{E}  \|y^B(t) -& y^\dagger(t)\|^2   \notag \\ 
 = &
 \mathbb{E} \int_0^t 2 \langle y^B(r) - y^\dagger(r) , S^\dagger(t-r) (B - B^\dagger ) u(r)  \rangle dr 
 + \mathbb{E}\int_0^t \big\| 
 S^\dagger(t-r) D ( y^B(r) -  y^\dagger(r) ) \big\|^2 dr
 \notag \\ 
 \leq &  4 \mathbb{E}\int_0^t  \big\| y^B(r) - y^\dagger(r) \big\|^2 
  + (M^{A^\dagger}_T)^2 \|B - B^\dagger \|^2
   \| u(r)\|^2  dr  
   \notag \\ 
&   + (M^{A^\dagger}_T)^2 \| D \|^2 \mathbb{E}\int_0^t  \big\| y^B(r) -  y^\dagger (r) \big\|^2 dr , 
 \label{CauchySchwarz|y1-y2|}
\end{align} 
where $M^{A^\dagger}_T$ is the upper bound for $S^\dagger$ such that $\|S^\dagger(t)\|\leq M^{A^\dagger}_T$, for all $t\in \mathcal{T}$. 

By~\eqref{B:yi} and the inequality $\|a+b+c\|^2 \leq 3(\|a\|^2 + \|b\|^2 + \|c\|^2)$,  
we have  
\begin{align}  
 \mathbb{E} | y^\dagger(t) |^2 = & 3 (M^{A^\dagger}_T)^2 \|\xi\|^2 + 3 \mathbb{E} \Big\| \int_0^t S^\dagger(t-r) B^\dagger u(r) dr \Big\|^2   
  + 3 \mathbb{E} \Big\| \int_0^t S^\dagger(t-r) (D y^\dagger (r) + E u(r)) d W(r) \Big\|^2 
 \notag \\ 
 \leq & 3 (M^{A^\dagger}_T)^2 \|\xi\|^2 
 + 3 (M^{A^\dagger}_T)^2 \mathbb{E} \int_0^t  \| B^\dagger \|  \| u(r) \| dr   
 \notag \\ 
 & 
 + 3 (M^{A^\dagger}_T)^2 \mathbb{E}  \int_0^t  2 \big( \| D \|^2 
  \|y^\dagger (r) \|^2
  +  \| E \|^2  \| u(r) \|^2 \big) d r . 
 \label{B|yi|2<=int|yi|2}
\end{align} 

From~\eqref{CauchySchwarz|y1-y2|} and~\eqref{B|yi|2<=int|yi|2} with $\|u(t)\|=1$, for all $t\in \mathcal{T}$, we have by the Gr\"{o}nwall's inequality that    
\begin{align} 
 & \mathbb{E}  \|y^B(t) - y^\dagger (t)\|^2   
  \leq  4 T (M^{A^\dagger}_T)^2 \big\| B - B^\dagger \big\|^2   
   \exp\big\{  \big( 4 +  (M^{A^\dagger}_T)^2 \| D\|^2 \big) T  \big\} , 
\label{B|y1-y2|2bound} \\
 &\mathbb{E} \|y^\dagger(t)\|^2 \leq 
  3 (M^{A^\dagger}_T)^2 \big[ \|\xi\|^2 + ( \|B^\dagger \| + 2 \|E\|^2 )T  \big] 
 \exp\big\{ 6 (M^{A^\dagger}_T)^2 \| D \|^2 T \big\}  . 
 \label{B|yi|2bound}
\end{align}

By the proof of~\cite[Proposition 4.3]{LF24}, we have for any $u(t)$, $t\in \mathcal{T}$, that 
\begin{align} 
 \langle \Pi^\dagger(t) \xi, \xi \rangle 
  = & \mathbb{E}\int_0^t \big( \big\| M^{1/2} y^\dagger(r) \big\|^2 + \|u(r)\|^2 \big) dr 
   + \mathbb{E}\big\| G^{1/2} y^\dagger(t) \big\|^2 
   \notag \\ 
 &    - \mathbb{E} \int_0^t \big\| u(r) + (K^\dagger)^{-1}L^\dagger(T-r) y^\dagger(r) \big\|^2 dr , 
\notag 
\end{align} 
and furthermore, 
\begin{align} 
  \langle ( \Pi^B&(t) - \Pi^\dagger(t) ) \xi, \xi \rangle \notag \\ 
 = & \mathbb{E} \int_0^t \big\| M^{1/2} y^B(r) \big\|^2 - \big\| M^{1/2} y^\dagger(r) \big\|^2 dr 
 + \mathbb{E} \big( \big\| G^{1/2} y^B(t) \big\|^2 - \big\| G^{1/2} y^\dagger(t) \big\|^2 \big)   
 \notag \\ 
 & - \mathbb{E}\int_0^t \big( \big\| u(r) + (K^B)^{-1} L^B(T-r) y^B(r) \big\|^2 
 % \notag \\ 
 %&\hspace{2cm}  
 - \big\| u(r) + (K^\dagger)^{-1}L^\dagger(T-r) y^\dagger(r) \big\|^2 \big) dr  
  \notag \\ 
  = &  \mathbb{E} \int_0^t \langle M^{1/2}(y^B(r) - y^\dagger(r)) ,  M^{1/2}(y^B(r) + y^\dagger(r)) \rangle dr 
  \notag \\ 
  &\hspace{2cm} 
  + \mathbb{E}  \langle G^{1/2}(y^B(t) - y^\dagger(t)) ,  G^{1/2}(y^B(t) + y^\dagger(t)) \rangle
  \notag \\ 
  & - \mathbb{E} \int_0^t \bigl\langle (K^B)^{-1} L^B(T-r) y^B(r) - (K^\dagger)^{-1}L^\dagger(T-r) y^\dagger(r) , \notag \\ 
 &\hspace{2cm}  (K^B)^{-1} L^B(T-r) y^B(r) + (K^\dagger)^{-1}L^\dagger(T-r) y^\dagger(r) + 2 u(r)  \bigr\rangle dr . 
  \label{<(Pi1-Pi2)xi,xi>bound|y10-y20|} 
\end{align} 
We take $\| u(t) \|=1$, for all $t \in \mathcal{T}$, 
and let $y$ and $y^\dagger$ be the states corresponding to the $u(t)$, $t\in \mathcal{T}$. 
By the Cauchy-Schwarz inequality, we obtain 
\begin{align} 
 \langle ( \Pi^B(t)& - \Pi^\dagger(t) ) \xi, \xi \rangle 
 \notag \\ 
\leq &  \mathbb{E} \int_0^t \| M \|  \big\| y^B(r) - y^\dagger(r) \big\|  \big\| y^B(r) + y^\dagger(r) \big\| dr   
  + \| G \|  \big\| y^B(T) - y^\dagger(T) \big\|  \big\| y^B(T) + y^\dagger(T) \big\| 
 \notag \\ 
& + \mathbb{E} \int_0^t 
\big\| (K^B)^{-1} L^B(T-r) y^B(r) - (K^\dagger)^{-1}L^\dagger(T-r) y^\dagger(r) \big\| 
\notag \\ 
& \times
 \big( \big\| (K^B)^{-1} L^B(T-r) y^B(r) + (K^\dagger)^{-1}L^\dagger(T-r) y^\dagger(r) \big\| 
  + 2 \big)
 dr . 
\label{<(Pi1-Pi2)xi,xi>|u|=1bound|y10-y20|}  
\end{align} 
By~\cite[Proposition 4.3]{LF24}, we have the following estimate of $\Pi^\dagger$  
\begin{align} 
& \| \Pi^\dagger(t) \| \leq C^{\Pi^\dagger} , \quad \forall t\in \mathcal{T}, \label{Pibound} \\ 
& C^{\Pi^\dagger} = 2 (M_T^{A^\dagger})^2 \exp \big( 8 T (M_T^{A^\dagger} )^2 \| D \|^2 \mathrm{tr}(Q) \big) 
\big( \| G \| + T \| M \| \big) . 
\notag 
\end{align} 
By the estimates for $\Delta_1$ and $\Delta_2$ in \cite[Theorem 4.1]{LF24}, the estimate for $\Pi^\dagger$ in \cite[Proposition 4.3]{LF24}, and the bounds $\|(K^B)^{-1}(t)\|\leq 1$ and $\|(K^\dagger)^{-1}(t)\|\leq 1$, $\forall t \in \mathcal{T}$, we have 
\begin{align} 
& \big\| L^\dagger(t) \big\| = \big\| ( B^\dagger )^\star \Pi^\dagger(t) + \Delta_1(\Pi^\dagger(t)) \big\| 
\notag \\ 
 & \hspace{1.5cm} \leq  \big( \| B^\dagger \|  + \| \Delta_1 \| \big)  \big\| \Pi^\dagger(t) \big\| 
 \leq \big( \| B^\dagger \|  + R_3 \big) C^{\Pi^\dagger} ,  
 \notag \\ 
& \| (K^\dagger)^{-1}(t) \| 
 = \big\| ( I + \Delta_3(\Pi^\dagger(t)) )^{-1} \big\| 
 \leq 1 , 
 \notag \\ 
 & \big\| L^B(t) - L^\dagger (t) \big\| 
 = \big\| B^\star \Pi^B(t) - (B^\dagger)^\star \Pi^\dagger(t) + \Delta_1(\Pi^B(t) - \Pi^\dagger(t)) \big\| 
\notag \\ 
& \hspace{2.4cm}
\leq 
 C^{\Pi^\dagger}  \| B - B^\dagger \|  
+ ( \| B^\dagger \| + R_3 )  \big\| \Pi^B(t) - \Pi^\dagger (t) \big\|  , 
\notag  \\ 
& \| K^B(t) - K^\dagger (t) \| = \| \Delta_3( \Pi^B(t) - \Pi^\dagger (t) ) \| 
 \leq R_5 \big\| \Pi^B(t) - \Pi^\dagger(t) \big\| ,
\notag 
\end{align} 
and 
\begin{align} 
 \big\| (K^B)^{-1}&  L^B(t) - (K^\dagger)^{-1} L^\dagger(t) \big\| \notag \\ 
 = & \big\| (K^B)^{-1} \big( L^B - L^\dagger \big)(t) + \big( (K^B)^{-1} - (K^\dagger)^{-1} \big) L^\dagger(t) \big\| \notag \\ 
 \leq & 
  \big\| (K^B)^{-1}(t) \big\|  \big\| \big( L^B - L^\dagger \big)(t)  \big\| 
  + \big\|  (K^B)^{-1}(t) \big\|  \big\| ( K^\dagger  - K^B  )(t) \big\|  
  \big\| (K^\dagger)^{-1} (t) \big\|  \big\| L^\dagger(t) \big\| \notag
 \notag \\ 
 \leq & 
 C^{\Pi^\dagger}  \| B - B^\dagger \|  
+ ( \| B^\dagger \| + R_3 + R_5 )  \big\| \Pi^B(t) - \Pi^\dagger(t) \big\|  .  
  \notag 
\end{align} 
Then we have the estimates 
\begin{align} 
 \big\| (K^B)^{-1}& L^B(T-r) y^B(r) - (K^\dagger)^{-1}L^\dagger(T-r) y^\dagger(r) \big\|  \notag \\ 
\leq &  
\big\| (K^B)^{-1} \big\|  \big\| L^B(T-r)  \big\|  \big\| y^B(r) - y^\dagger(r) \big\| 
%\notag \\ 
% & 
 + \big\| (K^B)^{-1} L^B(t-r) - (K^\dagger)^{-1} L^\dagger(t-r) \big\|  \big\| y^\dagger(r) \big\| 
 \notag \\ 
 \leq & (\|B \| + R_3 ) C^{\Pi^\dagger} \big\| y^B(r) - y^\dagger(r) \big\|   
 \notag \\ 
 &  + \big\{ 
 C^{\Pi^\dagger}  \| B - B^\dagger \|  
+ ( \| B^\dagger \| + R_3 + R_5 ) \big\| ( \Pi^B - \Pi^\dagger )(t) \big\|  
 \big\}   
   \big\| y^\dagger(r) \big\| 
 \label{|K1invL1y1-K2invL2y2|<=|y1-y2|+|Pi1-Pi2|} 
\end{align}  
and 
\begin{align} 
\big\| (K^B)^{-1}& L^B(T-r) y^B(r) + (K^\dagger)^{-1}L^\dagger(T-r) y^\dagger(r) \big\|  \notag \\ 
\leq & 
 \big\| (K^B)^{-1} \big\|  \big\| L^B(T-r) \big\|  \| y^B(r) \| 
   + \big\| (K^\dagger)^{-1} \big\|  \big\| L^\dagger(T-r) \big\|  \big\| y^\dagger(r) \big\| 
 \notag \\ 
 \leq & \big[ \|B \| \|y^B(r)\|  
  + \|B^\dagger \|  \|y^\dagger(r)\| +  R_3 ( \|y^B(r)\| + \|y^\dagger(r)\| ) \big]  C^{\Pi^\dagger} . 
 \label{|K1invL1y1+K2invL2y2|<=|y1|+|y2|} 
\end{align} 

By \eqref{<(Pi1-Pi2)xi,xi>|u|=1bound|y10-y20|}, 
\eqref{|K1invL1y1-K2invL2y2|<=|y1-y2|+|Pi1-Pi2|}, 
\eqref{|K1invL1y1+K2invL2y2|<=|y1|+|y2|}, and the Cauchy-Schwarz inequality, 
we have 
\begin{align} 
\frac{1}{\|\xi\|^2} & \langle (\Pi^B - \Pi^\dagger )(t) \xi , \xi \rangle  \notag \\ 
 \leq & 
\frac{1}{\|\xi\|^2} 
 (\| M \| T + \| G \| ) 
 \big[ \sup_{r\in\mathcal{T}} \mathbb{E} \|y^B(r) - y^\dagger(r)\|^2 \big]^{1/2} 
 \big[ \sup_{r\in\mathcal{T}} \mathbb{E}\|y^B(r) + y^\dagger (r)\|^2 \big]^{1/2}
\notag \\ 
& + \frac{1}{\|\xi\|^2} \big[ ( \|B \|  + \|B^\dagger \| + 2 R_3 ) C^{\Pi^\dagger} 
 (\sup_{r\in \mathcal{T}} \mathbb{E} \|y^B(r)\|^2 )^{1/2} 
  + 2 \big]  
 \notag \\ 
 &
\times\int_0^t 
 (\| B \| + R_3 ) C^{\Pi^\dagger} ( \sup_{r\in\mathcal{T}} \| y^B(r) - y^\dagger(r)\|^2 )^{1/2} 
  \notag \\ 
 &  + \big\{ 
 C^{\Pi^\dagger}   \| B - B^\dagger \|  
+ ( \| B^\dagger \| + R_3 + R_5 ) \big\| \Pi^B(r) - \Pi^\dagger(r) \big\|  
 \big\} 
\big\{ \sup_{r\in\mathcal{T}} ( \mathbb{E} \|y^\dagger(r)\|^2 )^{1/2} \big\}   dr. 
\notag 
\end{align}
It then follows from the Gr\"{o}nwall's inequality that 
\begin{align} 
 \big\| (\Pi^B& - \Pi^\dagger )(t) \big\| \notag \\ 
 \leq & \frac{1}{\|\xi\|^2} 
  \Big\{ 
  (\| M \| T + \| G \| ) 
 \big[ \sup_{r\in\mathcal{T}} \mathbb{E} \| y^B(r) - y^\dagger(r)\|^2 \big]^{1/2} 
 2 \big[ \sup_{r\in\mathcal{T}} \mathbb{E}\|y^B(r) \|^2 \big]^{1/2} 
 \notag \\ 
 & + ( \|B \|  + \|B^\dagger \| + 2 R_3 )  C^{\Pi^\dagger} 
 (  \sup_{t\in \mathcal{T}} \mathbb{E}  \big\| y^\dagger(t) \big\|^2 )^{1/2} 
  (  \sup_{t\in \mathcal{T}} \mathbb{E}  \big\| y^B(t) - y^\dagger (t) \big\|^2 )^{1/2} 
  \notag \\ 
& \times T (\|B^\dagger \| + R_3) (C^{\Pi^\dagger})^2 \| B - B^\dagger \|  \Big\}  
  \exp \Big\{ ( \|B^\dagger \| + R_3 + R_5 )T ( \sup_{t\in \mathcal{T}} \mathbb{E} \|y^\dagger(t)\|^2 )^{1/2} \Big\}  
 \notag \\ 
 %%%%%%% 
 \leq & 
 \frac{1}{\|\xi\|^2} 
 \big( \sup_{r\in\mathcal{T}} \mathbb{E} \| y^B(r) - y^\dagger (r)\|^2 \big)^{1/2} 
  \big( \sup_{r\in\mathcal{T}} \mathbb{E}|y^B(r) |^2 \big)^{1/2}   
  \Big\{ 
  2 ( \| M \| T + \| G \| ) 
  \notag \\ 
&  + ( \|B \|  + \|B^\dagger \| + 2 R_3 )    
 T (\|B \| + R_3) (C^{\Pi^\dagger})^3 \| B - B^\dagger \|  \Big\} 
 \notag \\ 
 & \times \exp \Big\{ ( \|B^\dagger \| + R_3 + R_5 )T ( \sup_{t\in \mathcal{T}} \mathbb{E} \|y^\dagger (t)\|^2 )^{1/2} \Big\} .  
 \notag 
\end{align} 
By~\eqref{B|y1-y2|2bound} and~\eqref{B|yi|2bound}, we have 
 \begin{align}
  \big\| (\Pi^B - \Pi^\dagger )(t) \big\| 
 \leq &
 \sqrt{T} \big\| B - B^\dagger \big\|
   \exp\Big\{ \frac{1}{2}  ( M_T^{A^\dagger} )^2 (1 + \| D\|^2 )T  \Big\}  
  \sqrt{2} M^{A^\dagger}_T  \exp\big(8T (M^{A^\dagger}_T)^2 \| D\|^2 \mathrm{tr}(Q) \big)     
  \notag \\ 
 & \times\Big\{ 
  2 ( \| M \| T + \| G \| ) 
  + ( \|B \|  + \|B^\dagger \| + 2 R_3 )    
 T (\|B \| + R_3) (C^{\Pi^\dagger})^3 \| B - B^\dagger \|  \Big\} 
 \notag \\ 
 & \times \exp \Big\{ ( \|B^\dagger \| + R_3 + R_5 )T 
 \big[ \sqrt{2} M^{A^\dagger}_T  \exp\big(8T (M^{A^\dagger}_T)^2 \| D\|^2 \mathrm{tr}(Q) \big)  \big] \Big\}  . 
 \label{|Pi1-Pi2|<=E|y10-y20|}
\end{align} 

\end{proof} 

\subsubsection{{Proof of Lemma~\ref{lem:B:|q1-q2|and|barx1-barx2|}}}\label{sec:lem:B:|q1-q2|and|barx1-barx2|}
\begin{proof}
From~\eqref{eq:barxi_MFG_int} and~\eqref{eq:qi_MFG_int}, we have 
\begin{align} 
& \big\| (\bar{x}^B - \bar{x}^\dagger )(t) \big\| 
 \leq   \int_0^t \big\| S^\dagger (t-r) \big\|   
 \big\| \Phi^B(r) - \Phi^\dagger (r) \big\| dr , 
 \notag \\ 
%\end{align} 
%\begin{align} 
 \|(q - q^\dagger )(t)\| 
& \leq  
   \| S(t) \|  \| (\bar{x} - \bar{x}^\dagger)(T) \|  \| G \|  \|\widehat{F}_2\| 
  + \int_0^t  \|S(t-r)\|  \|(\Psi - \Psi^\dagger )(r)\| dr . 
 \notag 
\end{align}

Since 
\begin{align} 
   \big\| \Phi^B(r) - \Phi^\dagger(r) \big\|  
  \leq &  C^{\Phi, \bar{x} }_{B, B^\dagger} \| ( \bar{x}^B - \bar{x}^\dagger )(r)\| 
  + C^{\Phi, q }_{B, B^\dagger} \|(q^B - q^\dagger )(T-r)\| 
  \notag \\ 
 & + C^{\Phi, \Pi}_{B, B^\dagger} \| (\Pi^B - \Pi^\dagger)(T-r) \|
  + C^{\Phi,c }_{B, B^\dagger} \| B - B^\dagger \| , 
  \notag \\ 
%\end{align}  
%\begin{align} 
  \big\| \Psi^B(r) - \Psi^\dagger(r) \big\|  
  \leq & C^{\Psi,\bar{x} }_{ B, B^\dagger} \| ( \bar{x}^B - \bar{x}^\dagger )(T-r)\| 
  + C^{\Psi, q }_{B, B^\dagger} \|(q^B - q^\dagger )(r)\| 
  \notag \\ 
 & + C^{\Psi, \Pi}_{B, B^\dagger} \| (\Pi^B - \Pi^\dagger)(r) \|
  + C^{\Psi,c }_{B, B^\dagger} \| B - B^\dagger \| , 
  \notag 
\end{align} 
where the constants $C^{\Phi, \bar{x} }_{B, B^\dagger}$, $C^{\Phi, q }_{B, B^\dagger}$, $C^{\Phi, \Pi }_{B, B^\dagger}$, 
$C^{\Psi, \bar{x} }_{B, B^\dagger}$, $C^{\Psi, q }_{B, B^\dagger}$, and $C^{\Psi, \Pi }_{B, B^\dagger}$ are defined in the statement of the lemma, it then follows that  
\begin{align} 
 \big\| (\bar{x}^B - \bar{x}^\dagger )(t) \big\| 
 \leq &   \int_0^t M^{A^\dagger}_T   \big[
 C^{\Phi, \bar{x} }_{B, B^\dagger} \| ( \bar{x}^B - \bar{x}^\dagger )(r) \|  
  + C^{\Phi, q }_{B, B^\dagger} \| (q^B - q^\dagger )(T-r) \|  
  \notag \\ 
 & \hspace{2cm} + C^{\Phi, \Pi}_{B, B^\dagger} \| (\Pi^B - \Pi^\dagger)(T-r) \|
  + C^{\Phi,c }_{B, B^\dagger} \| B - B^\dagger \| \big]  dr , 
 \notag \\ 
%\end{align} 
%\begin{align} 
 \|(q^B - q^\dagger )(t)\| 
 \leq & 
   M^{A^\dagger}_T  \| (\bar{x}^B - \bar{x}^\dagger)(T) \|  \| G \|  \|\widehat{F}_2\| 
   \notag \\ 
&  + \int_0^t  M^{A^\dagger}_T \big[ C^{\Psi,\bar{x} }_{ B, B^\dagger} \| ( \bar{x}^B - \bar{x}^\dagger )(T-r) \|  
  + C^{\Psi, q }_{B, B^\dagger} \|( q^B - q^\dagger )(r)\| 
  \notag \\ 
 & \hspace{2cm} + C^{\Psi, \Pi}_{B, B^\dagger} \| ( \Pi^B - \Pi^\dagger)(r) \|
  + C^{\Psi,c }_{B, B^\dagger} \| B - B^\dagger \| \big]   dr . 
 \notag 
\end{align} 
Applying Lemma~\ref{lem:FBGronwall} to the above two inequalities, we obtain 
the desired estimates for 
$\|\bar{x}^B - \bar{x}^\dagger\|_{C(\mathcal{T}; H )} $ 
and 
$\| q^B - q^\dagger \|_{ C(\mathcal{T}; H) }$.

\end{proof}

\subsubsection{{Proof of Lemma~\ref{lem:B:|x1-x2|}} }\label{sec:lem:B:|x1-x2|}
\begin{proof}
The bound $C^x_B$ for $x^B$ follows from Lemma~\ref{lem:E|x|<=Cx} with $B^\dagger$ replaced by $B$. 

By~\eqref{x_MFG_eqm}, we have 
\begin{align} 
  (x^B - x^\dagger )(t) 
 \leq &   
   \int_0^t S^\dagger(t-r) ( \Xi_1^B - \Xi_1^\dagger )(r) dr
 %\notag \\ 
 %&    
 + \int_0^t S^\dagger(t-r) ( \Xi_2^B - \Xi_2^\dagger )(r) d W(r) , 
 \notag 
\end{align} 
where, for $i=1, 2$,   
\begin{align} 
  \mathbb{E} \| ( \Xi_i^B - \Xi_i^\dagger )(t) \|^2   
\leq & C^{\Xi_i, x}_{B, B^\dagger} \mathbb{E}\| (x^B - x^\dagger)(t) \|^2 
+ C^{\Xi_i, \Pi}_{B, B^\dagger } \| (\Pi^B - \Pi^\dagger)(t) \|^2
\notag \\ 
& 
+ C^{\Xi_i, \bar{x}}_{B, B^\dagger} \| \bar{x}^B - \bar{x}^\dagger \|_{C(\mathcal{T}; H)}^2 
+ C^{\Xi_i, q}_{B, B^\dagger } \| q^B - q^\dagger \|_{C(\mathcal{T}; H)}^2  
 + C^{\Xi_i, c}_{B, B^\dagger} \| B - B^\dagger \|^2 . 
\notag
\end{align} 

It then follows from the Cauchy-Schwarz inequality that 
\begin{align} 
  \mathbb{E}\|(x^B - x^\dagger)(t)\|^2  
\leq &  
 \int_0^t \| S^\dagger(t-r) \|^2 \big\{ \mathbb{E} \| (\Xi_1^B - \Xi_1^\dagger)(r) \|^2 
+ \mathbb{E} \| (\Xi_2^B - \Xi_2^\dagger )(r) \|^2 \big\} dr
\notag \\ 
%%%%%%%
\leq &  
  ( M^{A^\dagger}_T )^2 \int_0^t ( C^{\Xi_1, x}_{B, B^\dagger} + C^{\Xi_2, x}_{B, B^\dagger} ) 
   \mathbb{E} \|(x^B-x^\dagger)(r)\|^2 dr 
\notag \\
& + T ( M^{A^\dagger}_T )^2 
 \Big[ \big( C^{\Xi_1, \bar{x}}_{B, B^\dagger} + C^{\Xi_2, \bar{x}}_{B, B^\dagger} \big) 
  \| \bar{x}^B - \bar{x}^\dagger \|_{C(\mathcal{T}; H)}^2 
  \notag \\ 
&  + \big( C^{\Xi_1, q}_{B, B^\dagger } + C^{\Xi_2, q}_{B, B^\dagger } \big) \| q^B - q^\dagger \|_{C(\mathcal{T}; H)}^2  
  + \big( C^{\Xi_1, c}_{B, B^\dagger} +  C^{\Xi_2, c}_{B, B^\dagger} \big) \| B - B^\dagger \|^2 \Big] 
\notag . 
\end{align} 
By the Gr\"{o}nwall's inequality, we have that 
\begin{align} 
  \mathbb{E} \| (x^B - x^\dagger)(t) \|^2   
\leq & T ( M^{A^\dagger}_T )^2 
 \Big[ \big( C^{\Xi_1, \bar{x}}_{B, B^\dagger} + C^{\Xi_2, \bar{x}}_{B, B^\dagger} \big) 
 \| \bar{x}^B - \bar{x}^\dagger \|_{C(\mathcal{T}; H)}^2 
+ \big( C^{\Xi_1, q}_{B, B^\dagger } + C^{\Xi_2, q}_{B, B^\dagger } \big) \| q^B - q^\dagger \|_{C(\mathcal{T}; H)}^2 
\notag \\ 
& \hspace{0cm}  + \big( C^{\Xi_1, c}_{B, B^\dagger} +  C^{\Xi_2, c}_{B, B^\dagger} \big) \| B - B^\dagger \|^2 \Big] 
\exp\Big\{  T ( M^{A^\dagger}_T )^2  ( C^{\Xi_1, x}_{B, B^\dagger} + C^{\Xi_2, x}_{B, B^\dagger} )  \Big\} , 
\notag 
\end{align} 
which further implies that 
\begin{align} 
 \big( \sup_{t\in\mathcal{T}} 
  \mathbb{E} \| &(x^B - x^\dagger)(t) \|^2   \big)^{\frac{1}{2}}
\notag \\ 
\leq & T^{1/2}  M^{A^\dagger}_T \sqrt{3}  
 \Big[ \big( C^{\Xi_1, \bar{x}}_{B, B^\dagger} + C^{\Xi_2, \bar{x}}_{B, B^\dagger} \big)^{1/2} 
 \| \bar{x}^B - \bar{x}^\dagger \|_{C(\mathcal{T}; H)} 
+ \big( C^{\Xi_1, q}_{B, B^\dagger } + C^{\Xi_2, q}_{B, B^\dagger } \big)^{1/2} \| q^B - q^\dagger \|_{C(\mathcal{T}; H)} 
\notag \\ 
& \hspace{0cm}  + \big( C^{\Xi_1, c}_{B, B^\dagger} +  C^{\Xi_2, c}_{B, B^\dagger} \big)^{1/2} \| B - B^\dagger \| \Big] 
\exp\Big\{  T ( M^{A^\dagger}_T )^2  ( C^{\Xi_1, x}_{B, B^\dagger} + C^{\Xi_2, x}_{B, B^\dagger} ) /2  \Big\} . 
\notag 
\end{align} 
The desired estimate then follows. 
\end{proof}

%%%%%%%%%%%%%%%%%%%%%%%%%%%%%%%% 
\subsection{Proofs of Lipschitz Stability with respect to \texorpdfstring{$F_2$}{F2}}
\label{sec:Proof_Stability_F2}
\subsubsection{{Proof of Lemma~\ref{lem:F2:|q1-q2|and|barx1-barx2|}}}\label{sec:lem:F2:|q1-q2|and|barx1-barx2|}
\begin{proof} 
From~\eqref{eq:barxi_MFG_int} and~\eqref{eq:qi_MFG_int}, we have that   
\begin{align} 
& \big\| (\bar{x}^{F_2} - \bar{x}^\dagger )(t) \big\| 
 \leq    
   \int_0^t \big\| S^\dagger (t-r) \big\|   
 \big\| \Phi^{F_2}(r) - \Phi^\dagger (r) \big\| dr ,  
 \notag \\  
%\end{align}  
%\begin{align} 
 &\|(q^{F_2} - q^\dagger )(t)\| 
 \leq  
  \| S^\dagger \|  \| (\bar{x}^{F_2} - \bar{x}^\dagger )(T)\|  \| G \|  \|\widehat{F}_2\|  
  + \int_0^t  \|S^\dagger(t-r)\|  \|(\Psi^{F_2} - \Psi^\dagger )(r)\| dr .  
 \notag 
\end{align} 

\begin{align} 
 (\Phi^{F_2} - \Phi^\dagger)(r) 
 = & \big( B^\dagger (K^\dagger)^{-1} L^\dagger (T-r) - F_1 \big)(\bar{x}^{F_2} - \bar{x}^\dagger)(r) 
 + (B^\dagger)^\star ( q^{F_2}-q^\dagger) (T-r) 
 \notag \\ 
 & + B^\dagger (K^\dagger)^{-1}(T-r)\Gamma_2 \big( ( (F_2 - F_2^\dagger ) \bar{x}^{F_2} + F_2^\dagger ( \bar{x}^{F_2} - \bar{x}^\dagger ) )^\star \Pi^\dagger(T-r) \big)  , 
 \notag \\ 
 %%%%% 
 (\Psi^{F_2} - \Psi^\dagger)(r) = & - (L^\dagger)^\star (K^\dagger)^{-1}(r) B^\star (q^{F_2} - q^\dagger)(r) 
 \notag \\ 
 & + \Gamma_1 \big[  \big(  (F_2 - F_2^\dagger ) \bar{x}^{F_2} + F_2^\dagger ( \bar{x}^{F_2} - \bar{x}^\dagger )  \big)^\star (T-r) \Pi^\dagger(r) \big]  
 \notag \\ 
 & - (L^\dagger)^\star (K^\dagger)^{-1}(r) \Gamma_2 \big[  \big(  (F_2 - F_2^\dagger ) \bar{x} + F_2^\dagger ( \bar{x}^{F_2} - \bar{x}^\dagger )  \big)^\star (T-r) \Pi^\dagger(r) \big] 
 \notag \\ 
& + (\Pi^\dagger(r)F_1 - M \widehat{F}_1 ) ( \bar{x}^{F_2} - \bar{x}^\dagger )(T-r) . 
 \notag 
\end{align} 
Since 
\begin{align} 
 &  \big\| \Phi^{F_2}(r) - \Phi^\dagger(r) \big\|  
  \leq  C^{\Phi, \bar{x} }_{F_2, F_2^\dagger} \| ( \bar{x}^{F_2} - \bar{x}^\dagger )(r) \|  
  + C^{\Phi, q }_{F_2, F_2^\dagger} \| (q^{F_2} - q^\dagger )(T-r) \|      
  + C^{\Phi,c }_{F_2, F_2^\dagger} \| F_2 - F_2^\dagger \| , 
  \notag \\ 
%\end{align}  
%\begin{align} 
 &  \big\| \Psi^{F_2}(r) - \Psi^\dagger(r) \big\|  
  \leq  C^{\Psi,\bar{x} }_{F_2, F_2^\dagger} \| ( \bar{x}^{F_2} - \bar{x}^\dagger )(T-r) \|  
  + C^{\Psi, q }_{F_2, F_2^\dagger} \| (q^{F_2} - q^\dagger )(r) \|  
  + C^{\Psi,c }_{F_2, F_2^\dagger} \| F_2 - F_2^\dagger \| , 
  \notag 
\end{align} 
where the constants $C^{\Phi, \bar{x} }_{F_2, F_2^\dagger}$, $C^{\Phi, q }_{F_2, F_2^\dagger}$, $C^{\Phi, \Pi }_{F_2, F_2^\dagger}$, 
$C^{\Psi, \bar{x} }_{F_2, F_2^\dagger}$, $C^{\Psi, q }_{F_2, F_2^\dagger}$, and $C^{\Psi, \Pi }_{F_2, F_2^\dagger}$ are given in the statement of the lemma, it then follows that   
\begin{align} 
 \big\| (\bar{x}^{F_2} - \bar{x}^\dagger )(t) \big\| 
 \leq & 
    \int_0^t\!\! M^{A^\dagger}_T  
\big( C^{\Phi, \bar{x} }_{F_2,F_2^\dagger} \| (\bar{x}^{F_2} - \bar{x}^\dagger )(r) \|  
  + C^{\Phi, q }_{F_2, F_2^\dagger} \| (q^{F_2} - q^\dagger)(T-r) \|  
    + C^{\Phi, c}_{F_2, F_2^\dagger} \| F_2 - F_2^\dagger \| 
  \big) dr , 
   \notag \\ 
%\end{align} 
%\begin{align} 
\big\|(q^{F_2} - q^\dagger )(t)\big\| 
\leq &  M^{A^\dagger }_T \| (\bar{x}^{F_2} - \bar{x}^\dagger)(T) \|  \|G\| \|\widehat{F}_2 \| 
\notag \\ 
 & \hspace{-0.5cm} + \int_0^t M^{A^\dagger}_T  \big( 
 C^{\Psi, q}_{F_2, F_2^\dagger } \| (q^{F_2} - q^\dagger )(r) \| 
 + C^{\Psi, \bar{x}}_{F_2, F_2^\dagger } \| (\bar{x}^{F_2} - \bar{x}^\dagger) (T-r) \|   
 + C^{\Psi, c}_{F_2, F_2^\dagger} \| F_2 - F_2^\dagger \| 
 \big) dr . 
  \notag 
\end{align} 
Applying Lemma~\ref{lem:FBGronwall} to the above two inequalities, we obtain 
the desired estimates for 
$\|\bar{x}^{F_2} - \bar{x}^\dagger\|_{C(\mathcal{T}; H)} $ 
and 
$\| q^{F_2} - q^\dagger \|_{ C(\mathcal{T}; H) }$. 
\end{proof} 

\subsubsection{Proof of Lemma~\ref{lem:F2:|x1-x2|}}\label{sec:lem:F2:|x1-x2|} 
\begin{proof}
The bound $C^x_{F_2}$ for $x^{F_2}$ follows from Lemma~\ref{lem:E|x|<=Cx} with $F_2^\dagger$ replaced by $F_2$. 

By~\eqref{x_MFG_eqm}, we have 
\begin{equation} 
  (x^{F_2} - x^\dagger )(t) 
 \leq 
   \int_0^t S^\dagger(t-r) ( \Xi_1^{F_2} - \Xi_1^\dagger )(r) dr
 %\notag \\ 
 %&    
 + \int_0^t S^\dagger(t-r) ( \Xi_2^{F_2} - \Xi_2^\dagger )(r) d W(r) . 
 \notag 
\end{equation} 
Since 
\begin{align} 
( \Xi_1^{F_2} - \Xi_1^\dagger)(r) = & 
  B^\dagger (K^\dagger)^{-1}L^\dagger (x^{F_2}-x^\dagger)(r)  
  + B^\dagger (\tau^{F_2} - \tau^\dagger) (r) - F_1 (\bar{x}^{F_2} - \bar{x}^\dagger ) , 
 \notag \\ 
 %%%%%
 ( \Xi_2^{F_2} - \Xi_2^\dagger )(r) = &  
  F_2 ( \bar{x}^{F_2} - \bar{x}^\dagger )(r) 
  + [D - E ( K^\dagger )^{-1} L^\dagger (T-r) ] 
  ( x^{F_2} - x^\dagger )(r)  
  - E(\tau^{F_2} - \tau^\dagger )(r) , 
 \notag 
\end{align} 
we have that for $i=1, 2$,   
\begin{align} 
  \mathbb{E} \| &( \Xi_i^{F_2} - \Xi_i^\dagger )(t) \|^2   \notag \\ 
\leq & C^{\Xi_i, x}_{F_2, F_2^\dagger} \mathbb{E}\| (x^{F_2} - x^\dagger)(t) \|^2  
%\notag \\ 
%& 
+ C^{\Xi_i, \bar{x}}_{F_2, F_2^\dagger} \| \bar{x}^{F_2} - \bar{x}^\dagger \|_{C(\mathcal{T}; H)}^2 
+ C^{\Xi_i, q}_{F_2, F_2^\dagger } \| q^{F_2} - q^\dagger \|_{C(\mathcal{T}; H)}^2  
+ C^{\Xi_i, c}_{F_2, F_2^\dagger} \|F_2 - F_2^\dagger \|^2 
\notag \\ 
%%%%% 
\leq & C^{\Xi_i, x}_{F_2, F_2^\dagger} \mathbb{E} \| (x^{F_2} - x^\dagger)(t) \|^2 
%\notag \\ 
%& 
+ \Big[   
 C^{\Xi_i, \bar{x}}_{F_2, F_2^\dagger} C^{\bar{x}}_{F_2, F_2^\dagger}  
+ C^{\Xi_i, q}_{F_2, F_2^\dagger } C^q_{F_2, F_2^\dagger} 
 + C^{\Xi_i,c}_{F_2, F_2^\dagger}
\Big] \| F_2 - F_2^\dagger \|^2 . 
\notag 
\end{align} 
It follows from the Cauchy-Schwartz inequality that 
\begin{align} 
 \mathbb{E} \|(x^{F_2} - x^\dagger)(t) \|^2 
\leq & 
 \int_0^t \| S^\dagger(t-r) \|^2 \big\{ \mathbb{E} \| (\Xi_1^{F_2} - \Xi_1^\dagger)(r) \|^2 
+ \mathbb{E} \| (\Xi_2^{F_2} - \Xi_2^\dagger )(r) \|^2 \big\} dr
\notag \\ 
%%%%%%%
\leq &  
  ( M^{A^\dagger}_T )^2 \int_0^t ( C^{\Xi_1, x}_{F_2, F_2^\dagger} + C^{\Xi_2, x}_{F_2, F_2^\dagger} ) 
   \mathbb{E} \|(x-x^\dagger)(r)\|^2 dr + T ( M^{A^\dagger}_T )^2 
 \Big[ ( C^{\Xi_1, q}_{F_2, F_2^\dagger}
\notag \\
&  + C^{\Xi_2, q}_{F_2, F_2^\dagger} ) C^q_{F_2, F_2^\dagger}  
+ ( C^{\Xi_1, \bar{x}}_{F_2, F_2^\dagger} + C^{\Xi_2, \bar{x}}_{F_2, F_2^\dagger} ) C^{\bar{x}}_{F_2, F_2^\dagger} 
+ C^{\Xi_1, c}_{F_2, F_2^\dagger} + C^{\Xi_2, c}_{F_2, F_2^\dagger}
\Big] \| F_2 - F_2^\dagger \|^2
\notag . 
\notag 
\end{align} 
By the Gr\"{o}nwall's inequality, we have 
\begin{align} 
\mathbb{E}\|(x^{F_2} - x^\dagger)(t)\|^2 \notag 
\leq& T ( M^{A^\dagger}_T )^2 
 \Big[ ( C^{\Xi_1, q}_{F_2, F_2^\dagger} + C^{\Xi_2, q}_{F_2, F_2^\dagger} ) C^q_{F_2, F_2^\dagger}  
+ ( C^{\Xi_1, \bar{x}}_{F_2, F_2^\dagger} + C^{\Xi_2, \bar{x}}_{F_2, F_2^\dagger} ) C^{\bar{x}}_{F_2, F_2^\dagger} 
\notag \\ 
& + C^{\Xi_1, c}_{F_2, F_2^\dagger} + C^{\Xi_2, c}_{F_2, F_2^\dagger}
\Big]  \| F_2 - F_2^\dagger \|^2  \exp \big[  T ( M^{A^\dagger}_T )^2 ( C^{\Xi_1, x}_{F_2, F_2^\dagger} + C^{\Xi_2, x}_{F_2, F_2^\dagger} ) \big] 
.
  \notag 
\end{align} 
It then follows that 
\begin{align} 
\big( \sup_{t\in \mathcal{T}} \mathbb{E}\|(x^{F_2} - x^\dagger)(t)\|^2 \big)^{\frac{1}{2}} \leq &  \sqrt{T}  M^{A^\dagger}_T  
 \Big[ ( C^{\Xi_1, q}_{F_2, F_2^\dagger} + C^{\Xi_2, q}_{F_2, F_2^\dagger} ) C^q_{F_2, F_2^\dagger}  
+ ( C^{\Xi_1, \bar{x}}_{F_2, F_2^\dagger} + C^{\Xi_2, \bar{x}}_{F_2, F_2^\dagger} ) C^{\bar{x}}_{F_2, F_2^\dagger} 
\notag \\ 
& + C^{\Xi_1, c}_{F_2, F_2^\dagger} + C^{\Xi_2, c}_{F_2, F_2^\dagger}
\Big]^{\frac{1}{2}} \| F_2 - F_2^\dagger \|  \exp \big[  T ( M^{A^\dagger}_T )^2 ( C^{\Xi_1, x}_{F_2, F_2^\dagger} + C^{\Xi_2, x}_{F_2, F_2^\dagger} )/2 \big] 
.
  \notag 
\end{align}

\end{proof}

%% file: 0MAIN_FILE__MFG_DL_arXiv.bbl
\begin{thebibliography}{58}
\expandafter\ifx\csname natexlab\endcsname\relax\def\natexlab#1{#1}\fi
\providecommand{\url}[1]{\texttt{#1}}
\providecommand{\href}[2]{#2}
\providecommand{\path}[1]{#1}
\providecommand{\DOIprefix}{doi:}
\providecommand{\ArXivprefix}{arXiv:}
\providecommand{\URLprefix}{URL: }
\providecommand{\Pubmedprefix}{pmid:}
\providecommand{\doi}[1]{\href{http://dx.doi.org/#1}{\path{#1}}}
\providecommand{\Pubmed}[1]{\href{pmid:#1}{\path{#1}}}
\providecommand{\bibinfo}[2]{#2}
\ifx\xfnm\relax \def\xfnm[#1]{\unskip,\space#1}\fi
%Type = Article
\bibitem[{Bauso et~al.(2016)Bauso, Tembine, and Basar}]{bauso2016opinion}
\bibinfo{author}{D.~Bauso}, \bibinfo{author}{H.~Tembine}, \bibinfo{author}{T.~Basar},
\newblock \bibinfo{title}{Opinion dynamics in social networks through mean-field games},
\newblock \bibinfo{journal}{SIAM Journal on Control and Optimization} \bibinfo{volume}{54} (\bibinfo{year}{2016}) \bibinfo{pages}{3225--3257}.
%Type = Article
\bibitem[{Shrivats et~al.(2022)Shrivats, Firoozi, and Jaimungal}]{Firoozi2022MAFI}
\bibinfo{author}{A.~V. Shrivats}, \bibinfo{author}{D.~Firoozi}, \bibinfo{author}{S.~Jaimungal},
\newblock \bibinfo{title}{A mean-field game approach to equilibrium pricing in solar renewable energy certificate markets},
\newblock \bibinfo{journal}{Mathematical Finance} \bibinfo{volume}{32} (\bibinfo{year}{2022}) \bibinfo{pages}{779--824}.
%Type = Article
\bibitem[{Gomes and Saúde(2021)}]{gomes_mean-field_2018}
\bibinfo{author}{D.~Gomes}, \bibinfo{author}{J.~Saúde},
\newblock \bibinfo{title}{A mean-field game approach to price formation in electricity markets},
\newblock \bibinfo{journal}{Dynamic Games and Applications} \bibinfo{volume}{11} (\bibinfo{year}{2021}) \bibinfo{pages}{29--53}.
%Type = Article
\bibitem[{Fujii and Takahashi(2022)}]{fujii_mean_2021}
\bibinfo{author}{M.~Fujii}, \bibinfo{author}{A.~Takahashi},
\newblock \bibinfo{title}{A mean field game approach to equilibrium pricing with market clearing condition},
\newblock \bibinfo{journal}{SIAM Journal on Control and Optimization} \bibinfo{volume}{60} (\bibinfo{year}{2022}) \bibinfo{pages}{259--279}.
%Type = Article
\bibitem[{Laguzet and Turinici(2015)}]{laguzet2015individual}
\bibinfo{author}{L.~Laguzet}, \bibinfo{author}{G.~Turinici},
\newblock \bibinfo{title}{Individual vaccination as nash equilibrium in a sir model with application to the 2009--2010 influenza a (h1n1) epidemic in france},
\newblock \bibinfo{journal}{Bulletin of mathematical biology} \bibinfo{volume}{77} (\bibinfo{year}{2015}) \bibinfo{pages}{1955--1984}.
%Type = Article
\bibitem[{Bauso et~al.(2016)Bauso, Zhang, and Papachristodoulou}]{bauso2016density}
\bibinfo{author}{D.~Bauso}, \bibinfo{author}{X.~Zhang}, \bibinfo{author}{A.~Papachristodoulou},
\newblock \bibinfo{title}{Density flow in dynamical networks via mean-field games},
\newblock \bibinfo{journal}{IEEE Transactions on Automatic Control} \bibinfo{volume}{62} (\bibinfo{year}{2016}) \bibinfo{pages}{1342--1355}.
%Type = Phdthesis
\bibitem[{Del~Moral et~al.(2008)Del~Moral, Patras, and Rubenthaler}]{del2008mean}
\bibinfo{author}{P.~Del~Moral}, \bibinfo{author}{F.~Patras}, \bibinfo{author}{S.~Rubenthaler}, \bibinfo{title}{A mean field theory of nonlinear filtering}, Ph.D. thesis, INRIA, \bibinfo{year}{2008}.
%Type = Article
\bibitem[{Carrillo et~al.(2024)Carrillo, Hoffmann, Stuart, and Vaes}]{carrillo2024mean}
\bibinfo{author}{J.~A. Carrillo}, \bibinfo{author}{F.~Hoffmann}, \bibinfo{author}{A.~M. Stuart}, \bibinfo{author}{U.~Vaes},
\newblock \bibinfo{title}{The mean-field ensemble kalman filter: near-gaussian setting},
\newblock \bibinfo{journal}{SIAM Journal on Numerical Analysis} \bibinfo{volume}{62} (\bibinfo{year}{2024}) \bibinfo{pages}{2549--2587}.
%Type = Article
\bibitem[{Ertel(2025)}]{ertel2025mean}
\bibinfo{author}{S.~W. Ertel},
\newblock \bibinfo{title}{On the mean field theory of ensemble kalman filters for spdes},
\newblock \bibinfo{journal}{SIAM/ASA Journal on Uncertainty Quantification} \bibinfo{volume}{13} (\bibinfo{year}{2025}) \bibinfo{pages}{891--930}.
%Type = Inproceedings
\bibitem[{Mei et~al.(2019)Mei, Misiakiewicz, and Montanari}]{mei2019mean}
\bibinfo{author}{S.~Mei}, \bibinfo{author}{T.~Misiakiewicz}, \bibinfo{author}{A.~Montanari},
\newblock \bibinfo{title}{Mean-field theory of two-layers neural networks: dimension-free bounds and kernel limit},
\newblock in: \bibinfo{booktitle}{Conference on learning theory}, \bibinfo{organization}{PMLR}, \bibinfo{year}{2019}, pp. \bibinfo{pages}{2388--2464}.
%Type = Article
\bibitem[{Huang et~al.(2006)Huang, Malham\'e, and Caines}]{huang2006large}
\bibinfo{author}{M.~Huang}, \bibinfo{author}{R.~P. Malham\'e}, \bibinfo{author}{P.~E. Caines},
\newblock \bibinfo{title}{Large population stochastic dynamic games: closed-loop {M}c{K}ean-{V}lasov systems and the {N}ash certainty equivalence principle},
\newblock \bibinfo{journal}{Commun. Inf. Syst.} \bibinfo{volume}{6} (\bibinfo{year}{2006}) \bibinfo{pages}{221--251}. \URLprefix \url{https://doi.org/10.4310/cis.2006.v6.n3.a5}. \DOIprefix\doi{10.4310/cis.2006.v6.n3.a5}.
%Type = Article
\bibitem[{Lasry and Lions(2007)}]{lasry2007mean}
\bibinfo{author}{J.-M. Lasry}, \bibinfo{author}{P.-L. Lions},
\newblock \bibinfo{title}{Mean field games},
\newblock \bibinfo{journal}{Japanese journal of mathematics} \bibinfo{volume}{2} (\bibinfo{year}{2007}) \bibinfo{pages}{229--260}.
%Type = Article
\bibitem[{Nguyen and Huang(2012)}]{Huang-Nguyen-2012}
\bibinfo{author}{S.~L. Nguyen}, \bibinfo{author}{M.~Huang},
\newblock \bibinfo{title}{Linear-quadratic-gaussian mixed games with continuum-parametrized minor players},
\newblock \bibinfo{journal}{SIAM Journal on Control and Optimization} \bibinfo{volume}{50} (\bibinfo{year}{2012}) \bibinfo{pages}{2907--2937}. \DOIprefix\doi{10.1137/110841217}.
%Type = Article
\bibitem[{Liu and Firoozi(2025)}]{LF24}
\bibinfo{author}{H.~Liu}, \bibinfo{author}{D.~Firoozi},
\newblock \bibinfo{title}{Hilbert space-valued {LQ} mean field games: An infinite-dimensional analysis},
\newblock \bibinfo{journal}{SIAM Journal on Control and Optimization} \bibinfo{volume}{63} (\bibinfo{year}{2025}) \bibinfo{pages}{3297--3327}.
%Type = Misc
\bibitem[{Federico et~al.(2024)Federico, Gozzi, and Ghilli}]{FGG24}
\bibinfo{author}{S.~Federico}, \bibinfo{author}{F.~Gozzi}, \bibinfo{author}{D.~Ghilli}, \bibinfo{title}{{Linear-Quadratic} mean field games in {Hilbert} spaces}, \bibinfo{howpublished}{arXiv:2402.14935}, \bibinfo{year}{2024}.
%Type = Article
\bibitem[{Huang et~al.(2007)Huang, Caines, and Malham{\'e}}]{huang2007large}
\bibinfo{author}{M.~Huang}, \bibinfo{author}{P.~E. Caines}, \bibinfo{author}{R.~P. Malham{\'e}},
\newblock \bibinfo{title}{Large-population cost-coupled {LQG} problems with nonuniform agents: individual-mass behavior and decentralized $\varepsilon$-{Nash} equilibria},
\newblock \bibinfo{journal}{IEEE transactions on automatic control} \bibinfo{volume}{52} (\bibinfo{year}{2007}) \bibinfo{pages}{1560--1571}.
%Type = Article
\bibitem[{Huang(2010)}]{huang2010large}
\bibinfo{author}{M.~Huang},
\newblock \bibinfo{title}{Large-population {{LQG}} games involving a major player: the {Nash} certainty equivalence principle},
\newblock \bibinfo{journal}{SIAM Journal on Control and Optimization} \bibinfo{volume}{48} (\bibinfo{year}{2010}) \bibinfo{pages}{3318--3353}.
%Type = Article
\bibitem[{Briceno-Arias et~al.(2019)Briceno-Arias, Kalise, Kobeissi, Lauriere, Gonz{\'a}lez, and Silva}]{briceno2019implementation}
\bibinfo{author}{L.~Briceno-Arias}, \bibinfo{author}{D.~Kalise}, \bibinfo{author}{Z.~Kobeissi}, \bibinfo{author}{M.~Lauriere}, \bibinfo{author}{A.~M. Gonz{\'a}lez}, \bibinfo{author}{F.~J. Silva},
\newblock \bibinfo{title}{On the implementation of a primal-dual algorithm for second order time-dependent mean field games with local couplings},
\newblock \bibinfo{journal}{ESAIM: Proceedings and Surveys} \bibinfo{volume}{65} (\bibinfo{year}{2019}) \bibinfo{pages}{330--348}.
%Type = Article
\bibitem[{Angiuli et~al.(2019)Angiuli, Graves, Li, Chassagneux, Delarue, and Carmona}]{angiuli2019cemracs}
\bibinfo{author}{A.~Angiuli}, \bibinfo{author}{C.~V. Graves}, \bibinfo{author}{H.~Li}, \bibinfo{author}{J.-F. Chassagneux}, \bibinfo{author}{F.~Delarue}, \bibinfo{author}{R.~Carmona},
\newblock \bibinfo{title}{Cemracs 2017: numerical probabilistic approach to mfg},
\newblock \bibinfo{journal}{ESAIM: Proceedings and Surveys} \bibinfo{volume}{65} (\bibinfo{year}{2019}) \bibinfo{pages}{84--113}.
%Type = Article
\bibitem[{Camilli and Tang(2023)}]{camilli2023convergence}
\bibinfo{author}{F.~Camilli}, \bibinfo{author}{Q.~Tang},
\newblock \bibinfo{title}{A convergence rate for the newton’s method for mean field games with non-separable hamiltonians},
\newblock \bibinfo{journal}{arXiv preprint arXiv:2311.05416}  (\bibinfo{year}{2023}).
%Type = Article
\bibitem[{Fouque and Zhang(2020)}]{fouque2020deep}
\bibinfo{author}{J.-P. Fouque}, \bibinfo{author}{Z.~Zhang},
\newblock \bibinfo{title}{Deep learning methods for mean field control problems with delay},
\newblock \bibinfo{journal}{Frontiers in Applied Mathematics and Statistics} \bibinfo{volume}{6} (\bibinfo{year}{2020}) \bibinfo{pages}{11}.
%Type = Article
\bibitem[{Carmona and Lauri{\`e}re(2021)}]{carmona2021convergence}
\bibinfo{author}{R.~Carmona}, \bibinfo{author}{M.~Lauri{\`e}re},
\newblock \bibinfo{title}{Convergence analysis of machine learning algorithms for the numerical solution of mean field control and games i: The ergodic case},
\newblock \bibinfo{journal}{SIAM Journal on Numerical Analysis} \bibinfo{volume}{59} (\bibinfo{year}{2021}) \bibinfo{pages}{1455--1485}.
%Type = Article
\bibitem[{Germain et~al.(2022)Germain, Mikael, and Warin}]{germain2022numerical}
\bibinfo{author}{M.~Germain}, \bibinfo{author}{J.~Mikael}, \bibinfo{author}{X.~Warin},
\newblock \bibinfo{title}{Numerical resolution of mckean-vlasov fbsdes using neural networks},
\newblock \bibinfo{journal}{Methodology and Computing in Applied Probability} \bibinfo{volume}{24} (\bibinfo{year}{2022}) \bibinfo{pages}{2557--2586}.
%Type = Article
\bibitem[{Cao et~al.(2024)Cao, Guo, and Lauri{\`e}re}]{cao2024connecting}
\bibinfo{author}{H.~Cao}, \bibinfo{author}{X.~Guo}, \bibinfo{author}{M.~Lauri{\`e}re},
\newblock \bibinfo{title}{Connecting gans, mean-field games, and optimal transport},
\newblock \bibinfo{journal}{SIAM Journal on Applied Mathematics} \bibinfo{volume}{84} (\bibinfo{year}{2024}) \bibinfo{pages}{1255--1287}.
%Type = Article
\bibitem[{Soner et~al.(2025)Soner, Teichmann, and Yan}]{soner2025learning}
\bibinfo{author}{H.~M. Soner}, \bibinfo{author}{J.~Teichmann}, \bibinfo{author}{Q.~Yan},
\newblock \bibinfo{title}{Learning algorithms for mean field optimal control},
\newblock \bibinfo{journal}{arXiv preprint arXiv:2503.17869}  (\bibinfo{year}{2025}).
%Type = Book
\bibitem[{Lorentz et~al.(1996)Lorentz, Golitschek, and Makovoz}]{lorentz1996constructive}
\bibinfo{author}{G.~G. Lorentz}, \bibinfo{author}{M.~v. Golitschek}, \bibinfo{author}{Y.~Makovoz}, \bibinfo{title}{Constructive approximation: advanced problems}, volume \bibinfo{volume}{304} of \textit{\bibinfo{series}{Grundlehren der mathematischen Wissenschaften [Fundamental Principles of Mathematical Sciences]}}, \bibinfo{publisher}{Springer-Verlag, Berlin}, \bibinfo{year}{1996}. \URLprefix \url{https://doi.org/10.1007/978-3-642-60932-9}. \DOIprefix\doi{10.1007/978-3-642-60932-9}, \bibinfo{note}{advanced problems}.
%Type = Article
\bibitem[{Cuchiero and Teichmann(2019)}]{cuchiero2019markovian}
\bibinfo{author}{C.~Cuchiero}, \bibinfo{author}{J.~Teichmann},
\newblock \bibinfo{title}{Markovian lifts of positive semidefinite affine volterra-type processes},
\newblock \bibinfo{journal}{Decisions in Economics and Finance} \bibinfo{volume}{42} (\bibinfo{year}{2019}) \bibinfo{pages}{407--448}.
%Type = Article
\bibitem[{Cuchiero and Teichmann(2020)}]{cuchiero2020generalized}
\bibinfo{author}{C.~Cuchiero}, \bibinfo{author}{J.~Teichmann},
\newblock \bibinfo{title}{Generalized feller processes and markovian lifts of stochastic volterra processes: the affine case},
\newblock \bibinfo{journal}{Journal of evolution equations} \bibinfo{volume}{20} (\bibinfo{year}{2020}) \bibinfo{pages}{1301--1348}.
%Type = Article
\bibitem[{Hamaguchi(2024)}]{hamaguchi2024markovian}
\bibinfo{author}{Y.~Hamaguchi},
\newblock \bibinfo{title}{Markovian lifting and asymptotic log-harnack inequality for stochastic volterra integral equations},
\newblock \bibinfo{journal}{Stochastic Processes and their Applications} \bibinfo{volume}{178} (\bibinfo{year}{2024}) \bibinfo{pages}{104482}.
%Type = Article
\bibitem[{Wang et~al.(2021)Wang, Wang, and Perdikaris}]{wang2021learning}
\bibinfo{author}{S.~Wang}, \bibinfo{author}{H.~Wang}, \bibinfo{author}{P.~Perdikaris},
\newblock \bibinfo{title}{Learning the solution operator of parametric partial differential equations with physics-informed deeponets},
\newblock \bibinfo{journal}{Science advances} \bibinfo{volume}{7} (\bibinfo{year}{2021}) \bibinfo{pages}{eabi8605}.
%Type = Article
\bibitem[{De~Ryck and Mishra(2022)}]{de2022generic}
\bibinfo{author}{T.~De~Ryck}, \bibinfo{author}{S.~Mishra},
\newblock \bibinfo{title}{Generic bounds on the approximation error for physics-informed (and) operator learning},
\newblock \bibinfo{journal}{Advances in Neural Information Processing Systems} \bibinfo{volume}{35} (\bibinfo{year}{2022}) \bibinfo{pages}{10945--10958}.
%Type = Article
\bibitem[{de~Hoop et~al.(2022)de~Hoop, Lassas, and Wong}]{de2022deep}
\bibinfo{author}{M.~V. de~Hoop}, \bibinfo{author}{M.~Lassas}, \bibinfo{author}{C.~A. Wong},
\newblock \bibinfo{title}{Deep learning architectures for nonlinear operator functions and nonlinear inverse problems},
\newblock \bibinfo{journal}{Mathematical Statistics and Learning} \bibinfo{volume}{4} (\bibinfo{year}{2022}) \bibinfo{pages}{1--86}.
%Type = Incollection
\bibitem[{Goswami et~al.(2023)Goswami, Bora, Yu, and Karniadakis}]{goswami2023physics}
\bibinfo{author}{S.~Goswami}, \bibinfo{author}{A.~Bora}, \bibinfo{author}{Y.~Yu}, \bibinfo{author}{G.~E. Karniadakis},
\newblock \bibinfo{title}{Physics-informed deep neural operator networks},
\newblock in: \bibinfo{booktitle}{Machine learning in modeling and simulation: methods and applications}, \bibinfo{publisher}{Springer}, \bibinfo{year}{2023}, pp. \bibinfo{pages}{219--254}.
%Type = Article
\bibitem[{Benitez et~al.(2024)Benitez, Furuya, Faucher, Kratsios, Tricoche, and de~Hoop}]{benitez2024out}
\bibinfo{author}{J.~A.~L. Benitez}, \bibinfo{author}{T.~Furuya}, \bibinfo{author}{F.~Faucher}, \bibinfo{author}{A.~Kratsios}, \bibinfo{author}{X.~Tricoche}, \bibinfo{author}{M.~V. de~Hoop},
\newblock \bibinfo{title}{Out-of-distributional risk bounds for neural operators with applications to the {Helmholtz} equation},
\newblock \bibinfo{journal}{Journal of Computational Physics}  (\bibinfo{year}{2024}) \bibinfo{pages}{113168}.
%Type = Article
\bibitem[{Li et~al.(2024)Li, Zheng, Kovachki, Jin, Chen, Liu, Azizzadenesheli, and Anandkumar}]{li2024physics}
\bibinfo{author}{Z.~Li}, \bibinfo{author}{H.~Zheng}, \bibinfo{author}{N.~Kovachki}, \bibinfo{author}{D.~Jin}, \bibinfo{author}{H.~Chen}, \bibinfo{author}{B.~Liu}, \bibinfo{author}{K.~Azizzadenesheli}, \bibinfo{author}{A.~Anandkumar},
\newblock \bibinfo{title}{Physics-informed neural operator for learning partial differential equations},
\newblock \bibinfo{journal}{ACM/JMS Journal of Data Science} \bibinfo{volume}{1} (\bibinfo{year}{2024}) \bibinfo{pages}{1--27}.
%Type = Article
\bibitem[{Azizzadenesheli et~al.(2024)Azizzadenesheli, Kovachki, Li, Liu-Schiaffini, Kossaifi, and Anandkumar}]{azizzadenesheli2024neural}
\bibinfo{author}{K.~Azizzadenesheli}, \bibinfo{author}{N.~Kovachki}, \bibinfo{author}{Z.~Li}, \bibinfo{author}{M.~Liu-Schiaffini}, \bibinfo{author}{J.~Kossaifi}, \bibinfo{author}{A.~Anandkumar},
\newblock \bibinfo{title}{Neural operators for accelerating scientific simulations and design},
\newblock \bibinfo{journal}{Nature Reviews Physics} \bibinfo{volume}{6} (\bibinfo{year}{2024}) \bibinfo{pages}{320--328}.
%Type = Article
\bibitem[{Alvarez et~al.(2024)Alvarez, Ekren, Kratsios, and Yang}]{alvarez2024neural}
\bibinfo{author}{G.~Alvarez}, \bibinfo{author}{I.~Ekren}, \bibinfo{author}{A.~Kratsios}, \bibinfo{author}{X.~Yang},
\newblock \bibinfo{title}{Neural operators can play dynamic stackelberg games},
\newblock \bibinfo{journal}{arXiv preprint arXiv:2411.09644}  (\bibinfo{year}{2024}).
%Type = Article
\bibitem[{Huang and Lai(2025)}]{HUANG2025114057}
\bibinfo{author}{H.~Huang}, \bibinfo{author}{R.~Lai},
\newblock \bibinfo{title}{Unsupervised solution operator learning for mean-field games},
\newblock \bibinfo{journal}{Journal of Computational Physics} \bibinfo{volume}{537} (\bibinfo{year}{2025}) \bibinfo{pages}{114057}. \URLprefix \url{https://www.sciencedirect.com/science/article/pii/S0021999125003407}. \DOIprefix\doi{https://doi.org/10.1016/j.jcp.2025.114057}.
%Type = Article
\bibitem[{Chen et~al.(2024)Chen, Liu, and Di}]{chen2024physics}
\bibinfo{author}{X.~Chen}, \bibinfo{author}{S.~Liu}, \bibinfo{author}{X.~Di},
\newblock \bibinfo{title}{Physics-informed graph neural operator for mean field games on graph: A scalable learning approach},
\newblock \bibinfo{journal}{Games} \bibinfo{volume}{15} (\bibinfo{year}{2024}) \bibinfo{pages}{12}.
%Type = Article
\bibitem[{Ichikawa(1979)}]{ichikawa1979dynamic}
\bibinfo{author}{A.~Ichikawa},
\newblock \bibinfo{title}{Dynamic programming approach to stochastic evolution equations},
\newblock \bibinfo{journal}{SIAM Journal on Control and Optimization} \bibinfo{volume}{17} (\bibinfo{year}{1979}) \bibinfo{pages}{152--174}.
%Type = Article
\bibitem[{Tessitore(1992)}]{tessitore1992some}
\bibinfo{author}{G.~Tessitore},
\newblock \bibinfo{title}{Some remarks on the {R}iccati equation arising in an optimal control problem with state-and control-dependent noise},
\newblock \bibinfo{journal}{SIAM Journal on Control and Optimization} \bibinfo{volume}{30} (\bibinfo{year}{1992}) \bibinfo{pages}{717--744}.
%Type = Article
\bibitem[{Dunyak and Caines(2024)}]{dunyak2024quadratic}
\bibinfo{author}{A.~Dunyak}, \bibinfo{author}{P.~E. Caines},
\newblock \bibinfo{title}{Quadratic optimal control of graphon q-noise linear systems},
\newblock \bibinfo{journal}{arXiv preprint arXiv:2407.00212}  (\bibinfo{year}{2024}).
%Type = Article
\bibitem[{Hong and Kratsios(2024)}]{hong2024bridging}
\bibinfo{author}{R.~Hong}, \bibinfo{author}{A.~Kratsios},
\newblock \bibinfo{title}{Bridging the gap between approximation and learning via optimal approximation by relu {MLPs} of maximal regularity},
\newblock \bibinfo{journal}{arXiv preprint arXiv:2409.12335}  (\bibinfo{year}{2024}).
%Type = Book
\bibitem[{Goldstein(2017)}]{goldstein2017semigroups}
\bibinfo{author}{J.~A. Goldstein}, \bibinfo{title}{Semigroups of linear operators and applications}, \bibinfo{edition}{second} ed., \bibinfo{publisher}{Dover Publications, Inc., Mineola, NY}, \bibinfo{year}{2017}.
%Type = Book
\bibitem[{Benyamini and Lindenstrauss(2000)}]{BenyaminiLindenstrauss_2000_NonlinearFunctionalAnalysis}
\bibinfo{author}{Y.~Benyamini}, \bibinfo{author}{J.~Lindenstrauss}, \bibinfo{title}{Geometric nonlinear functional analysis. {V}ol. 1}, volume~\bibinfo{volume}{48} of \textit{\bibinfo{series}{American Mathematical Society Colloquium Publications}}, \bibinfo{publisher}{American Mathematical Society, Providence, RI}, \bibinfo{year}{2000}. \URLprefix \url{https://doi.org/10.1090/coll/048}. \DOIprefix\doi{10.1090/coll/048}.
%Type = Article
\bibitem[{Galimberti et~al.(2022)Galimberti, Kratsios, and Livieri}]{galimberti2022designing}
\bibinfo{author}{L.~Galimberti}, \bibinfo{author}{A.~Kratsios}, \bibinfo{author}{G.~Livieri},
\newblock \bibinfo{title}{Designing universal causal deep learning models: The case of infinite-dimensional dynamical systems from stochastic analysis},
\newblock \bibinfo{journal}{arXiv preprint arXiv:2210.13300}  (\bibinfo{year}{2022}).
%Type = Article
\bibitem[{Lanthaler(2023)}]{lanthaler2023operator}
\bibinfo{author}{S.~Lanthaler},
\newblock \bibinfo{title}{Operator learning with pca-net: upper and lower complexity bounds},
\newblock \bibinfo{journal}{Journal of Machine Learning Research} \bibinfo{volume}{24} (\bibinfo{year}{2023}) \bibinfo{pages}{1--67}.
%Type = Article
\bibitem[{Riegler et~al.(2024)Riegler, B{\"u}hler, Pan, and B{\"o}lcskei}]{riegler2024generating}
\bibinfo{author}{E.~Riegler}, \bibinfo{author}{A.~B{\"u}hler}, \bibinfo{author}{Y.~Pan}, \bibinfo{author}{H.~B{\"o}lcskei},
\newblock \bibinfo{title}{Generating rectifiable measures through neural networks},
\newblock \bibinfo{journal}{arXiv preprint arXiv:2412.05109}  (\bibinfo{year}{2024}).
%Type = Article
\bibitem[{Murari et~al.(2025)Murari, Furuya, and Sch{\"o}nlieb}]{murari2025approximation}
\bibinfo{author}{D.~Murari}, \bibinfo{author}{T.~Furuya}, \bibinfo{author}{C.-B. Sch{\"o}nlieb},
\newblock \bibinfo{title}{Approximation theory for 1-lipschitz resnets},
\newblock \bibinfo{journal}{arXiv preprint arXiv:2505.12003}  (\bibinfo{year}{2025}).
%Type = Article
\bibitem[{Amit et~al.(2022)Amit, Epstein, Moran, and Meir}]{amit2022integral}
\bibinfo{author}{R.~Amit}, \bibinfo{author}{B.~Epstein}, \bibinfo{author}{S.~Moran}, \bibinfo{author}{R.~Meir},
\newblock \bibinfo{title}{Integral probability metrics pac-bayes bounds},
\newblock \bibinfo{journal}{Advances in Neural Information Processing Systems} \bibinfo{volume}{35} (\bibinfo{year}{2022}) \bibinfo{pages}{3123--3136}.
%Type = Article
\bibitem[{Hou et~al.(2023)Hou, Kassraie, Kratsios, Krause, and Rothfuss}]{hou2023instance}
\bibinfo{author}{S.~Hou}, \bibinfo{author}{P.~Kassraie}, \bibinfo{author}{A.~Kratsios}, \bibinfo{author}{A.~Krause}, \bibinfo{author}{J.~Rothfuss},
\newblock \bibinfo{title}{Instance-dependent generalization bounds via optimal transport},
\newblock \bibinfo{journal}{Journal of Machine Learning Research} \bibinfo{volume}{24} (\bibinfo{year}{2023}) \bibinfo{pages}{1--51}.
%Type = Article
\bibitem[{Kratsios et~al.(2024)Kratsios, Neuman, and Pammer}]{kratsios2024tighter}
\bibinfo{author}{A.~Kratsios}, \bibinfo{author}{A.~M. Neuman}, \bibinfo{author}{G.~Pammer},
\newblock \bibinfo{title}{Tighter generalization bounds on digital computers via discrete optimal transport},
\newblock \bibinfo{journal}{arXiv preprint arXiv:2402.05576}  (\bibinfo{year}{2024}).
%Type = Article
\bibitem[{Detering et~al.(2025)Detering, Galimberti, Kratsios, Livieri, and Neuman}]{detering2025learninggraphtransductivelearning}
\bibinfo{author}{N.~Detering}, \bibinfo{author}{L.~Galimberti}, \bibinfo{author}{A.~Kratsios}, \bibinfo{author}{G.~Livieri}, \bibinfo{author}{A.~M. Neuman},
\newblock \bibinfo{title}{Learning from one graph: Transductive learning guarantees via the geometry of small random worlds},
\newblock \bibinfo{journal}{arXiv preprint arXiv:2509.06894}  (\bibinfo{year}{2025}). \URLprefix \url{https://arxiv.org/abs/2509.06894}.
%Type = Article
\bibitem[{Lei(2020)}]{LeiConcentrationUnbounded2020Bernoulli}
\bibinfo{author}{J.~Lei},
\newblock \bibinfo{title}{Convergence and concentration of empirical measures under {W}asserstein distance in unbounded functional spaces},
\newblock \bibinfo{journal}{Bernoulli} \bibinfo{volume}{26} (\bibinfo{year}{2020}) \bibinfo{pages}{767--798}. \URLprefix \url{https://doi.org/10.3150/19-BEJ1151}. \DOIprefix\doi{10.3150/19-BEJ1151}.
%Type = Article
\bibitem[{Shen et~al.(2022)Shen, Yang, and Zhang}]{shen2022optimal}
\bibinfo{author}{Z.~Shen}, \bibinfo{author}{H.~Yang}, \bibinfo{author}{S.~Zhang},
\newblock \bibinfo{title}{Optimal approximation rate of relu networks in terms of width and depth},
\newblock \bibinfo{journal}{Journal de Math{\'e}matiques Pures et Appliqu{\'e}es} \bibinfo{volume}{157} (\bibinfo{year}{2022}) \bibinfo{pages}{101--135}.
%Type = Article
\bibitem[{Shi et~al.(2023)Shi, Xu, and Zhang}]{shi2023deep}
\bibinfo{author}{X.~Shi}, \bibinfo{author}{D.~Xu}, \bibinfo{author}{Z.~Zhang},
\newblock \bibinfo{title}{Deep learning algorithms for hedging with frictions},
\newblock \bibinfo{journal}{Digital Finance} \bibinfo{volume}{5} (\bibinfo{year}{2023}) \bibinfo{pages}{113--147}. \URLprefix \url{https://doi.org/10.1007/s42521-023-00075-z}. \DOIprefix\doi{10.1007/s42521-023-00075-z}.
%Type = Book
\bibitem[{Ryan(2002)}]{RyanTensor_2002}
\bibinfo{author}{R.~A. Ryan}, \bibinfo{title}{Introduction to tensor products of {B}anach spaces}, Springer Monographs in Mathematics, \bibinfo{publisher}{Springer-Verlag London, Ltd., London}, \bibinfo{year}{2002}. \URLprefix \url{https://doi.org/10.1007/978-1-4471-3903-4}. \DOIprefix\doi{10.1007/978-1-4471-3903-4}.
%Type = Book
\bibitem[{Kadison and Ringrose(1997)}]{kadison1986fundamentalsOlAlgs}
\bibinfo{author}{R.~V. Kadison}, \bibinfo{author}{J.~R. Ringrose}, \bibinfo{title}{Fundamentals of the theory of operator algebras. {V}ol. {I}}, volume~\bibinfo{volume}{15} of \textit{\bibinfo{series}{Graduate Studies in Mathematics}}, \bibinfo{publisher}{American Mathematical Society, Providence, RI}, \bibinfo{year}{1997}. \URLprefix \url{https://doi.org/10.1090/gsm/015}. \DOIprefix\doi{10.1090/gsm/015}, \bibinfo{note}{elementary theory, Reprint of the 1983 original}.

\end{thebibliography}
